\documentclass[11pt]{amsart}   	
\usepackage{geometry}             
\usepackage[dvipsnames]{xcolor}   		             		
\usepackage{graphicx}				
\usepackage{amssymb,mathrsfs}
\usepackage{amsthm,amsmath,stmaryrd}
\usepackage{tikz}
\usepackage{tikz-cd}
\usepackage{accents,upgreek,enumerate}
\usepackage[headings]{fullpage}
\usepackage{bm}
\usepackage{mathtools}
\usepackage[all]{xy}
\usepackage{caption}
\usepackage[euler-digits]{eulervm}
\usepackage[normalem]{ulem}

\usepackage{setspace}
\tikzset{
  commutative diagrams/.cd, 
  arrow style=tikz, 
  diagrams={>=stealth}
}

\newtheorem{innercustomthm}{Theorem}
\newenvironment{customthm}[1]
  {\renewcommand\theinnercustomthm{#1}\innercustomthm}
  {\endinnercustomthm}

  \newtheorem{innercustomconj}{Conjecture}
\newenvironment{customconj}[1]
  {\renewcommand\theinnercustomconj{#1}\innercustomconj}
  {\endinnercustomconj}
  
  \makeatletter
\def\@tocline#1#2#3#4#5#6#7{\relax
  \ifnum #1>\c@tocdepth 
  \else
    \par \addpenalty\@secpenalty\addvspace{#2}%
    \begingroup \hyphenpenalty\@M
    \@ifempty{#4}{%
      \@tempdima\csname r@tocindent\number#1\endcsname\relax
    }{%
      \@tempdima#4\relax
    }%
    \parindent\z@ \leftskip#3\relax \advance\leftskip\@tempdima\relax
    \rightskip\@pnumwidth plus4em \parfillskip-\@pnumwidth
    #5\leavevmode\hskip-\@tempdima
      \ifcase #1
       \or\or \hskip 1em \or \hskip 2em \else \hskip 3em \fi%
      #6\nobreak\relax
    \dotfill\hbox to\@pnumwidth{\@tocpagenum{#7}}\par
    \nobreak
    \endgroup
  \fi}
\makeatother

\usetikzlibrary{calc}
\usetikzlibrary{fadings}
\usetikzlibrary{decorations.pathmorphing}
\usetikzlibrary{decorations.pathreplacing}

\newcounter{marginnote}
\DeclareMathAlphabet{\mathpzc}{OT1}{pzc}{m}{it}

\usepackage[backref=page]{hyperref}
\hypersetup{
  colorlinks   = true,          
  urlcolor     = purple,          
  linkcolor    = blue,          
  citecolor   = violet             
}

\newtheorem{theorem}{Theorem}[subsection]
\newtheorem{corollary}[theorem]{Corollary}
\newtheorem{lemma}[theorem]{Lemma}
\newtheorem{proposition}[theorem]{Proposition}
\newtheorem{conjecture}[theorem]{Conjecture}
\newtheorem{quasi-theorem}[theorem]{Quasi-Theorem}

\theoremstyle{definition}
\newtheorem{definition}[theorem]{Definition}
\newtheorem{warning}[theorem]{$\lightning$ Warning}
\newtheorem{remark}[theorem]{Remark}

\newtheorem{construction}[theorem]{Construction}
\newtheorem{example}[theorem]{Example}
\newtheorem{blank remark}[theorem]{}

\newtheorem{not1}[theorem]{Notation}

\newcommand{\A}{{\mathbb{A}}}           
\newcommand{\CC} {{\mathbb C}}

\newcommand{\PP}{\mathbb{P}}         
\newcommand{\QQ} {{\mathbb Q}}		
\newcommand{\RR} {{\mathbb R}}		
\newcommand{\ZZ} {{\mathbb Z}}		

\def\setminus{\smallsetminus}

\newcommand{\Hom}{\operatorname{Hom}}

\newcommand{\cal}{\mathcal}

\def\cC{{\cal C}}
\def\cD{{\cal D}}

\def\cF{{\cal F}}

\def\cP{{\cal P}}

\def\cX{{\cal X}}
\def\cY{{\cal Y}}
\def\cZ{{\cal Z}}

\def\trop{\mathsf{trop}}

\newcommand{\Spec}{\operatorname{Spec}}

\makeatletter
\def\blfootnote{\xdef\@thefnmark{}\@footnotetext}
\makeatother

\title{Logarithmic enumerative geometry for curves and sheaves}

\date{}

\author{Davesh Maulik {\it \&} Dhruv Ranganathan}

\address{Davesh Maulik \\ Department of Mathematics\\
Massachusetts Institute of Technology, Cambridge, MA, USA}
\email{\href{mailto:maulik@mit.edu}{maulik@mit.edu}}

\address{Dhruv Ranganathan \\ Department of Pure Mathematics {\it \&} Mathematical Statistics\\
University of Cambridge, Cambridge, UK}
\email{\href{mailto:dr508@cam.ac.uk}{dr508@cam.ac.uk}}

\begin{document}

\maketitle

\begin{abstract}
We propose a logarithmic enhancement of the Gromov--Witten/Donaldson--Thomas correspondence, with descendants, and study its behavior under simple normal crossings degenerations. The formulation of the logarithmic correspondence requires a matching of tangency conditions with relative insertions. This is achieved via a version of the Nakajima basis for the cohomology of the Hilbert schemes of points on a logarithmic surface. 

Next, we establish a strong form of the degeneration formula in logarithmic DT theory -- the numerical DT invariants of the general fiber of a degeneration are determined by the numerical DT invariants attached to strata of the special fiber. The GW version of this result, which we prove in all target dimensions, strengthens currently known formulas. A key role is played by a certain exotic class of insertions, introduced here, that impose non-local incidence conditions coupled across multiple boundary strata of the target geometry.

Finally, we prove compatibility of the new logarithmic GW/DT correspondence with degenerations. In particular, the logarithmic conjecture for all strata of the special fiber of a degeneration implies the traditional GW/DT conjecture on the general fiber. Compatibility is a strong constraint, and can be used to calculate logarithmic DT invariants. Several examples are included to illustrate the nature and utility of the formula. 
\end{abstract}

\vspace{0.75in}
\setcounter{tocdepth}{1}
\tableofcontents
\newpage

\section*{Introduction}

\subsection{Enumerative geometry of pairs} Let $X$ be a smooth algebraic $3$-fold and $D\subset X$ a simple normal crossings divisor. We are interested in numerical invariants obtained from intersection theory on moduli spaces attached to $(X|D)$. In Gromov--Witten (GW) theory, the {``curve side''}, the moduli space is the space of logarithmic stable maps from nodal curves to $X$~\cite{AC11,Che10,GS13}. Donaldson--Thomas (DT) theory and the closely related Pandharipande--Thomas (PT) theory, are the {``sheaf side''} of enumerative geometry. In these contexts, the moduli spaces parameterize sheaves, or more generally, stable pairs in the derived category that are appropriately framed along $D$ ~\cite{MR20}.

In the absence of a divisor $D$, there is a network of conjectures -- what we loosely refer to as the {\it curve/sheaf correspondence} -- which gives a precise equivalence between the curve and sheaf sides of enumerative geometry: for a fixed $3$-fold $X$, for each curve class $\beta$, after appropriate normalizations, generating functions of GW, DT, and PT invariants are related by an explicit change of variables~\cite{MNOP06a,MNOP06b,PT09}. The conjecture has been proved in a number of interesting cases~\cite{BP08,MNOP06a,MNOP06b,MOOP,Ob21,PP12,PP14}, including striking recent progress by Pardon~\cite{Par23}. 

Our goal here is to provide a conjectural extension of the correspondence to the setting of simple normal crossings pairs, and then study how this conjecture behaves under simple normal crossings degenerations. One of the main results of the paper, independent of the conjectural framework, is a complete picture of the behavior of DT and PT invariants under degeneration, and a strengthening of the existing GW results. We find that the conjecture is compatible with degeneration in a strong sense. 

There are several motivations for our study. Our work generalizes results for pairs $(X|D)$ when $D$ is smooth and for {\it double point} degenerations~\cite{MNOP06b,LiWu15}; already in this simpler case, the compatibility of the correspondence with degeneration is an important constraint. It provides inductive strategies to prove the correspondence, and is used in a crucial fashion in~\cite{BP08,PP14,PP12}. We expect that compatibility with simple normal crossings degenerations strongly constrains the full descendant curve/sheaf conjecture, which remains largely open.
Analogous constraints in Gromov--Witten theory have proved to be very useful computational tools, see for instance~\cite{Bou17,Par17}. 

The logarithmic correspondence also appears to be interesting in its own right, for example via its interactions with work of Bousseau and others related to the refined topological string~\cite{Bou17,Bou19,BFGW21}, with the Gross--Siebert mirror construction~\cite{GS19}, and with the double ramification cycle on the moduli space of curves~\cite{HMPPS,KHSUK,RUK22}. The logarithmic context also has a rich potential link, in its $K$-theoretic incarnation, with ideas in geometric representation theory~\cite{Oko17}. 

\subsection{The logarithmic curve/sheaf correspondence} In order to state the correspondence precisely, we need to organize DT/PT and GW invariants. Let $X$ be a smooth and projective complex threefold and let $D\subset X$ be a simple normal crossings divisor with divisor components $D_1,\ldots,D_r$. We often use $(X|D)$ to denote the pair. Fix a curve class $\beta$ on $X$. It determines a sequence of intersection numbers $\mathbf n = (\beta\cdot D_i)_i$.

For each choice of holomorphic Euler characteristic $\chi$, we have a logarithmic DT space $\mathsf{DT}_{\beta,\chi}(X)$, constructed in~\cite{MR20}, which
parametrizes subschemes of expansions of $X$ along $D$; we review this in Section~\ref{sec: review-of-moduli}.  It is a proper Deligne--Mumford stack, equipped with a virtual fundamental class.


Similarly, for each divisor $D_i$ we can consider the Hilbert scheme of $n_i$ points on expansions of $D_i$ along its induced boundary. These are smooth Deligne--Mumford stacks of dimension $2n_i$. Each is equipped with a natural simple normal crossings boundary divisor, given by the locus where the target expands nontrivially. We let $\mathsf{Hilb}^{\bf n}(D)$ be the {\it product} of all of these Hilbert schemes $\mathsf{Hilb}^{n_i}(D_i)$, taken over all divisors. There is a natural evaluation morphism:
\[
\mathsf{DT}_{\beta,\chi}(X|D)\to \mathsf{Hilb}^{\bf n}(D).
\]
This induces {\it boundary insertions} in DT/PT theory by pulling back cohomology classes. 

For the applications we have in mind, we find the appropriate level of generality is to work with insertions in the logarithmic cohomology of this target -- the direct limit of cohomology rings over all logarithmic blowups of $\mathsf{Hilb}^{\bf n}(D)$. 
In Section~\ref{sec: Hilb-cohomology}, we prove that the logarithmic cohomology $\mathsf{Hilb}^{\bf n}(D)$ has a natural Nakajima-type basis.  Each basis element $\bm{\mu}(\delta)$ 
is given by a {\it cohomology decorated vector of partitions}.  This is a vector $\bm \mu = (\mu_1, \dots, \mu_r)$ where each $\mu_i$ a partition of $n_i$, together with a cohomology class $\delta$ in the logarithmic cohomology of a certain product of strata of $D$. 
See Section~\ref{sec: exotic-insertions} below for further discussion of these insertions. 

In addition, one can use operators built out of the universal sheaf to impose primary and descendant {\it bulk insertions} as in the usual theory~\cite{MNOP06b}. If we fix a boundary class $\bm{\mu}(\delta)$, and bulk descendants $\tau_{a_1}(\zeta_1),\ldots,\tau_{a_k}(\zeta_k)$, then the general theory of~\cite{MR20} gives us a DT or PT invariant. By fixing $\beta$ and the insertions, and then summing over the possible values of $\chi$ we are led to a DT partition function, denoted:
\[
\mathsf{Z}_{\mathsf{DT}}\left(X,D;q | \prod_{i=1}^{k} \tau_{a_{i}}(\zeta_i)|\bm{\mu}(\delta)\right)_{\beta}
= \sum_{\chi} \langle \prod_{i=1}^{k} \tau_{a_{i}}(\zeta_i)| \bm{\mu}(\delta) \rangle^{\mathsf{DT}}_{\beta,\chi} \cdot {q^{\chi}}.
\]
The {\it normalized} DT partition function is obtained by dividing the partition function by the $\beta = 0$ contribution. The PT case is analogous, and the PT and normalized DT series are conjectured to be equal~\cite{MR20}. 

On the Gromov-Witten side, we fix a genus $g$ and use the 
$\bm \mu$ to determine
{\it tangency} data for curves with respect to $D$.  We have a logarithmic GW space $\mathsf{GW}_{\beta,g}(X)_{k, \bm \mu}$, which parametrizes 
 stable maps of curves to expansions of $X$, with appropriate tangency, and with $k$ internal marked points; see Sections~\ref{sec: expanded-maps} and \ref{subsec:comparisons} for background here and comparisons to other approaches.
 
There is an associated evaluation space $\mathsf{Ev}^{\bm\mu}(D)$, which parametrizes ordered points in expansions of $D_i$, as well as an evaluation map
\[
\mathsf{GW}_{\beta,g}(X)_{k,\bm\mu}\to \mathsf{Ev}^{\bm\mu}(D)
\]
which sends a logarithmic stable map to the locations of the points of tangency.
The class $\delta$ defines a logarithmic cohomology class on $\mathsf{Ev}^{\bm\mu}(D)$ which imposes boundary conditions for these points.
Disconnected logarithmic GW partition functions may now be defined, with the usual normalization that the map is non-constant on every connected component of the domain curve~\cite{MNOP06a}. 
Formally, we have 
\[
\mathsf{Z}_{\mathsf{GW}}\left(X,D; u | \prod_{i=1}^{k} \tau_{a_{i}}(\zeta_i)|\bm{\mu}(\delta)\right)_{\beta} = \sum_{g} \langle \prod_{i=1}^{k} \tau_{a_{i}}(\gamma_i) | \bm{\mu}(\delta) \rangle^{\bullet, \mathsf{GW}}_{g,\beta} u^{2g-2}.
\]

We now state the simplest version of our main conjecture -- the case of primary insertions.

\begin{customconj}{A}
Fix boundary conditions $\bm{\mu}(\delta)$ and primary bulk $\tau_{0}(\zeta_1),\ldots,\tau_{0}(\zeta_k)$. Let $d_\beta$ be the anticanonical degree of $\beta$. Then under the identification of variables $-q = e^{iu}$, we have the equality
\[
(-q)^{-d_{\beta}/2} \mathsf{Z}_{\mathsf{PT}}\left(X,D;q | \prod_{i=1}^{k} \tau_{0}(\zeta_i)|\bm{\mu}(\delta)\right)_{\beta}=
(-iu)^{d_{\beta} + \sum \ell(\mu_j) - |\mu_j|} \mathsf{Z}_{\mathsf{GW}}\left(X,D; u|\prod_{i=1}^{k} \tau_{0}(\zeta_i)|\bm{\mu}(\delta)\right)_{\beta}.
\]
\end{customconj}

In Section~\ref{sec:curve/sheaf-conjectures} we present complete descendant conjectures extending the conjecture above and~\cite{PP12}. 

\subsection{The degeneration package} 
The second goal of this paper is to establish techniques for studying traditional curve-counting invariants in terms of logarithmic invariants, after passing to a simple normal crossings degeneration.  

The complete degeneration formalism is given in Section~\ref{sec: splitting-compatibility} with explicit cycle-theoretic and numerical formulas, after appropriate notation has been established. In this introduction we give a more qualitative statement. 

\begin{customthm}{B}
Let $\mathcal Y\to B$ be a simple normal crossings degeneration and assume that the fiber over $0$ in $B$ is the unique singular fiber. Let $Y_\eta$ denote the general fiber and $Y_0$ this singular fiber. 
\begin{enumerate}[(i)]
\item The virtual class of the space $\mathsf{DT}_{\beta,\chi}(Y_0)$ decomposes as a sum
\[
[\mathsf{DT}_{\beta,\chi}(Y_0)]^{\mathsf{vir}} = \sum_{\gamma} [\mathsf{DT}_{\gamma}(Y_0)]^{\mathsf{vir}},
\]
where $\gamma$ is an appropriately enhanced graph, embedded in the dual complex of $Y_0$, with vertices and edges decorated by curve class, Euler characteristic, and multiplicity data. These graphs are called Hilbert $1$-complexes.

\item Each vertex $v$ of the Hilbert $1$-complex $\gamma$ determines a $3$-fold pair $X_v$, which is birational to a projective bundle over a stratum of $Y_0$. The choice of vertex $v$ uniquely determines curve class and Euler characteristic, together denoted $\gamma(v)$ for a DT moduli problem associated to $X_v$. 

\item The logarithmic DT invariants associated to the general fiber $Y_\eta$, with bulk insertions from $\mathcal Y$, are algorithmically determined by the DT invariants associated to the pairs $X_v$ with numerical data $\gamma(v)$, where $\gamma$ ranges over all Hilbert $1$-complexes $\gamma$, and $v$ ranges over all its vertices. 
\end{enumerate}

Parallel degeneration formulas hold for both PT and GW invariants.

\end{customthm}

Explicit formulas are given in Section~\ref{sec: splitting-compatibility}. In Section~\ref{sec: examples} we explain in examples how the compatibility can be used as a new constraint to study DT invariants and the curve/sheaf conjectures. We also wish to stress that in the final point of the theorem statement, the {\it numerical} invariants of the general fiber, obtained after integration, can be obtained by the numerical invariants of the individual components of the special fiber. In this aspect, 
the formula is stronger and more practical than the statements that can be easily deduced from~\cite{R19} in GW theory.

\subsection{Degeneration and the curve/sheaf correspondence} The logarithmic curve/sheaf correspondence is formulated to be an extension of the traditional conjectures~\cite{MNOP06a,MNOP06b}. However, several aspects of the proposed correspondence are new.  For instance, in the logarithmic context, the foundations of GW and DT/PT theory are rather different from each other. There are versions of logarithmic GW theory using ``collapsed'' targets ~\cite{AC11,Che10,GS13} and expanded targets~\cite{R19,Li01}; however, in DT theory we only have access to an expanded theory. In line with this, the logarithmic Nakajima basis, required for matching boundary conditions, takes a more subtle form here than in the classical case. 

As a result we view the following result as evidence that the logarithmic curve/sheaf correspondence is robust. 

\begin{customthm}{C}
The descendant logarithmic curve/sheaf correspondence is compatible with degeneration. 

In more detail: let $\mathcal Y\to B$ be a simple normal crossings degeneration and assume that the fiber over $0$ in $B$ is the unique singular fiber. If the logarithmic curve/sheaf correspondence holds for all strata of the singular fiber $Y_0$, then it holds for $\mathcal Y_\eta$, with insertions in the cohomology of $\mathcal Y$. 
\end{customthm}

It is our expectation that this compatibility places strong constraints on the GW/DT correspondence. In simple terms, a single variety can admit many several different degenerations and the formulas coming out of these degenerations tend to be very different. A manifestation of this can already be seen when using tropical geometry to count rational members in a pencil of cubics in $\mathbb P^2$, see for instance~\cite[Section~4]{vGOR}.

In fact, we expect that when $(X|D)$ is the pair consisting of a toric threefold and the toric boundary, compatibility should be sufficient to establish the full primary correspondence for this geometry. The expectation is based on parallel work in GW theory of Bousseau and Parker~\cite{Bou17,Par17}, and we explore these ideas in a detailed example in Section~\ref{sec: examples}. Examples that illustrate this may be found in Section~\ref{sec: examples}.

\subsection{Exotic invariants and Nakajima basis}

For pairs $(X|D)$ when $D$ is smooth, our work reduces to the constructions of relative GW/DT theory by Li and Li--Wu~\cite{Li01,LiWu15} and the degeneration formula when $\mathcal Y\to B$ is a double point degeneration~\cite{Li02,LiWu15,MNOP06b,MPT10}.  In the simple normal crossings setting, both the correspondence and the degeneration formula become more delicate.  We highlight two new features.

\subsubsection{Exotic insertions}\label{sec: exotic-insertions} 

For double point degenerations, the numerical degeneration formula involves boundary conditions associated to the divisor $D$ along which components are glued.
In order to generalize this, an important new feature is the introduction of certain {\it exotic} insertions. 
On the GW side, we have evaluation maps
\[
\mathsf{GW}_{\beta,g}(X)_{k,\bm\mu}\to \mathsf{Ev}^{\bm\mu}(X)
\]
along which we pull back logarithmic cohomology classes to impose boundary conditions.  
However, this allows us to impose conditions from the cohomology of {\it blowups} of $\mathsf{Ev}^{\bm\mu}(X)$ along strata.  While the space on the right is a product of strata of $D$, these {\it exotic insertions} include classes on modifications that do not respect the product structure, so we have boundary insertions that are ``coupled'' across different boundary components.
Parallel insertions arise from the evaluation map in DT theory and
the logarithmic cohomology of $\mathsf{Hilb}^{\bf n}(D)$.  Again the evaluation space is a product, but its logarithmic cohomology does not decompose along with the product structure. 
These extra classes are necessary to have a K\"unneth decomposition in our setting and obtain numerical splitting formulas.

\subsubsection{Nakajima basis} The logarithmic curve/sheaf correspondence  requires a logarithmic version of the {\it Nakajima basis}. Recall that if $S$ is a smooth projective surface, the rational cohomology of $\mathsf{Hilb}^n(S)$ has a distinguished basis given by unordered partitions of $n$, with parts weighted by $H^\star(S)$, see~\cite{dCM02,GS93,Groj96,Nak97}. Under the conjecture, these partitions correspond to the different tangency conditions in relative GW theory.

In the logarithmic theory, each component $D_i\subset D$ has a boundary of its own. The Hilbert scheme of $D_i$, relative to the simple normal crossings divisor $(D_i\cap D_j)_j$, has a more complicated cohomology. In Section~\ref{sec: Hilb-cohomology} we give a complete calculation of the cohomology of this space in terms of cohomology rings of logarithmic modifications of symmetric products of surfaces. The partition structure that arises from this result is what allows us to match with GW theory.   This matching with GW theory seems to rely on the expanded formulation of GW theory, and it is unclear how to do this in other formalisms of logarithmic GW invariants.

The construction of the Nakajima basis is straightforward, but only after working relative to an appropriate stack of expansions; the key geometric idea is to separate two geometric phenomena -- points moving into each other and points moving into the boundary.

This construction arises in several places in this paper. First, as already noted, it is needed even just to formulate the correspondence. Second, 
the expression of certain modified diagonal classes in the K\"unneth decomposition of a product of such Hilbert schemes, is necessary to make the degeneration formula useable. Finally, the compatibility of the conjectures with the degeneration formula requires that the basis is a compatible with diagonal self-correspondences. 

\subsection{Relationship with other approaches} 

On the Gromov-Witten side, we discuss briefly some of the other approaches to studying GW invariants under simple normal crossings degenerations. 

\subsubsection{Punctures} Abramovich, Chen, Gross, and Siebert have developed a theory of {\it punctured} GW invariants that performs a similar function to our degeneration package~\cite{ACGS15,ACGS17}. The theory of punctured invariants is a priori larger than ordinary logarithmic GW invariants, but the results of~\cite{ACGS17} establish a different degeneration package using these invariants. Our results expands the class of invariants in a different way, by allowing the exotic insertions. It is not clear how to compare these. Indeed the evaluation spaces in the punctured theory are Artin stacks with subtle properties~\cite{ACGS17}.
While the punctured formalism is well-adapted to a different set of applications, there is currently no punctured DT theory, so it is not suited to our current study. 

 Additionally, given a simple normal crossings degeneration, it is not obvious to us that the GW invariants of the general fiber are determined by the punctured log GW invariants of the components of the special fiber. We do note, however, that there has been important recent progress in developing computational tools in punctured GW theory~\cite{Gross23,Wu21}. We hope these methods can be fruitfully combined with our work here. A theory of punctured DT invariants also appears to be an interesting potential direction. 

\subsubsection{Explosions} Our approach is also close in spirit to Parker's degeneration formulas in {\it exploded} GW theory~\cite{Par17a}. Parker's formulas take value in {\it refined de Rham cohomology}, see~\cite{Par17b}, which is analogous to logarithmic cohomology. As Parker notes, refined cohomology does not support a K\"unneth theorem so one cannot immediately extract complete numerical consequences from the theorem~\cite{Par19}. Nevertheless, Parker has developed several methods to do explicit computations in his theory~\cite{Par14p2,Par15,Par17} and his calculations have played an important role in the development of this paper.

\subsection{Outline} The paper is broken into two parts. In Part~\ref{part: 1} we deal exclusively with the enumerative geometry of pairs, and work towards a precise statement of the logarithmic/curve sheaf correspondence. Section~\ref{sec: tropical-toolbox} contains a review of the basics of cone complexes and Artin fans, their subdivisions, and some basic constructions and properties of these that are used throughout. In Section~\ref{sec: review-of-moduli} we review the foundations of logarithmic DT theory from~\cite{MR20}, and explain a slightly new point of view on logarithmic GW theory that is better adapted to the comparison. In Section~\ref{sec: points} we explain the structure of the evaluation spaces in these two theories, i.e. the logarithmic Hilbert scheme of points and the spaces of stable maps from points. We also construct the evaluation maps from the sheaf/curve spaces to the spaces of points. In Section~\ref{sec: Hilb-cohomology} we prove that the cohomology of the evaluation spaces on the two sides match, relying on our setup in the previous section. With this setup complete, we state the logarithmic curve/sheaf conjectures in Section~\ref{sec:curve/sheaf-conjectures}. The conjectures are stated for both DT and PT theory, including descendants in the PT setting. 

In Part~\ref{part: 2} we introduce the logarithmic degeneration package. In Section~\ref{sec: DT-over-base} we construct the DT (and PT) space for a degeneration, making the necessary adaptations from our previous work~\cite{MR20} which deals with the pairs case. The degeneration package contains roughly three steps, and the first two of these are essentially formal and are handled in Sections~\ref{sec: DT-over-base} and~\ref{sec: decomposition-step}. In Section~\ref{sec: splitting-compatibility} we state the formulas in the degeneration package in DT, PT, and GW theory. We include numerical degeneration formulas in each case. The proofs are delayed until after Section~\ref{sec: examples}, which contains several examples that are meant to illustrate the features of the degeneration package, including how it can be used to constraint invariants in practice. The rest of the paper is then dedicated to proving the main formulas. We largely focus on the DT case, but explain/review the GW cases as we go. In Section~\ref{sec: combinatorial-degeneration} we prove a purely combinatorial analogue of the full degeneration package. This translates to the formula on the stack of expansions. In Section~\ref{sec: moving-gluing} we prove a coarse version of the gluing formula. This essentially brings the DT side of the picture in line with the GW picture completed in~\cite{R19}, although the proofs proceed quite differently due to the different nature of the construction of DT  theory. The final formula is deduced in Section~\ref{sec: main-splitting-formula} by using a logarithmic intersection theory argument, and here we present the full GW picture as well. 

\subsection*{Conventions and notation} When logarithmic structures are mentioned, they will always be fine and usually saturated, unless explicitly stated to the contrary. However, consistent with Olsson's point of view~\cite{Ols03} and our own previous work~\cite{MR20}, we typically make arguments in this paper by working relative to an Artin fan, and we rarely manipulate sheaves of monoids directly. In order to have compact notation, when working with simple normal crossings pairs $(X,D)$, we will typically use $X$ for the associated logarithmic scheme, thereby suppressing $D$ from the notation. The logarithmic sheaf and mapping spaces will then be denoted $\mathsf{DT}(X)$ and $\mathsf{GW}(X)$. The main exception to this notation will be in circumstances where we feel the presence of $D$ should be emphasized to avoid confusion. 

\subsection*{Acknowledgements} We have benefited from countless conversations with friends and colleagues over the years on the degeneration formulas, GW/DT, and logarithmic intersection theory, and other ideas related to this project. In particular, we're grateful to L. Battistella, F. Carocci, M. Gross, J. Guzman, P. Kennedy-Hunt, D. Holmes, S. Molcho, N. Nabijou, A. Okounkov, R. Pandharipande, O. Randal-Williams, Q. Shafi, B. Siebert, R. Thomas, A. Urundolil Kumaran, and J. Wise. D.R. presented the results here at online seminars at ETH and Stanford at the height of the pandemic, and these significantly influenced the presentation in this paper. We thank R. Pandharipande and R. Vakil, respectively, for creating these opportunities. Finally, we thank an anonymous referee for feedback on the first version. 

\noindent
D.R. is supported by EPSRC New Investigator Award EP/V051830/1.
\newpage

{\Large \part{The logarithmic curve/sheaf conjectures}\label{part: 1}}

\section{A tropical toolbox for logarithmic geometry}\label{sec: tropical-toolbox}

We introduce a few notions from the tropical geometry of logarithmic schemes. The background section~\cite[Section~1]{MR20} is a well-adapted reference. Detailed introductions can be found in~\cite{ACMUW,CCUW}. 

\subsection{Cone complexes} We use the term {\it cone} to mean an abstract rational polyhedral cone with integral structure. This is a pair $(\sigma,M)$ where $\sigma$ is a topological space and $M$ is a finitely generated group of continuous real valued functions on $\sigma$, with the property that the evaluation map:
\[
\sigma\to\Hom(M,\RR)
\]
is a homeomorphism from $\sigma$ to a strongly convex cone in the vector space $\Hom(M,\RR)$. This map also defines the subset of {\it integral points} of $\sigma$. It is customary to suppress $M$ from the notation. The elements of $M$ are the {\it linear functions} on $\sigma$. The {\it dual monoid} of a cone $\sigma$ is the monoid $S_\sigma$ of non-negative linear functions. 

In practice, all such objects are obtained from cones inside vector spaces with lattices, in the sense of toric geometry, but without remembering the ambient vector space explicitly. 

A {\it morphism} of cones $\sigma'\to\sigma$ is a continuous map that carries $M$ to $M'$ under pullback. 

A {\it face} of $\sigma$ is the vanishing set of an element of $S_\sigma$. A face of a cone is itself a cone, with the map on functions given by restriction. The set of faces of a cone is naturally ordered by inclusion, and forms a partially ordered set. More generally, a morphism $\tau\to \sigma$ is a {\it face morphism} if it is an isomorphism onto a face. 

A cone has a well-defined dimension, given by the length of the longest chain of proper face inclusions, minus one. A $1$-dimensional cone is called a ray and has a well-defined first nonzero lattice point, called its {\it primitive generator}. A cone $\sigma$ of dimension $k$ is said to be {\it smooth} if it has exactly $k$ rays and the primitive generators of the rays span the lattice points of $\sigma$.

We obtain more general objects by gluing cones together:

\begin{definition}[Cone spaces, smoothness, morphisms, and face maps]
A {\it cone space} is a collection $\mathcal C = \{\sigma_\alpha\}$ of cones with a collection of face morphisms $\cF$ that satisfies the following three properties: (i) the collection is closed under composition, (ii) the collection contains all identity maps of the form $\mathsf{Id}_{\sigma_{\alpha}}$, and (iii) every face of a cone in $\cC$ is the image of exactly one arrow in $\cF$. A {\it cone complex} is a cone space where there is at most one morphism between any two cones in the $\cC$.

A cone space is {\it smooth} if all its constituent cones are smooth. 

A {\it morphism} $f\colon \Sigma_1\to\Sigma_2$ of cone spaces is a choice, for each $\sigma_{\alpha_1}$ in $\cC_1$, of a cone $\sigma_{\alpha_2}$ in $\cC_2$, and a map
\[
\sigma_{\alpha_1}\to\sigma_{\alpha_2}
\]
that does not factor through a proper face.  Moreover, for any face map $\sigma_{\alpha_1} \to \sigma_{\beta_1}$ in $\cF_1$, there exists a corresponding face map $\sigma_{\alpha_2}\to \sigma_{\beta_2}$ making the natural square commute. 
\end{definition}

The traditional definition of cone complex -- a topological space presented as the colimit of a partially ordered set of face morphisms -- is equivalent to the one above. A cone space additionally allows, for example, two isomorphic faces of a cone to be glued to each other. We allow cone spaces in this text for consistency with earlier work on degenerations of Hilbert schemes~\cite{LiWu15,MR20}, however all the cone spaces appearing here can be replaced by cone complexes after subdivision, so they can be avoided if one wishes so.

There is also a more technical notion of a {\it cone stack}, which is a fibered category over the category of rational polyhedral cones with face morphisms, satisfying natural properties~\cite{CCUW}. We will not need this, except in passing.

\subsection{Artin fans} There is a sequence of parallel notions in the theory of algebraic stacks. The cone spaces are easier to manipulate, while the stacks are more readily connected to geometry. As we point out below, the category of cone spaces can be faithfully embedded into a category of algebraic stacks, so we can play these off each other. 

The equivalence starts with the notion of an {\it Artin cone}, which is analogous to a cone, and proceeds via gluing. 

Let $\sigma$ be a cone and $S_\sigma$ its dual monoid. The {\it Artin cone} associated to $\sigma$ is
\[
\mathsf A_\sigma := \Spec \CC[S_\sigma]/\Spec \CC[M].
\]
An Artin cone is naturally a {\it logarithmic algebraic stack} in the topology of smooth morphisms.

By the basic toric geometry dictionary, a morphism $\sigma'\to\sigma$ gives rise to maps
\[
\mathsf A_{\sigma'}\to\mathsf A_\sigma.
\]
It is straightforward to check that face maps are open immersions, and strict with the natural logarithmic structures. 

As a consequence of this construction, a cone space gives rise to a natural diagram of logarithmic algebraic stacks. We record the following:

\begin{definition}
An {\it Artin fan} is a logarithmic algebraic stack $\mathsf A$ that has a strict \'etale cover by a disjoint union of Artin cones. 
\end{definition}

Cone complexes and cone spaces give rise to Artin fans -- given a cone space $\Sigma$, we get a natural diagram of Artin cones with strict open immersions. The colimit in the category of logarithmic algebraic stacks is an Artin fan $\mathsf A_\Sigma$. 

The category of cone spaces with its natural class of morphisms embeds fully faithfully into the category of Artin fans above. If we enlarge the domain to include cone stacks, the $2$-categories of cone stacks and of Artin fans, become equivalent. 

\subsection{Logarithmic structures via Artin fans} Given a scheme $X$, an Artin fan $\mathsf A$, and a morphism
\[
X\to \mathsf A
\]
we can pull back the logarithmic structure from $\mathsf A$ to $X$. 

Olsson constructed an Artin fan $\mathsf{LOG}$ such that the moduli space of logarithmic structures on a scheme $X$ is isomorphic to the mapping stack from $X$ to $\mathsf{LOG}$. 

As a consequence, {\it any} logarithmic structure on a scheme $X$ comes from a map to some Artin fan. Building on this, Abramovich and Wise introduced the {\it Artin fan of a logarithmic scheme $X$}, which can be seen as the ``best possible'' factorization of the map
\[
X\to \mathsf{LOG}
\]
through an Artin fan. Precisely, under mild hypotheses, there is an initial factorization $$X\to \mathsf A_X\to \mathsf{LOG}$$ such that $\mathsf A_X\to \mathsf{LOG}$ is \'etale and representable. 

We contrast the notion of {\it an} {Artin fan}, which is a particular type of logarithmic algebraic stack, with {\it the Artin fan of a logarithmic scheme $X$}. While we use Artin fans throughout this text, we will make very limited use of this latter notion. In fact, the only situation we will make use of this is when $(X|D)$ is a simple normal crossings pair, in which case the Artin fan $\mathsf A_X$ is the one associated to the cone over the dual complex of $D$. 

All the logarithmic structures in this text will arise from {\it specific} maps to Artin fans. We will always be in the following situation for moduli spaces: we will have a moduli space $\mathsf M$ an Artin fan $\mathsf A$, and a map $\mathsf M\to\mathsf A$ that defines the logarithmic structure. 

\subsection{Logarithmic modifications and tropical models} Given an Artin fan $\mathsf A$, a {\it logarithmic modification} is a proper and birational morphism
\[
\mathsf A'\to \mathsf A.
\]
More generally, given a logarithmic scheme $X$, a {\it logarithmic modification} is the pullback of a logarithmic modification $\mathsf A'\to \mathsf A$ under a strict map
\[
X\to\mathsf A. 
\]
In fact, Abramovich and Wise show that the Artin fan $\mathsf A$ can always be chosen to be the Artin fan $\mathsf A_X$, see~\cite{AW}.

\begin{warning}
Since the pullback of a birational map is not always birational, a logarithmic modification of a logarithmic scheme $X$ need not be a modification in the traditional sense. Logarithmic modifications are proper and birational when $X\to \mathsf A$ is a flat map, but this is rarely the case for moduli spaces of maps or sheaves.
\end{warning}

Artin fans have smooth covers by toric varieties. In this cover, logarithmic modification corresponds to performing a torus equivariant modification. As a result, there is a natural bijection between {\it complete subdivisions} of a cone space $\Sigma$ and {\it logarithmic modifications} of $\mathsf A_\Sigma$. 

\noindent
{\bf Terminology} (Complete and incomplete subdivisions). We recall that a {\it possibly incomplete} subdivision $\Sigma'\to\Sigma$ is one that factorizes as a face inclusion followed by a complete subdivision. Most subdivisions in this paper will be complete, so defaulting to the standard, and in contrast with~\cite{MR20}, we assume subdivisions are complete and will explicitly use {\it incomplete subdivision} when this is not the case. 

The moduli spaces that are relevant to this paper are sufficiently general that it is impractical to compute their Artin fan. However, it is also unnecessary in our applications. We will always construct these spaces with specific strict maps to Artin fans. 

\noindent
{\bf Terminology} (Tropical models). Let $X\to \mathsf A$ be a strict map of logarithmic schemes. A logarithmic modification $X'\to X$ is called a {\it tropical model} if it is pulled back from a modification $\mathsf A'\to\mathsf A$. 

\subsection{The system of subdivisions} We often consider the collection of subdivisions of a cone space $\Sigma$ together. If we fix a cone space $\Sigma$, the set of subdivisions of $\Sigma$ is a category, in fact, a directed partially ordered set. Arrows are given by commuting triangles. Given subdivisions $\Sigma'$ and $\Sigma''$ in this category, there exists {\it at most one arrow} between $\Sigma'$ and $\Sigma''$. If an arrow exists, we say that the domain {\it dominates} the codomain. If $\Sigma$ is smooth, then the set of smooth cone spaces forms a subsystem. 

Given $\Sigma'$ and $\Sigma''$ subdivisions of $\Sigma$, there does always exist an element of this category, $\Sigma'''$ that maps to both $\Sigma'$ and $\Sigma''$. We call this category the {\it inverse system of subdivisions}. 

The category of cone spaces admits {\it fiber products} and subdivisions are stable under base extension, see~\cite{Mol16}. Suppose we are given a diagram
\[
\begin{tikzcd}
&\Sigma_1\arrow{d}\\
\Sigma_2'\arrow{r}&\Sigma_2,
\end{tikzcd}
\]
where $\Sigma'_2\to\Sigma_2$ is a subdivision. Given {\it any} subdivision $\Sigma_1'\to\Sigma_1$, it makes sense to ask {\it does $\Sigma_1'$ dominate the fiber product $\Sigma_2\times_{\Sigma_1}\Sigma_1'$ in the inverse system of subdivisions?} If it does, then such an arrow is unique, and composing this map with the map from the fiber product to $\Sigma_1'$, we say that there is {\it an induced map} 
$$
\Sigma_1'\to\Sigma_2'.
$$

Given a map of cone complexes
\[
\varphi\colon \Sigma_1\to\Sigma_2,
\]
and subdivisions $\Sigma'_1\to \Sigma_1$ and $\Sigma'_2\to \Sigma_2$, if there is an induced map $\varphi'\colon\Sigma_1'\to\Sigma_2'$, we refer to this as a {\it subdivision of $\varphi$}.

If $X\to Y$ is a morphism of logarithmic schemes, and $X'$ and $Y'$ are logarithmic modifications, we can then ask whether there is an induced map $X'\to Y'$. If there is, we will refer to it as the {\it induced map}.

Let us return to tropical models; suppose $X\to \mathsf A$ is strict. If $\mathsf A$ is not {\it the} Artin fan of $X$, then there may exist, and typically do exist, logarithmic modifications that are not tropical models. However, the system of tropical models (for a given choice of strict map to an Artin fan) is cofinal in the system of modifications. 

\subsection{Star of a cone}\label{sec:star-of-cone}

Let $\Sigma$ be a cone space and $\sigma$ a cone in it. The \textit{overstar} of $\sigma$ is the complement of cones in $\Sigma$ that do not contain $\sigma$ as a face\footnote{Note that this is \textit{not} the union of cones that contain $\sigma$!}; this is not a cone space, but there is a natural associated topological space, obtained as a colimit of conical interiors. 

The overstar has an action of the additive monoid $\sigma$. Indeed, every cone that contains $\sigma$ as a face inherits a monoid action of $\sigma$, which restricts to an action on the overstar. 

\begin{definition}
The \textit{star} of $\sigma$ has support equal to the quotient of the overstar $\Sigma^\sigma_\circ$ by this $\sigma$-action:
\[
|\Sigma^\sigma_\circ|\to |\Sigma(\sigma)|.
\]
If $\tau$ is a cone of $\Sigma$ that contains $\sigma$ as a face, let $\tau^\sigma_\circ$ denote the cell given by the complement of the cones of $\tau$ that do not contain $\sigma$. The image of $\tau^\sigma_\circ$  is a cone in the quotient. The cones of this form give a fan structure on the star, and we denote it $\Sigma(\sigma)$. 
\end{definition}

The star construction satisfies basic compatibilities.  If $\sigma$ and $\tau$ are cones of $\Sigma$ such that $\sigma$ is a face of $\tau$, and $\overline{\tau}$ is the corresponding cone in
$\Sigma(\sigma)$, then $\Sigma(\tau) = (\Sigma(\sigma))(\overline{\tau})$.  Similarly, given a strict map of cone complexes $\Sigma \rightarrow \Sigma'$ that sends a cone $\sigma$ onto a cone $\sigma'$, there is an induced map of stars $\Sigma(\sigma) \rightarrow \Sigma'(\sigma')$.

\subsection{Combinatorial flatness} We frequently use a form of flatness for maps between cone spaces.

\begin{definition}
A map $\Sigma_1\to\Sigma_2$ of cone spaces is {\it combinatorially flat} if every cone of the source maps surjectively onto a cone of the target. It has {\it reduced fibers} if, in addition, for every cone of the source, the image of its lattice generates the lattice of the target cone. 
\end{definition}

When the map of cone spaces arises from an equivariant map of toric varieties, the condition of combinatorial flatness is equivalent to the map having equidimensional fibers. The combinatorial condition of having reduced fibers ensures that these fibers are reduced. We will typically work in contexts where $\Sigma_2$ is a {\it smooth} cone space. In this case, combinatorial flatness is equivalent to flatness of the map on Artin fans. 

Any map of cone spaces can be made combinatorially flat by subdivision of source and target. 

\begin{proposition}\label{prop: strong-ss-reduction}
Let $\Sigma_1\to\Sigma_2$ be a map of cone complexes of finite type. There exists subdivision of this map $\Sigma_1'\to\Sigma_2'$ that is combinatorially flat. By replacing the integral lattice of $\Sigma_2'$ by a finite index sublattice, we can guarantee that the map has reduced fibers. The subdivision can be chosen in such a way as to guarantee that $\Sigma_1'$ and $\Sigma_2'$ are both smooth. 
\end{proposition}

\begin{proof}
This is the statement of the toroidal strong semistable reduction theorem~\cite{ALT18}.\footnote{Note however that the weak semistable reduction theorem can be used to guarantee a combinatorially flat modification where $\Sigma'_2$ is smooth~\cite{AK00,Mol16}, and in fact, by an observation of Molcho, can be used to guarantee that source and target are {\it simplicial}, and this is always sufficient for our purposes. }
\end{proof}

\subsubsection{The relevance of combinatorial flatness} Combinatorial flatness arises in intersection problems for logarithmic schemes, and more generally in fiber products; it is the basis of the results on the degeneration formula in~\cite{R19}, and is explored in depth in~\cite{MR21}. 

The essential geometry can be seen for maps between simple normal crossings pairs. Suppose we have maps
\[
\begin{tikzcd}
& X\arrow{d} \\
Y\arrow{r} & B,
\end{tikzcd}
\]between simple normal crossings pairs, with divisors $D_X,D_Y$, and $D_B$. We assume that these are maps of pairs -- the preimage of $D_B$ in $X$ and in $Y$ is contained, respectively, in $D_X$ and $D_Y$. 

One can take the fiber product of these schemes in two different categories -- in schemes, and in fine and saturated logarithmic schemes; let us momentarily denote the latter $(X\times_B Y)^\mathsf{log}$. By abstract nonsense, there is a map
\[
(X\times_B Y)^\mathsf{log}\to X\times_B Y.
\]
In fact, the map is always finite, but is typically not surjective. If $X\to B$ is, in addition, dominant, then this map is the normalization of a component. 

The basic tension in our results is that the methods of intersection theory are directly applicable to $X\times_B Y$ but the geometric spaces of interest will be expressed as logarithmic fiber products, so play the role of $(X\times_B Y)^\mathsf{log}$. As a result, the following proposition is useful. 

Given $X\to B$, there is an induced map of cone complexes $\Sigma_X\to\Sigma_B$. We maintain the setup above.  

\begin{proposition}
Suppose the induced map $\Sigma_X\to\Sigma_B$ is combinatorially flat with reduced fibers. Then the map
\[
(X\times_B Y)^\mathsf{log}\to X\times_B Y.
\]
is an isomorphism.
\end{proposition}

\begin{proof}
The result follows from general facts about integral and saturated maps and fiber products in the categories of fine, and fine and saturated logarithmic schemes; see~\cite{Kat89}. 
\end{proof}

\section{Moduli spaces of curves and sheaves}\label{sec: review-of-moduli}

We present an overview of logarithmic DT and PT theory, as developed in~\cite{MR20} and explain how logarithmic GW theory can be constructed in a parallel fashion, using expansions. We include a brief sketch the relationship with the earlier literature~\cite{AC11,Che10,GS13,R19}. 

Let $(X,D)$ be a pair of a smooth projective threefold with a simple normal crossings divisor. The components of $D$ are $D_1,\ldots, D_r$. We assume that the intersection of any non-empty subset of divisors is connected. This assumption is not essential, but simplifies the exposition.

\subsection{Expansions of a pair} Logarithmic DT theory is based on a moduli space of subschemes in {\it expansions} of $X$ along $D$, subject to transversality and stability conditions. 

Let $\mathcal J$ be an ideal sheaf on $X$, with associated closed subscheme $Z$. Recall that the {\it deformation to the normal cone of $\mathcal J$} is the fiber over $0$ of the morphism
\[
\mathsf{Bl}_{\mathcal J+(t)} \left(X\times\mathbb A^1\right)\to\mathbb A^1.
\]
We first define an expansion -- of $X$ along its divisor $D$ -- over a closed point. These are the special fibers of particular types of deformations to the normal cone. Each component $D_i$ of $D$ is a Cartier divisor, which we can think of as giving rise to a pair $(L_i,s_i)$ of a divisor and a section. These $s_i$ are the monomials in the logarithmic structure. A {\it monomial ideal sheaf} on $X$ with respect to $D$ is an ideal sheaf generated by monomials in the $s_i$. A {\it monomial subscheme} of $X$ with respect to $D$ is a subscheme of the form $\mathbb V(\mathcal J)$, where $\mathcal J$ is a monomial ideal sheaf.

\begin{definition}
    A {\it rough expansion} of a pair $(X,D)$ over $\Spec\mathbb C$ is the special fiber of a deformation to the normal cone of a monomial subscheme of $X$ with respect to $D$. A rough expansion is reduced if it is reduced as a scheme. 
\end{definition}

If a rough expansion is a normal crossings variety, it is naturally stratified by intersections of the irreducible components. In fact, a monomial deformation to the normal cone of $(X,D)$ is always a toroidal embedding in the sense of~\cite{KKMSD}, with divisor given by the strict transform of $D\times \mathbb A^1$ union the special fiber. As a result, it is locally isomorphic to a pair consisting of a toric variety and its toric boundary, and therefore is naturally stratified. 

If $\mathcal X$ is a monomial deformation to the normal cone, the {\it rank $1$ locus} is the union of locally closed strata of codimension $0$ and $1$. Its {\it rank $0$} locus is the union of locally closed strata of codimension $0$.

\begin{definition}
    A \textit{expansion} of $(X,D)$ over $\Spec(\mathbb C)$ is obtained from a reduced rough expansion by taking a union of strata contained inside the rank $1$ locus. An expansion has rank $0$ if it is a union of strata in the rank $0$ locus, and rank $1$ otherwise.
    \end{definition}

We record a few examples of expansions. 

\begin{example}
    The simplest rough expansions are obtained from deformation to the normal cones of strata of $D$ in $X$. Note that if $W$ is a stratum, then the total space of $\mathsf{Bl}_{W\times\{0\}} X\times\mathbb A^1$ is itself a simple normal crossings pair, and therefore comes equipped with a natural stratification. We can therefore perform further blowups, centered in the special fiber, to produce new flat degenerations of $X$. These are projective birational modifications of $X\times\mathbb A^1$ with monomial indeterminacy locus. By performing a base change to ensure reducedness and then removing the triple intersection of the components, we are left with a large class of expansions. In fact, all the expansions we will  need arise in this fashion.
\end{example}

An expansion $\cX$ is said to have {\it rank $0$} if it has no strata of codimension $1$. For example, the complement $X\setminus D$ has rank $0$, and if $\cX$ is an expansion of rank $1$, unions of the interiors of all irreducible components is a rank $0$ expansion. 

\subsection{The moduli stack of expansions}\label{sec: stack-of-expansions} The notion of an expansion is essentially elementary and can be made without reference to tropical geometry. However, the corresponding moduli theory is best handled via tropical geometry~\cite{MR20}. We first summarize the purely geometric output, before explaining the tropical input.

The main construction in~\cite{MR20} outputs, for each $(X,D)$, a directed partially ordered set $(\Lambda,\prec)$ and for each choice of $\lambda\in \Lambda$ an Artin fan $\mathsf{Exp}(X)_\lambda$. The set $\Lambda$ is typically infinite. If $\lambda \prec \nu$ there is a logarithmic modification
\[
\mathsf{Exp}(X)_\nu\to\mathsf{Exp}(X)_\lambda.
\] 
We remark further on the precise structure of the set $\Lambda$ in Section~\ref{sec: tropical-moduli-spaces} below, but for now, we note two basic properties: (i) that it is analogous to the system of modifications discussed above and in particular is closed under logarithmic modifications including generalized root constructions, and (ii) the partially ordered set $\Lambda$ does {\it not} typically have a terminal element. 

\noindent
{\bf Terminology and notation.} In the main text, we will fix a choice of $\lambda$ and then suppress it from the notation. We will frequently show that results hold for {\it sufficiently fine} $\mathsf{Exp}(X)_\lambda$, i.e. after replacing the choice with any $\lambda$ such that $\lambda_0\prec \lambda$, for some fixed $\lambda_0$. 


For each such fixed choice, points of $\mathsf{Exp}(X)_\lambda$ correspond roughly to expansions of $X$ along $D$. To do this precisely, we must track another system of combinatorial choices, analogous to $\Lambda$, but this time for the universal family. Precisely, there is another another partially ordered set $L$ and a system of spaces $\{\cX\}_{\ell\in L}$. Each $\cX_\ell$ is birational to $\mathsf{Exp}(X)\times X$. For each choice of $\ell$ and $\lambda$ we can therefore ask if the natural rational projection from $\cX_\ell$ to $\mathsf{Exp}_\lambda$ extends to a morphism. Furthermore, as can ask if this morphism is flat with reduced fibers. Say that such choices of $\ell$ and $\lambda$ are {\it compatible}.

For compatible choices $\ell$ and $\lambda$, we have a ``universal'' diagram:
\[
\begin{tikzcd}
    \mathcal X_\ell\arrow[swap]{d}{\pi}\arrow{r}{b} & X\\
    \mathsf{Exp}(X)_\lambda&.
\end{tikzcd}
\]
The morphism $\pi$ is flat and the fibers of $\pi$ are expansions of $X$, presented as expansions of $X$ by the restriction of $b$ to the fiber.  

In the main text, we use the phrase {\it ``fix a stack of expansions''}, and by this, we will mean a choice of $\cX_\ell$ and $\mathsf{Exp}(X)_\lambda$, compatible in the sense above. The $\ell$ and $\lambda$ will be then be suppressed from the notation.

\subsubsection{Tube components} A key subtlety dealt with in~\cite{MR20} is that of \textit{tube components}, which are certain distinguished components in expansions. To understand their need, we first observe that given distinct points of $\mathsf{Exp}(X)$, the fibers of $\pi$ over these points may be naturally isomorphic over $X$. Alternatively, the isomorphism type, over $X$, of an expansion over $\mathbb C$ does not uniquely determine a moduli map. Additional data is needed to uniquely determine a moduli map. The ambiguity is understood as follows. 

Let $\mathcal X$ be an expansion over $\Spec(\mathbb C)$. An irreducible component $W$ of $\mathcal X$ is called a \textit{$\mathbb P^1$-bundle} if there is another expansion $\mathcal X_c$ and a morphism
\[
\mathcal X\to \mathcal X_c
\]
over $X$, such that (i) the map is birational on all components except $Y$, and (ii) the component $W$ in question is a $\mathbb P^1$-bundle over its image in $\mathcal X_c$. 

Given a point of $\mathsf{Exp}(X)$, the fiber of $\pi$ over this point may have several $\mathbb P^1$-bundle components. However, by the construction of the stack $\mathsf{Exp}(X)$ in our earlier work, for every such fiber, a subset of the $\mathbb P^1$-bundle components are designated as \textit{tube components}. We emphasize that this is part of the data of the point in $\mathsf{Exp}(X)$, and a map to the stack $\mathsf{Exp}(X)$ gives rise to an expansion \textit{together} with the designation of tube components in each fiber. 

The locus of rank $0$ expansions of $X$ is an open and closed substack $\mathsf{Exp}_0(X)\subset\mathsf{Exp}(X)$. In the main body, we will typically examine these for the divisor components $D_i$, i.e. we will use $\mathsf{Exp}(D_i$).

\subsection{Tropical moduli spaces}\label{sec: tropical-moduli-spaces} The main method for understanding the stacks $\mathsf{Exp}(X)$, and the method used for its construction in~\cite{MR20}, is the Artin fan/cone space dictionary.

Given the pair $(X|D)$, one canonically associates a cone complex $\Sigma_{X}$ or simply $\Sigma$ when $(X,D)$ or $D$ is clear from context or irrelevant. The divisor $D$ consists of components $D_1,\ldots, D_r$. Consider the fan $\RR_{\geq 0}^r$ with its natural integral structure $\mathbb N^r\subset \RR_{\geq 0}^r$. The faces of $\Delta_r$ correspond to subsets $I$ of $\{1,\ldots, r\}$. The corresponding intersection $D_I$ of the divisors $D_i$, over $i\in I$ is a stratum of $(X|D)$ that may or may not be empty. 

\begin{definition}
    The cone complex $\Sigma_X$ associated to $(X|D)$ is the subfan of $\RR_{\geq 0}^r$ given by the union of cones corresponding to nonempty intersections $D_I$. 
\end{definition}

The key combinatorial notion introduced in~\cite{MR20} is that of a \textit{$1$-complex}, which is an embedded polyhedral complex in $\Sigma_X$ of dimension at most $1$. A $1$-complex is not required to be pure dimensional. Similarly, a {\it $0$-complex} is an embedded polyhedral complex of dimension $0$, so is a finite set of points in $\Sigma_X$. Experts should note that that $1$-complexes do not have any edge or vertex decorations, nor is any balancing condition imposed. 

In~\cite{MR20}, we build a topological space $|T(\Sigma)|$ which is a moduli space for families of $1$-complexes. We sometimes use the notation $|T|$ when $\Sigma$ is clear. The space comes with a natural universal family $|G|\to|T|$, and the fibers are $1$-complexes. Although it is not a cone space or even a cone stack, the set $|T|$ genuinely functions as a moduli space -- a point of $|T|$ is uniquely specified by a $1$-complex, and the fiber of  $|G|\to|T|$ is the corresponding $1$-complex. 

The spaces $|T|$ and $|G|$ are the underlying sets of cone complexes -- however, there is an entire family of such cone space structures, and there is no (obvious) minimal representative. Instead, we \textit{choose} a cone complex for both $|T|$ and $|G|$ to form $G\to T$, such that this map is combinatorially flat. 

The set of all cone space structures on $|T|$ forms an inverse system, organized by a directed partially ordered set $(\Lambda,\prec)$, and two structures $T_\lambda$ and $T_{\lambda'}$ always have a smooth $T_{\lambda''}$ that dominates both via subdivisions. We will choose a cone complex structure $G_\lambda$ on $|G|$ such that the map to $G_\lambda\to T_\lambda$ is combinatorially flat. We typically drop $\lambda$ from the notation.

This is exactly parallel to the structure we have already explained for $\mathsf{Exp}(X)$ -- the system of cone space structures exactly matches the system of birational models. 

We explain the tube components in the previous section in more concrete terms. Fix a point $p$ in $|T|$. The fiber in $|G|$ over $p$ is a $1$-complex $|G|_p$. The corresponding fiber in $G\to T$ is a $1$-complex $G_p$. The spaces $|G|_p$ and $G_p$ differ only in polyhedral structure -- the latter is obtained from the former by subdividing edges by adding $2$-valent vertices. We call these \textit{tube vertices}. 

Fix cone complex structures compatible with a map to obtain $G\to T$. We can choose the cone structure on $T$ in such a way that the map $G\to T$ is combinatorially flat with reduced fibers, see~\cite[Section~3.5]{MR20}. We will not need $G$ to be smooth, but note that it can be guaranteed. Indeed, applying semistable reduction (resp. weak semistable reduction) we can ensure that the total space $G$ is smooth (resp. simplicial), see~\cite{AK00,ALT18}. 

Both $G$ and $T$ are cone spaces and have associated Artin fans. If $T$ is presented as a colimit
\[
T=\varinjlim_D \sigma,
\]
of rational polyhedral cones, the stack $\mathsf{Exp}(X)$ is defined by
\[
\mathsf{Exp}(X):=\varinjlim_D A_\sigma.
\]
In particular we have the following:

\begin{proposition}
    The points of the stack $\mathsf{Exp}(X)$ are in natural bijection with the cones of the stack $T$. 
\end{proposition}

Similarly, the space $G$ is a possibly incomplete subdivision of $\Sigma_X\times T$. The dictionary relating subdivisions of cone spaces with birational models gives us a non-proper birational model of $X\times \mathsf{Exp}(X)$. We denote it $\cX$.  Since $\mathsf{Exp}(X)$ is smooth, combinatorial flatness of $G\to T$ ensures flatness of the family
\[
\cX\to \mathsf{Exp}(X). 
\]

Let us spell out the correspondence a little more. Given a cone $\tau$ of $T$, we have an associated point $e\in\mathsf{Exp}(X)$. If $p\in |T|$ is in the interior of this cone, the vertices of $G_p$ are in bijection with irreducible components of the expansion 
$\mathcal X_e$ of $X$ lying over this point $e$.
As we have noted, $G_p$ has a subset of $2$-valent vertices obtained by subdividing $|G|_p$ called tube vertices. The corresponding components of $\mathcal X_e$ are known as \textit{tube components}. 

The stack of rank $0$ expansions $\mathsf{Exp}_0(X)$ is naturally in bijection with a space parameterizing $0$-complexes. We denote this $\mathsf{P}(\Sigma)$. It is a union of connected components in $T(\Sigma)$. 

\subsection{Moduli spaces of sheaves} With the construction of $\mathsf{Exp}(X)$ completed, we can recall the definition of the DT and PT moduli spaces. Consider an expansion $\mathcal X_p$ over $\Spec(\mathbb C)$.

\begin{definition}\label{def: dt-stability}
A subscheme $Z\subset \mathcal X_p$ of dimension at most $1$ is \textit{algebraically transverse} if all locally closed strata of $\cX_p$ meet $Z$, and for any codimension $1$ stratum $S$ (i.e. either double intersections or boundary divisors), the multiplication map
\[
\mathcal I_Z\otimes \mathcal O_S\to \mathcal O_{\mathcal X_p}\otimes \mathcal O_S,
\]
is injective. The tensor product is taken on $\mathcal X_p$. 
\end{definition}

Concretely, this means that no embedded points or components are contained in any divisor stratum, and furthermore, that the subscheme on the two sides of a double divisor induce the same subscheme on the divisor. The condition is the same as the one in ``traditional'' relative DT theory~\cite{LiWu15,MNOP06b}. The departure comes from the next definition. For it, note that we fix a stack $\mathsf{Exp}(X)$, and therefore given a point $\Spec(\mathbb C)\to \mathsf{Exp}(X)$, the pullback of the universal expansion comes with a distinguished set of tube vertices. 

\begin{definition}
    Fix a map $p\colon\Spec(\mathbb C)\to \mathsf{Exp}(X)$ and let $\mathcal X_p$ be the associated expansion. Let $Z\subset \mathcal X_p$ be a algebraically transverse subscheme. It is called \textit{DT stable} if for every tube component $Y$, the subscheme $Z\cap Y$ is stable under the $\mathbb G_m$-action in the fibers of the associated $\mathbb P^1$-bundle. 
\end{definition}

Putting these pieces together, we can define the DT moduli space. 

\begin{definition}
    The DT moduli stack is the fibered category over schemes, whose fiber over $S$ is the groupoid of maps $S\to \mathsf{Exp}(X)$ together with families
\[
\begin{tikzcd}
    Z\arrow{dr} \arrow{rr}& & \mathcal X\arrow{dl}\arrow{r} & X\\
    & S, & &
\end{tikzcd}
\]
such that over every geometric point, the subscheme $Z$ is algebraically transverse and DT stable.
\end{definition} 

For stable pairs, we proceed similarly.

\begin{definition}
Fix a map $p:\Spec(\mathbb C)\to \mathsf{Exp}(X)$ and associated expansion $\mathcal X_p$.  A PT stable pair on $\mathcal X_p$ is given by a pair $(\mathcal E, \sigma)$ where $\mathcal E$ is a pure one-dimensional sheaf on $\mathcal X_p$ and 
$$\sigma: \mathcal{O}_{\mathcal{X}_p} \rightarrow \mathcal E$$
is a section satisfying the following conditions:
\begin{enumerate}[(i)]
\item the cokernel of $\sigma$ is supported in dimension zero and away from the rank $1$ strata and tube components. In a neighborhood of these loci, the stable pair is given by the structure sheaf of a pure one-dimensional subscheme, with its canonical section.
\item
The support of $\mathcal E$ satisfies the same algebraic transversality conditions of 
Definition \ref{def: dt-stability}.
\item
On every tube component $Y$, the stable pair is invariant under the $\mathbb G_m$-action in the fibers of the associated $\mathbb P^1$-bundle.
\end{enumerate}

\end{definition}

The PT moduli stack is defined in parallel to the DT moduli stack, and again lives over 
$\mathsf{Exp}(X)$.

The pushforward $\beta$ of the curve class $Z$ (or $[\mathcal E]$) to $X$ and the holomorphic Euler characteristic $\chi$ of $Z$ (or $\mathcal E$) are both locally constant in flat families. Once these invariants are fixed, the resulting substacks $\mathsf{DT}_{\beta,\chi}(X)$ and
 $\mathsf{PT}_{\beta,\chi}(X)$ possess all the required properties to define logarithmic DT and PT theory. 

\begin{theorem}
    The stack $\mathsf{DT}_{\beta,\chi}(X)$ is a proper Deligne--Mumford stack. The forgetful morphism
    \[
    \mathsf{DT}_{\beta,\chi}(X)\to\mathsf{Exp}(X)
    \]
    is equipped with a perfect obstruction theory. Since any finite type open substack of $\mathsf{Exp}(X)$ is pure of dimension $0$, virtual pullback gives rise to a virtual fundamental class on $\mathsf{DT}_{\beta,\chi}(X)$.  The analogous statements hold for 
    $\mathsf{PT}_{\beta,\chi}(X)$.
\end{theorem}

\begin{remark}
Since subschemes and stable pairs are the same near all rank $1$ strata, most of our arguments 
are the same for DT and PT moduli spaces.  Going forward, we will explain proofs in the DT setting
and indicate the necessary modifications, if any, for the PT moduli space.

\end{remark}

\subsection{Moduli spaces of expanded stable maps}\label{sec: expanded-maps} 
We turn to logarithmic stable maps and Gromov-Witten theory.

There are two different flavours of moduli space of logarithmic stable maps to a simple normal crossings pair $(X,D)$ that produce the same invariants. In the original construction~\cite{AC11,Che10,GS13}, a logarithmic stable map is simply a logarithmic map $C\to X$ from a logarithmically smooth curve, such that the underlying map on schemes is a stable map. These are shown to form a moduli stack $\mathsf M(X)$ equipped with logarithmic structure.

In earlier work of the second author, a different theory is presented. A moduli space $\mathsf{GW}(X)$ is constructed, whose closed points parameterize maps
\[
C\to \cX\to X,
\]
where $\cX\to X$ is an expansion in the sense we have just discussed. However, as originally presented, the construction of $\mathsf{GW}(X)$ is not logically independent from $\mathsf M(X)$. Rather, the former space is {\it constructed} as a logarithmic modification
\[
\mathsf{GW}(X)\to\mathsf M(X).
\]
In particular, the situation, as presented this space does not naturally live over the stack $\mathsf{Exp}(X)$. The situation is easily remedied, and the space may be constructed by using the methods of~\cite{MR20}, without relying on $\mathsf M(X)$.

We recall that a prestable map $C\to \mathcal X$ to an expansion is \textit{predeformable} if the preimage of every double divisor of $\cX$ is a union of nodes, and the tangency order at the two preimages of the node under normalization are equal, see~\cite{Li01}. 

Recall the notions of $\mathbb P^1$-bundle components and tube components from Section~\ref{sec: stack-of-expansions}.

\begin{definition}
    Fix a map $p:\Spec(\mathbb C)\to \mathsf{Exp}(X)$ and let $\cX_p$ be the associated expansion. A predeformable prestable map $f: C\to\cX_p\to X$ is \textit{GW stable} or simply \textit{stable} if every contracted irreducible component has at least three marked or nodal points and if, for every tube component $Y$, each component of $f^{-1}(Y)$ is stable under the $\mathbb G_m$-action in the fibers of the associated $\mathbb P^1$-bundle.  
\end{definition}

Over the stack $\mathsf{Exp}(X)$ is a birational modification:
\[
\cX\to X\times\mathsf{Exp}(X),
\]
from the universal expansion. Given a divisor $D_j\subset X$, the strict transform gives rise to a divisor expansion $\cD_j$. This induces a Cartier divisor on any expansion $\cX/S$.

\begin{definition}
Given a map $C\to\cX$ with a marked point $p_i$ on $C$, for each divisor $D_j$, the {\it tangency order} or {\it contact order} of $C$ with $D_j$ at $p_i$ is the usual scheme theoretic order of vanishing of the equation of the $\cD_j$ at $p_i$.
\end{definition}

We identify the moduli space of expanded stable maps as the moduli stack over $\mathsf{Exp}(X)$ parameterizing diagrams
\[
\begin{tikzcd}
    \mathcal C\arrow{dr} \arrow{rr}& & \mathcal X\arrow{dl}\arrow{r} & X\\
    & S, & &
\end{tikzcd}
\]
of predeformable prestable maps that are GW stable. 

The resulting stack $\mathsf{GW}'(X|D)$ can be shown to be algebraic, separated, and universally closed; the prime in the notation is there to indicate that this not quite the space we want to work with -- this has to do with saturating the logarithmic structure. It is straightforward to carry out stable maps analogues of the results in~\cite{MR20} and this is sketched in~\cite{CN21}.
We briefly remark on the key steps. We fix a choice of $\mathsf{Exp}(X)$. 
\begin{enumerate}[(i)]
\item (Constructing the mapping stack) First, we can consider the stack sending a scheme $S$ to the family of maps from prestable curves over $S$ to families of expansions arising from maps $S\to \mathsf{Exp}(X)$. This is an algebraic stack by general results on representability of mapping stacks~\cite[\href{https://stacks.math.columbia.edu/tag/0D19}{Tag 0D19}]{stacks-project}.

In this mapping stack, we have not yet imposed the tangency condition. We consider the locus of maps $C\to\cX\to X$ that are {\it predeformable}, i.e. where the preimages of double divisors in $\cX$ are nodes of $C$, the contact orders on the sides of each such node match, and preimages of boundary divisors on $\cX$ are unions of marked points. We now explain how to deduce algebraicity of this stack. 

Fix a predeformable map $C\to\cX\to X$. The scheme $\cX$ is equipped with a logarithmic structure, and one can equip $C$ with a logarithmic structure in the usual way for nodal curves. Predeformability is then precisely the condition that the map $C\to \cX$ can be enhanced to map of logarithmic schemes -- but in the category of fine (rather than fine and saturated) logarithmic schemes. Further discussion can be found in~\cite[Section~3]{AMW12}. Algebraicity of this locus in the mapping stack can now be deduced from Wise's general results~\cite{Wis16a}. 

\item (GW stability) The GW stability condition is an open condition in families. Indeed the proof on the Hilbert scheme side is found in~\cite[Section~4.3]{MR20} and the proofs are easily adaptable to the stable maps case. The GW stability condition guarantees finiteness of automorphism groups, see for instance~\cite[Section~4.5]{MR20}. The consequence of these results is that $\mathsf{GW}'(X)$ is a Deligne--Mumford stack equipped with a map
\[
\mathsf{GW}'(X)\to\mathsf{Exp}(X).
\]

\item (Properness) The properness of $\mathsf{GW}'(X)$ can also be completed in parallel to the Hilbert scheme case. Suppose $S^\circ$ is the spectrum of a discrete valued field and $S$ is the spectrum of a discrete valuation ring. Given a family of predeformable maps $C\to X$ over $S^\circ$ then the image of $C$ is a subscheme $Z$ defined over $S^\circ$. The arguments of~\cite[Section~2.6]{MR20} give rise to a ramified base change $S'\to S$ and an expansion $\cX'\to S'$ such that the usual stable maps limit of the map exists (note that $\cX$ is typically non-compact so this is a nontrivial statement). Since $\cX'$ has logarithmic rank $1$, the original relative stable maps theory~\cite{Li01} gives rise to a further base change and a further expansion $\cX''\to S''$ such that the stable maps limit of $C\to X$ over $S''$ now exists as a predeformable relative stable maps satisfying GW stability. The above argument can be adapted to the case when the map over $S^\circ$ itself has an expanded target, using the formalism of~\cite[Section~7]{MR20}. 

Boundedness follows similarly. If we fix the genus, marked points, curve class, and tangency orders the balancing condition of~\cite{GS13}, adapted also in~\cite{MR20}, shows that only a finite type substack of $\mathsf{Exp}(X)$ can carry a stable map with these discrete data. Now fix an expansion $\cX\to X$. Each irreducible component component of this expansion can be expressed as a projective bundle over a stratum in a blowup of $X$. Since the curve class is fixed, the pushforward to $X$ of the curve class of any map to such a component is bounded. But by the balancing condition, the contact orders of such a curve with the boundary of this component are also bounded. It follows that the curve class attached to each component is bounded. Since the genus is fixed from the outset, the mapping stack with this fixed target is also bounded. We conclude that the space $\mathsf{GW}'(X)$ is of finite type. Altogether, we deduce properness of the stack. 

\item (Virtual structure) The moduli space $\mathsf{GW}'(X)$ comes with a map
\[
\mathsf{GW}'(X)\to \mathsf{Exp}(X).
\]
This map is equipped with a relative perfect obstruction theory, following~\cite{Li01} or~\cite{AMW12}. The $K$-theory class of the obstruction theory is given, over a map $f:C\to\cX\to X$, by the vector bundle with fiber $H^0(C,f^\star T_X^{\mathsf{log}})-H^1(C,f^\star T_X^{\mathsf{log}})$. Note that since the universal expansion has logarithmic rank $1$, these older arguments, originally developed for smooth pairs, apply without change here.
\end{enumerate}

The logarithmic structure on $\mathsf{GW}'(X)$ is typically not saturated. However, we can pass to the saturation to form the moduli space $\mathsf{GW}(X)$. The logarithmic structures are defined so that the resulting map
\[
\mathsf{GW}(X)\to\mathsf{Exp}(X)
\]
is strict. 

\begin{remark}
Note that the use of logarithmic mapping stacks, and the saturation issues, can be simplified by using orbifold structures following Abramovich--Fantechi~\cite{AF11}. Precisely, in Step (i) above, we can introduce root stack structures along the double divisors and boundary divisors, and work with the moduli stack of orbifold curves. The usual mapping stack, in orbifolds, contains the mapping stack that we want as an open substack. Note that this effectively removes the need for logarithmic structures in logarithmic Gromov--Witten theory, and provides an exact normal crossings parallel to~\cite{AF11}. 
\end{remark}

\subsubsection{Bookkeeping the tangency data} 

To make the compatibility with sheaf theory clear, we handle tangency indexing slightly differently from the logarithmic GW theory literature~\cite{AC11,GS13}. Let $n_i$ be the intersection number $D_i\cdot \beta$. Let $\mathbf n = (n_1,\ldots, n_r)$. Let $\mu_i$ be an ordered partition of $n_i$ and let $\bm \mu = (\mu_1,\ldots, \mu_r)$ be the resulting vector of partitions. Finally, fix an auxiliary positive integer $k$. 

We consider maps
\[
f: C\to \cX\to X,
\]
with $\sum_i \ell(\mu_i)+k$ marked points, or more precisely, a set of marked points in bijection with the set of parts of all paritions in $\bm{\mu}$ and $k$ additional marked points. Let us label these:
\[
p^{D_1}_1,\ldots,p^{D_1}_{\ell(\mu_1)},\ldots,p^{D_r}_1,\ldots,p^{D_r}_{\ell(\mu_r)},q_1,\ldots, q_k.
\]
The contact order recorded by $\bm{\mu}$ demands that $p^{D_i}_j$ has positive contact order only with the $D_i$, and zero with all other divisors. Moreover, the contact order of $D_i$ with $p_j^{D_i}$ is equal to the $j^{th}$ part of the partition $\mu_i$. 

The condition that each point has nonzero contact with at most one divisor is sometimes called the {\it disjoint} case. Restriction to the disjoint case does not restrict the class of invariants we can consider. By logarithmic birational invariance~\cite{AW}, a space can always be replaced by a blowup so this condition holds. However in DT and PT theory, the disjointness hypothesis is part of the definition.

\subsubsection{Notation} When fixing the tangency data, we denote the moduli space $\mathsf{GW}_{g,\beta}(X)_{\bm\mu,k}$, and when fixing the number of marked points alone we use $\mathsf{GW}_{g,\beta}(X)_n$. We will suppress any irrelevant symbols, i.e. parameters whose value does not enter into the discussion, when there is no danger of confusion.

\subsection{Comparisons}\label{subsec:comparisons} Before proceeding with a study of these spaces, we mention how they are related to other moduli space of logarithmic stable maps. For the purposes of this paper, we note that the key point is that the degeneration formula of~\cite[Theorem~B]{R19} applies to the spaces of maps $\mathsf{GW}(X)$ constructed in Section~\ref{sec: expanded-maps}. 

We first note that the space $\mathsf{GW}(X)$ above can be shown to be isomorphic to a logarithmic modification of the Abramovich--Chen--Gross--Siebert moduli space $\mathsf M(X)$~\cite{AC11,Che10,GS13} as {\it algebraic stacks}, though not with the same logarithmic structure. One way to deduce this is to first observe that the construction in~\cite{R19} produces a $\mathsf M^\dagger(X)$ that carries a modified universal curve and a family of expansions of $X$ such that the map is GW stable:
\[
\begin{tikzcd}
\cC\arrow{rr} \arrow{dr} & &\cX\arrow{dl} \arrow{r} & X\\
& \mathsf M^\dagger(X).& &
\end{tikzcd}
\]
We can therefore find a morphism
\[
\mathsf M^\dagger(X)\to\mathsf{Exp}(X)
\]
for {\it some} choice of $\mathsf{Exp}(X)$. This produces a morphism to the space $\mathsf{GW}(X)$. By comparing the moduli functors directly, this map is can be seen to be an isomorphism (for an appropriate choice of $\mathsf{Exp}(X)$). 

\subsection{Decorated tropical moduli and stacks of expansions}\label{sec:decorated}

Recall from Section~\ref{sec: tropical-moduli-spaces} $1$-complexes of $\Sigma$ are undecorated embedded polyhedral complexes.  It will be convenient to allow decorations by discrete invariants associated to our moduli problems, especially for the purposes of working with generating functions.  

\begin{definition}
A Chow $1$-complex of $\Sigma$ is a $1$-complex, such that each edge $e$ or ray $r$ is decorated by a positive integer $n_e$ and each vertex $v$ is decorated by an effective curve class $\beta_v \in H_2(Y_v)$, where $Y_v$ is the closure of the stratum associated to $v$.
\end{definition}

Given a $1$-dimensional subscheme $Z$ in an expansion $\mathcal{X}$ of $X$, then we obtain a Chow $1$-complex as follows. We decorate each ray $r$ with the intersection number of $Z$ with the divisor $D_r$ corresponding to $r$; if $e$ is an edge incident to a vertex $v$, there is a divisor $D_e$ in the expansion component $X_v$, and we decorate the edge with the intersection number of the restricted subscheme $Z_v$ with $D_e$. For each vertex $v$, we take the restriction $Z_v$ to the irreducible component associated to $v$ and pushforward its cycle class $[Z_v]$ to $Y_v\subset X$.  We can recover the total curve class $[Z_v]$ from the data of the decorations at $v$ and its adjacent edges and rays, which is why we use the adjective Chow.

Recall that there are natural specializations among the combinatorial types of $1$-complexes, essentially given by contracting edge lengths to $0$, see~\cite[Section~2]{MR20}.  For Chow $1$-complexes, we demand that the decoration is additive: if a subset of vertices specialize to the same vertex, the corresponding decorations should be added.

If we fix $\beta \in H_2(X)$, we can consider tropical cycles whose vertex decorations sum to $\beta$ and obtain a topological space $|T_\beta(\Sigma)|$, equipped with a universal family and a forgetful map $|T_\beta(\Sigma)| \rightarrow |T(\Sigma)|$.  If we fix a cone structure $T(\Sigma)$, it induces a natural cone structure on $|T_\beta(\Sigma)|$ as follows; for each cone $\tau$ of $T(\Sigma)$ and distribution of labels to the vertices of the $1$-complex associated to a point interior of $\tau$, we define a cone of $T_\beta(\Sigma)$ by taking the corresponding
locus in the preimage of $\tau$.  Face maps are defined in the natural way.  The result is a map of cone complexes
$$T_\beta(\Sigma) \rightarrow T(\Sigma)$$
which is surjective and an isomorphism upon restriction to any cone in the domain.

By now passing from the cone complexes to Artin fans, we see there is a natural non-separated \'etale cover:
\[
\mathsf{Exp}_{\beta}(X)\to \mathsf{Exp}(X).
\]
The geometry of the \'etale cover can be understood from the corresponding map of cone complexes.

The morphism from the GW/DT/PT spaces factorize through this cover; for example, we have
\[
\mathsf{DT}_{\beta,\chi}\to \mathsf{Exp}_{\beta}(X)\to \mathsf{Exp}(X).
\]
Because the map is \'etale, using this more refined Artin fan still produces the same virtual structure on $\mathsf{DT}_{\beta,\chi}$.

We can refine further by decorating each vertex with Euler characteristic $\chi_v$ as well,

\begin{definition}
A Hilbert $1$-complex of $\Sigma$ is a Chow $1$-complex such that each vertex $v$ is decorated additionally with an integer $\chi_v$, to indicate the Euler characteristic of the subscheme supported on this component.
\end{definition}

We use the adjective Hilbert to indicate that these decorations recover the Hilbert polynomial of the subscheme on each irreducible component of our expansion.
In families of Hilbert $1$-complexes, we demand that the Euler characteristic specializes in the natural way.

If we fix $\chi \in \ZZ$, we consider Hilbert $1$-complexes with total Euler characteristic $\chi$, and obtain a cone complex $T_{\beta,\chi}(\Sigma)$, again equipped with a map $T_{\beta,\chi}(\Sigma) \rightarrow T(\Sigma)$ that is surjective and an isomorphism on each cone.  
As before, we have factorizations
\[\mathsf{DT}_{\beta,\chi}(X), \mathsf{PT}_{\beta,\chi}(X)\to \mathsf{Exp}_{\beta,\chi}(X)\]
which give the same virtual structure.

We recall, as in~\cite{R19}, that an {\it $n$-marked abstract tropical curve} is an abstract $1$-dimensional polyhedral complex $\mathbf G$, together with a marking of its unbounded rays by $\{1,\ldots,n\}$. 

We consider integral piecewise linear maps
\[
\mathbf G\to \Sigma_X
\]
such that no bounded edge of $\mathbf G$ is contracted by this map. Note that unbounded rays may be contracted. 

Each of the unbounded rays has a {\it direction}, given by the directional derivative, directed towards the unbounded direction of the ray. The directional derivative is an element of the lattice  of the cone in $\Sigma_X$ that contains the ray. We make the assumption that the direction of each non-contracted ray is a non-negative multiple of a ray of $\Sigma_X$. Therefore, all the unbounded edges of $\mathbf G$ parallel to a given ray of $\Sigma_X$ naturally form a partition. As a result, given such a map, we can encode a tuple of partitions $\bm{\mu}$, which we refer to as the {\it contact order} of the map. 

Given a cone $\sigma$ in $\Sigma_X$, we introduce the notation $X_\sigma$ for the closed stratum dual to the cone. 

\begin{definition}
A {\it tropical stable map} $\mathbf G\to \Sigma$ is a integer piecewise linear map from an $n$-marked tropical curve together with two vertex decorations, satisfying the conditions above, and with the following decorations: (i) a {\it genus} denoted $g_v$, valued in $\mathbb Z_{\geq 0}$, and (ii) a {\it curve class} denoted $\beta_v$, valued in $H^+_2(X_v;\ZZ)$, for each vertex of $\mathbf G$. 

The genus of a tropical curve is equal to the sum of its first Betti number $b_1(\Gamma)$ and all vertex genera. The total curve class is the sum of the pushforwards of $\beta_v$ to $X$, taken over all vertices. 
\end{definition}

Given an unbounded ray of $\mathbf G$, its direction has the same data structure of a contact order in GW theory. As in the GW section, the contact orders are given by the data $\bm{\mu}$, and we have $k$ additional marked points with trivial contact order. We fix these tangency data, as well as the curve class $\beta$ and the genus $g$. 

Following~\cite{CCUW,R19}, the moduli stack of tropical maps above can be constructed, and denote $TSM_{\beta,g}(\Sigma)$. It is a cone stack with a map $TSM_{\beta,g}(\Sigma)\to T$. Note that we have disallowed contractions of bounded edges, but taking this into account, the description of the cone stack exactly follows~\cite[Section~2]{R19}.

Since every edge in a tropical stable map (with the definition above, and so without contracted bounded edges) maps with positive slope onto an edge of an expansion, this map $TSM_{\beta,g}(\Sigma)\to T$ is an isomorphism rationally on each cone. By adjusting the integral structure on $T$, and then passing to Artin fans, we now obtain a non-separated strict \'etale cover
\[
\mathsf{Exp}_{\beta,g}(X)_{\bm{\mu},k}\to \mathsf{Exp}(X).
\]

A tropical stable map $\varphi: \mathbf G\to \Sigma_X$ determines a Chow $1$-complex: the $1$-complex is the image of $\mathbf G$, and the curve class decoration at each vertex of $\mathsf{im}(\varphi)$ is equal to the sum of the vertex decorations at all vertices of $\mathbf G$ that map onto it. Note that strictly speaking, the image can have more vertices than $\mathbf G$, due to images of edges intersecting. The definition is nevertheless valid, with the understanding that the empty sum is equal to $0$. 

With all this as stated, we have, as before, a factorization:
\[
\mathsf{GW}_{\beta,g}(X)_{\bm{\mu},k}\to \mathsf{Exp}_{\beta,g}(X)_{\bm{\mu},k}\to \mathsf{Exp}(X).
\]

\section{Moduli spaces of points and evaluation maps}\label{sec: points}

The DT/PT, and GW moduli spaces associated to pairs $(X,D)$ all admit evaluation maps. For the DT/PT moduli spaces, the target of the evaluation is the space of $0$-dimensional subschemes in expansions of components $D_i$ in $D$; in the GW case, the target is the space of maps from points to these expansions, which is just a modification of a product of divisor components. 

A basic conceptual asymmetry is present in the setup. In DT/PT theory, the evaluation maps are used both to {\it encode} tangency conditions and impose incidence conditions within the boundary strata, whereas in GW theory, the tangency conditions are imposed as part of the locally constant numerical data in the moduli space. The matching between the evaluation conditions is therefore nontrivial and plays a role in the conjectures.

\subsubsection{Evaluation on the non-degenerate locus} On the sheaf side, observe that there is an open subset of the Hilbert scheme of curves in $X$ on which we have a map to the Hilbert scheme of points on $D_i$:
\begin{eqnarray*}
-\cap D_i: \mathsf{Hilb}_\circ(X)&\to& \mathsf{Hilb}^d_\circ(D_i), \ \ \ d := \mathsf{deg}([Z]\cdot [D_i])\\
Z&\mapsto& Z\cap D_i.
\end{eqnarray*}
The circle on the domain indicates that we have passed to the open subspace of subschemes that are algebraically transverse to all components of $D$, and on the codomain, to the locus of points whose support does not meet $D_i\cap D_j$.

There is a similar map on the maps side:
\begin{eqnarray*}
\mathsf{GW}^\circ(X)_{{\bm \mu},k}&\to& D_i^{\ell(\mu_i)}, \\
{[F:C\to X]}&\mapsto& [F^{-1}(D_i)\to D_i].
\end{eqnarray*}
Again, the circle indicates that we have restricted to the open subspace of logarithmic maps where the target is unexpanded. Note that we are using the disjointness hypothesis on $\bm\mu$ here.

In general, given an expansion
$\mathcal X_p\to X$, for each divisor $D_i$, the expansion $\mathcal X_p$ determines an expansion 
$\mathcal{D}_{i,p}$ of $D_i$ along its induced boundary.  
Intersecting a subscheme or stable pair with this expansion gives
a $0$-dimensional subscheme in $\mathcal{D}_{i,p}$; algebraic transversality includes the condition of non-emptiness of intersections with strata, so
every component of $\mathcal{D}_{i,p}$ meets the subscheme nontrivially.  Similarly, restricting a stable map gives an ordered set of points in $\mathcal{D}_{i,p}$, indexed by the parts of the partition $\mu_i$.

In order to extend these to the compact moduli of logarithmic sheaves and maps, we first construct tropical evaluation maps and pass to geometric moduli spaces.

\subsection{Rank $0$ expansions} Each divisor component $D_i\subset D\subset X$ is itself a simple normal crossings surface pair, with divisor determined by the intersection of $D$ with the remaining components of $D$. In order to avoid overburdening notation, we use $\partial D_i$ to denote the boundary divisor on $D_i$. Let $\Sigma_i$ be the cone complex associated to $D_i$.

In Section~\ref{sec: tropical-moduli-spaces}, we described a topological space $|P(\Sigma_i)|$ whose points are in natural bijection with $0$-complexes in $\Sigma_i$, and have noted that it is the underlying topological space of a cone space, and in fact, several cone spaces, that are organized in an inverse system. We now choose such a cone space structure on $P(\Sigma_i)$. Note that when we work with $0$-complexes, there is no subtlety involving tube components -- the fibers of the universal family $G\to P(\Sigma_i)$ are $0$-complexes.

We pass to the Artin fan of $P$ and $G$ to obtain a family
\[
\begin{tikzcd}
    \mathcal D_i\arrow{d}{\pi}\arrow{r} & D_i\\
    \mathsf{Exp}_0(D_i),
\end{tikzcd}
\]
of expansions of $D_i$ along its boundary. The subscript $0$ indicates that we are looking at expansions that come from $0$-complexes. The morphism $\pi$ is smooth and typically has disconnected fibers. Every component is a torus bundle over a locally closed stratum in $D_i$, and has dimension equal to the dimension of $D_i$. 

\subsection{Evaluation spaces: moduli of logarithmic $0$-cycles} We will need to study the logarithmic Hilbert scheme of points on the disjoint union of the components $D_i$. One of the complications of the new logarithmic context, is that the logarithmic Hilbert scheme of a disjoint union is not a product of those corresponding to the factors\footnote{It is a subdivision of a product, and as such is rarely a product of subdivisions.}. 

In this section, we explain the picture for a single simple normal crossings surface pair, which is simpler. In order to control subscripts, we denote this surface $E$ with boundary divisor $\partial E$ though we will have $D_i$ and its boundary in mind. Its cone complex is denoted $\Sigma_E$. We fix a cone structure on the $P(\Sigma_E)$ as above, and work with the associated stack of rank $0$ expansions. 

We begin with a logarithmic version of the symmetric product.

\begin{definition}
A \textit{family of stable $0$-cycles on $E$ relative to $\partial E$ of length $n$ over $S$} is the data of a morphism
\[
S\to \mathsf{Exp}_0(E)
\]
and a flat family over $S$ of length $n$ $0$-cycles $A_S\hookrightarrow \mathcal E/S$ such that in each geometric fiber, every stratum of the expansion has a non-empty intersection with the $0$ cycle. We denote the resulting algebraic stack by $\mathsf{Sym}^n(E)$. 
\end{definition}

\begin{warning}
    The notation $\mathsf{Sym}^n(E)$ which should not be confused with the usual symmetric product of the underlying scheme of $E$. 
\end{warning}

\begin{proposition}
The moduli stack of stable $0$-cycles of length $n$ on $E$ relative to $\partial E$ is representable by a Deligne--Mumford stack and is equipped with a natural morphism from the logarithmic symmetric product:
\[
\mathsf{Sym}^n(E)\to \mathsf{Exp}_0(E)
\]
equipping it with a logarithmic structure by pullback. 
\end{proposition}

\begin{proof}
It is a simpler version of the main result of~\cite{MR20}. 
\end{proof}

\subsection{Stable maps from points} 

For evaluation in Gromov-Witten theory, we will require a version of the space above, but where the points in the $0$-cycle are ordered.
Given a scheme $S$, consider the moduli functor parameterizing a morphism 
\[
S\to \mathsf{Exp}_0(E)
\]
and sections $\sigma_i: S \rightarrow  \mathcal E/S$, for $i = 1\dots n$ of the universal expansion, such that $\sum \sigma_i$ defines a family of stable $0$-cycles.  

The moduli functor parametrizing these data is a Deligne--Mumford stack $\mathsf{Ev}^n(E)$, which is smooth over $\mathsf{Exp}(E).$ It is birational to the product $E^n$. 

There exists a natural action of the symmetric group $S_n$ on $\mathsf{Ev}^n(E)$ that permutes the labels.  If we equip $\mathsf{Sym}^n(E)$ with the trivial action of $S_n$, then the natural morphism
$$\mathsf{Ev}^n(E) \rightarrow\mathsf{Sym}^n(E)$$
over $\mathsf{Exp}_0(E)$ is $S_n$-equivariant.  

Suppose we are given a partition $\mu = (\mu^{(1)}, \dots, \mu^{(\ell)})$ of $n$.  Define the closed substack
$$\mathsf{Ev}^\mu(E) \subset \mathsf{Ev}^n(E)$$ to be the locus where, 
for each $1 \leq k \leq \ell$,
we have
$$\sigma_{\mu^{(1)}+\dots\mu^{(k-1)}+1} = \dots =\sigma_{\mu^{(1)}+\dots\mu^{(k)}}.$$
This is the locus where we have $\ell$ sections of the family of expansions, repeated with multiplicity given by the parts of $\mu$.

Let $$\mathsf{Sym}^{\mu}(E) \subset \mathsf{Sym}^{n}(E)$$ denote the image of $\mathsf{Ev}^\mu(E)$ under the quotient map.  
It is the closure of the locus of $0$-cycles of degree $n$ where the points have multiplicities given by the parts of $\mu$. 
Denote this relative interior, where the multiplicities of the $\ell$ sections are strictly as given, by $\mathsf{Sym}^{\mu, \circ}(E)$.

The map
$$\mathsf{Ev}^\mu(E) \rightarrow\mathsf{Sym}^\mu(E)$$
is a Galois quotient by the group $\mathrm{Aut}(\mu)$ of permutations preserving $\mu$.  It is \'{e}tale over
the open subset  $\mathsf{Sym}^{\mu, \circ}(E)$.
 These locally closed substacks form a stratification of $\mathsf{Sym}^{n}(E)$.

\subsection{Hilbert scheme of points on a logarithmic surface} 

On the sheaf side, once we have fixed $\mathsf{Exp}_0(E)$, then one can consider the moduli space of $0$-dimensional subschemes of length $n$ on the smooth locus of the universal expansion. This defines a 
logarithmic Hilbert scheme of points
$$\mathsf{Hilb}^{n}(E) \rightarrow \mathsf{Exp}_0(E)$$
which parametrizes families of expansions, along with a stable relative family of $0$-dimensional subschemes.  Here, stability is equivalent to stability of the associated family of $0$-cycles. See~\cite[Section~5.3]{MR20} for further details.

Using the Hilbert--Chow map, we obtain \textit{support morphisms}
\begin{equation}\label{eqn: support-map}
\begin{tikzcd}
\mathsf{Hilb}^{n}(E)\arrow{dr} & & \mathsf{Ev}^n(E)\arrow{dl}\\
& \mathsf{Sym}^{n}(E)\arrow{d}\\ &\mathsf{Exp}_0(E).
\end{tikzcd}    
\end{equation}

The diagram is compatible with modification in the following sense: if we change our model by a logarithmic modification
$$\mathsf{Exp}_0'(E) \rightarrow \mathsf{Exp}_0(E),$$ 
the corresponding evaluation spaces discussed are all modified by pulling back the above diagram via this modification.  

\subsection{Evaluation spaces for multiple boundary components} In DT theory, it is more natural to consider the Hilbert schemes of points on all boundary surfaces simultaneously. We explain how the picture \eqref{eqn: support-map} works in this setup.

We have fixed a threefold simple normal crossings pair $(X|D)$, and let $D_1, \dots, D_r$ denote the irreducible boundary components of $D$; each of these is a smooth surface $D_i$ equipped with a boundary $\partial D_i$ and cone complex $\Sigma_i$.  We will use the abuse of notation $(D, \partial D)$ to refer to the disjoint union of these pairs. 

We fix choices of polyhedral structures on $P(\Sigma_i)$ for all $i$, and thereby have stacks of expansions $\mathsf{Exp}_0(D_i)$ for each component. 
We also pick a subdivision of the product
\[
\left(\prod P(\Sigma_i)\right)^{\dagger} \rightarrow \prod P(\Sigma_i)
\]
and take the corresponding tropical model
$$\sigma:  \mathsf{Exp}_0(D) \rightarrow \prod \mathsf{Exp}_0(D_i).$$

\begin{remark}
As we have indicated, it will be important in what follows that this modification does not necessarily preserve the product structure. Indeed, when examining curves in threefolds, we will naturally be led to collections of $0$-dimensional subschemes on surfaces. The combinatorics of their degenerations are not independent, precisely because they are linked by the combinatorics of the expansion of $X$ carrying this curve. 
\end{remark}

Fix an $r$-tuple of non-negative integers $\mathbf{n} = (n_1, \dots, n_r)$.  For each $1 \leq i \leq r$, we have the relative Chow variety, evaluation space, and Hilbert scheme of points:
\[
\mathsf{Sym}^{n_i}(D_i), \ \ \mathsf{Ev}^{n_i}(D_i), \ \ \textsf{and}\ \  \mathsf{Hilb}^{n_i}(D_i)
\]
respectively, lying over a product of the stack of expansions for $D_i$. We define the corresponding evaluation spaces for $(D, \partial D)$ by taking their product and pulling back along the log modification $\sigma$.  The result is a sequence of spaces fitting into the diagram:
\[
\begin{tikzcd}
\mathsf{Hilb}^{\mathbf{n}}(D)\arrow{dr} & & \mathsf{Ev}^\mathbf{n}(D)\arrow{dl}\\
& \mathsf{Sym}^{\mathbf{n}}(D)\arrow{d}\\ &\mathsf{Exp}_0(D).
\end{tikzcd}
\]

Fix an $r$-tuple of partitions by $\bm{\mu} = (\mu_1, \dots, \mu_r)$ where, for each $1 \leq i \leq r$, the entry $\mu_i$ is a partition $n_i$.
We denote by 
$$\mathsf{Ev}^{\bm{\mu}}(D)$$ the modification, obtained by pullback along $\sigma$, of the product $$\prod_i\mathsf{Ev}^{\mu_{i}}(D_{i})\subset \prod_i \mathsf{Ev}^{n_i}(D_i)$$
and
$$\mathsf{Sym}^{\bm{\mu}}(D) \subset \mathsf{Sym}^{n}(D)$$
its image.
As before, the Chow variety version is obtained as the quotient by the action of 
$$\mathrm{Aut}(\bm{\mu}):= \prod_{i} \mathrm{Aut}(\mu_i).$$

\subsection{Tropical evaluation maps}\label{tropical-evaluation} We now discuss evaluation maps, and start with the tropical picture. We can use the star construction of Section~\ref{sec:star-of-cone} to define tropical evaluation maps, already considered in~\cite[Section~5.3]{MR20}.

Fix a cone complex $\Sigma$. We restrict our attention to $1$-complexes in $\Sigma$ whose unbounded rays are each parallel to some ray of of $\Sigma$ (these are the only ones that arise in practice).    Let
$$T^\circ(\Sigma)$$
denote the subcomplex of $T(\Sigma)$ parametrizing $1$-complexes with this property.

Fix a ray $\delta$ in $\Sigma$ and let $\Sigma(\delta)$ denote the star fan of the ray $\delta$. We let $P(\Sigma(\delta))$ denote the tropical moduli space
parametrizing $0$-complexes.
We will construct an evaluation map
\[
\mathsf{ev}_\delta: |T^\circ(\Sigma)|\to |P(\Sigma(\delta)).|
\]
Informally, it is described as follows. Given a $1$-complex, outside of a compact neighborhood of the origin of $\Sigma$, it consists of a disconnected union of rays, each parallel to one of the rays of $\Sigma$. In particular, there is a finite set of rays parallel to $\delta$, and by taking the quotient in the direction of $\delta$, we obtain a collection of points on $\Sigma(\delta)$. This defined $\mathsf{ev}_\delta$. We now construct this map formally. 

\begin{construction}
Fix a cone $\sigma$ in $T^\circ(\Sigma)$ and let $\mathbf G_\sigma\to \sigma$ be the associated family of $1$-complexes. 
Over the vertex point $v \in \sigma$ the graph $\mathbf G_v$ consists of a subset of rays of $\Sigma$.  If $\delta$ is not one of these rays, then there are no unbounded rays 
in the direction $\delta$ for $1$-complexes in $\sigma$, and $\mathsf{ev}_\delta$ sends $\sigma$ to the empty complex.

Otherwise, the ray $\delta \subset \mathbf G_v \subset \mathbf G_\sigma$ gives us a diagram of pairs
\[
\begin{tikzcd}
\mathbf (G_{\sigma}, \delta) \arrow{d}\arrow{r} & (\Sigma, \delta)\\
(\sigma, v) & 
\end{tikzcd}
\]
and by passing to stars, we get a family of $0$-complexes in $\Sigma_i(\delta)$ paramterized by $\sigma(v) = \sigma$.  Again, gluing over cones defines the map 
$\mathsf{ev}_\delta$.
\end{construction}

If $\Sigma$ has rays $\delta_1, \dots, \delta_r$, we can take the product of these maps to obtain
\[
|T^{\circ}_{\beta}| \rightarrow \prod |P(\Sigma_i)|.
\]
We can choose cone structures to obtain a map of cone complexes
\[
\mathsf{ev}: T^{\circ}_{\beta} \rightarrow \left(\prod |P(\Sigma_i)|\right)^{\dagger}.
\]
As we will discuss later, we can arrange for this map to be combinatorially flat.

\subsection{Geometric evaluations}\label{sec: geometric-evaluations} Let $(X,D)$ be a threefold pair as above, with divisor components $D_1,\ldots, D_r$. Given an algebraically transverse subscheme
\[
Z\hookrightarrow \cX\to X.
\]
Fix a divisor $D_i$. As we have noted in Section~\ref{sec: expanded-maps}, there is an induced expansion $\cD_i$. The algebraic transversality guarantees that we get a $0$-dimensional subscheme of $\cD_i$ that intersects every component of $\cD_i$. 

If $D_i$ is dual to the ray $\delta_i$ in $\Sigma$, and we write $\Sigma_i$ for the star fan $\Sigma(\delta_i)$, we constructed the evaluation map

\[
\mathsf{ev}: T^{\circ}_{\beta} \rightarrow \left(\prod |P(\Sigma_i)|\right)^{\dagger}
\]
in the previous section. There is an induced map on stacks of expansions:
\[
\mathsf{Exp}^\circ(X)\to \mathsf{Exp}_0(D),
\]
where we have restricted to the substack of the stack of expansions corresponding to $T^\circ(\Sigma)$ on the left. This operation above gives rise to a map, lifting the above one, encoded by a commutative diagram:
\[
\begin{tikzcd}
\mathsf{DT}_{\beta,\chi}(X)\arrow{d}\arrow{r} & \mathsf{Hilb}^{\mathbf{n}}(D)\arrow{d}\\
\mathsf{Exp}(X)\arrow{r}& \mathsf{Exp}_0(D),
\end{tikzcd}
\]
where $\mathbf{n} = (n_1,\dots, n_r)$ and $n_i$ is the intersection number $\beta\cdot D_i$.

Similarly,
on the GW side, if we 
fix an $r$-tuple of partitions by $\bm{\mu} = (\mu_1, \dots, \mu_r)$ 
of $\mathbf{n}$, 
then we can also lift the map on stacks of expansions to a diagram

\[
\begin{tikzcd}
\mathsf{GW}_{\beta,g}(X)_{\bm{\mu}}\arrow{d}\arrow{r} & \mathsf{Ev}^{\bm{\mu}}(D)\arrow{d}\\
\mathsf{Exp}(X)\arrow{r}& \mathsf{Exp}_0(D).
\end{tikzcd}
\]

\subsection{Pure complexes and uniform flattening}\label{sec: pure-complexes}

We return to the tropical evaluation maps from Section \ref{tropical-evaluation}.
 Our arguments will require us to choose cone structures such that these maps
\[
\mathsf{ev}: |T^\circ(\Sigma)|\to \prod_i |P(\Sigma(\delta_i))|
\]
are combinatorially flat. If one fixes all the discrete data, this can be done in a straightforward way, see Proposition~\ref{prop: strong-ss-reduction}. However, in GW/DT theory one typically forms a generating function by fixing the curve class $\beta$ and allowing the genus $g$ or Euler characteristic $\chi$ take all possible values. In order to state our degeneration setup in terms of generating functions, we need a stronger, {\it uniform flattening} statement. 

The following is the main result of this section.
\begin{theorem}\label{prop:flat-evaluation}
There exist cone structures on $T^{\circ}_{\beta}(\Sigma)$ and $(\prod P(\delta_i))^{\dagger}$ so that the evaluation map

\[
T^{\circ}_{\beta} \rightarrow (\prod P(\delta_i))^{\dagger}
\]
is well-defined and combinatorially flat with reduced fibers.

\end{theorem}

The reason for such a flattening to exist is that, if we fix $\beta$, then the space of possible $1$-complexes has infinitely many cones, but only due to adding floating points or subdividing edges, which 
don't contribute to the evaluation maps.  To make this precise, we make the following definitions.

We say that a $2$-valent vertex of a $1$-complex is \textit{linear $2$-valent} if its two adjacent edges/rays lie in the same cone and have the same direction.  We say that a vertex is \textit{free} if it has no adjacent edges or rays.  In each case, we can ``erase'' the vertex by coarsening the polyhedral structure on the $1$-complex. The linear condition means that the result is another $1$-complex. 

\begin{definition}
A $1$-complex is pure if it has no free or linear $2$-valent vertices.
\end{definition}

We have the subspace $|T^{\mathsf{pure}}| \subset |T|$ parametrizing pure $1$-complexes and a retraction map 
$$\rho:  |T| \rightarrow |T^{\mathsf{pure}}|$$
obtained by taking a $1$-complex and erasing all free and linear $2$-valent vertices.  The inclusion in the other direction is a section of $\rho$.

Note that by construction the evaluation map from $|T^\circ|$ factors through $\rho$:
$$|T^\circ| \rightarrow  |T^{\mathsf{pure}}| \rightarrow |\prod P(\delta_i)|.$$

\begin{lemma}
There exists a finite subcomplex $|T^{\mathsf{pure}}_{\sf bd}| \subset |T^{\mathsf{pure}}|$
such that the composition $|T_\beta| \rightarrow |T|  \rightarrow |T^{\mathsf{pure}}|$ factors through a map
$$ |T_\beta| \rightarrow |T^{\mathsf{pure}}_{\sf bd}|.$$
\end{lemma}

\begin{proof}
This is proven as part of the boundedness results in~\cite[Section~4.4]{MR20}.
\end{proof}

\begin{proposition}\label{prop: uniform-flattening}
Given a cone structure on $|T^{\mathsf{pure}}_{\sf bd}|$, after possibly passing to a subdivision, there exists a cone complex structure
on $|T_\beta|$ such that the map
\[
T_\beta \rightarrow T^{\mathsf{pure}}_{\sf bd}
\]
is combinatorially flat with reduced fibers.
\end{proposition}

\begin{proof}

Since the difference between a $1$-complex and its pure part is completely determined by extra vertices,
the map
\[
\rho:  |T| \rightarrow |T^{\mathsf{pure}}|
\]
can be realized as the 
relative tropical moduli space of $0$-complexes on $|\Sigma|\times  |T^{\mathsf{pure}}|$ over $ |T^{\mathsf{pure}}|$.
This description allows us to lift a cone structure from $|T^{\mathsf{pure}}|$ by following the construction of the tropical space of $0$-complexes, adapting it to a relative setting, and checking it preserves flatness.

More precisely, let $G \rightarrow T^{\mathsf{pure}}_{\sf bd}$ be a universal family over our specified bounded subcomplex.  We can extend $G$ to a complete subdivision 
\[
\Sigma_T \rightarrow \Sigma \times T^{\mathsf{pure}}_{\sf bd}
\]
so that $G \subset \Sigma_T$ is a subcomplex.  Furthermore, after possibly subdividing the base, we can assume that
$$\Sigma_T \rightarrow T^{\mathsf{pure}}_{\sf bd}$$
is combinatorially flat.

For each cone $\tau$ of $T^{\mathsf{pure}}_{\sf bd}$, we will construct the cone structure on $\rho^{-1}(\tau)$ as follows.  Let $\Sigma_\tau \subset \Sigma_T$ denote the preimage of $\tau$.
For each $n > 0$ , consider the relative fiber product
\[
\Sigma^n_{\tau}: = \Sigma_\tau \times_{\tau} \Sigma_\tau \dots \times_{\tau} \Sigma_{\tau}.
\]
This is a cone complex, combinatorially flat over $\tau$, since $\Sigma_\tau$ is such.

We will take the disjoint union over $n$ of $\Sigma^n_\tau$ and quotient by the equivalence relation associated to adding vertices of $G_\tau$, forgetting multiplicities, and permuting the ordering.
In order to put a cone complex structure on this quotient, we first pass to a subdivision.

Given a combinatorially flat map of cones $f: \sigma \rightarrow \tau$, we define the relative barycentric subdivison $\tilde{\sigma} \rightarrow \sigma$ to be the subdivision obtained by taking the barycentric subdivision of each polyhedral fiber 
of $f$.  By the flatness assumption, this defines a cone complex on $|\sigma|$ and each cone surjects onto a face of $\tau$.  Given a combinatorially flat map of cone complexes $C \rightarrow D$, this construction glues along faces to produce 
a subdivision $\widetilde{C} \rightarrow C$ which is still flat over $D$.

In our setting, we take the relative barycentric subdivision $\widetilde{\Sigma^n_{\tau}}$, which is still flat over $\tau$.   For each vertex $v$ of $|G|_t$, where $t$ is a point in the interior of $\tau$, we have a section
$\tau \rightarrow \Sigma_\tau$.  By adding $v$ to an $n$-tuple of points at position $i$, this defines inclusions of complexes
\[
f_{v,i}: \widetilde{\Sigma^n_{\tau}} \rightarrow \widetilde{\Sigma^{n+1}_{\tau}}.
\]
Similarly, we have various diagonal inclusions
\[
g_{i,j}: \widetilde{\Sigma^n_{\tau}} \rightarrow \widetilde{\Sigma^{n+1}_{\tau}}
\]
and an action of permutations $p \in S_n$
\[
p: \widetilde{\Sigma^n_{\tau}} \rightarrow \widetilde{\Sigma^{n}_{\tau}}.
\]

If we consider the equivalence relation on $\coprod \widetilde{\Sigma^n_{\tau}}$ induced by these inclusions and permutations,the quotient is set-theoretically given by $\rho^{-1}(\tau)$.
Because of the barycentric subdivision,  there is a natural cone complex structure on this quotient. Furthermore, since every cone of the prequotient surjects onto a face of $\tau$, the same is true afterwards.
The resulting construction is compatible with passing to faces of $\tau$, so the result is the desired cone structure on $\rho^{-1}(T^{\mathsf{pure}}_{\sf bd})$, flat with respect to $\rho$.
\end{proof}

Finally, we conclude with the proof of Theorem~\ref{prop:flat-evaluation}.

\begin{proof}

Choose cone structures so that we have a flat map
$$T^{\mathsf{pure}}_{\sf bd} \rightarrow (\prod P(\delta_i))^{\dagger}.$$
After further subdivision of both of these, we can further assume by the previous proposition that the map
$$T^{\circ}_{\beta}(\Sigma) \rightarrow T^{\mathsf{pure}}_{\sf bd}$$ is flat also, so by composition we are done. The claim about reduced fibers is similar. 

\end{proof}

\section{The cohomology of the Hilbert scheme of points}\label{sec: Hilb-cohomology}

The Hilbert scheme of $n$ points on a smooth and projective surface (without logarithmic structure) is a smooth and projective variety, and there is a simple description of the cohomology groups in terms of the famous Grojnowski--Nakajima basis~\cite{Groj96,Nak97}. The basis is a crucial input into the curve/sheaf correspondence for threefold smooth pairs in~\cite{MNOP06b}. 

The goal of this section is to extend this correspondence to surfaces with a simple normal crossings divisor, with the view towards the fully logarithmic curve/sheaf correspondence. We do so by studying the diagram from Section~\ref{sec: points}:
\[
\begin{tikzcd}
\mathsf{Hilb}^{\mathbf{n}}(D)\arrow{dr} & & \mathsf{Ev}^\mathbf{n}(D)\arrow{dl}\\
& \mathsf{Sym}^{\mathbf{n}}(D)\arrow{d}\\ &\mathsf{Exp}_0(D).
\end{tikzcd}
\]
As in the previous section, $D$ is the boundary divisor in our threefold $X$ with components $D_i$, and the space the space $\mathsf{Exp}_0(D)$ is a modification of the product of $\mathsf{Exp}(D_i)$:
\[
\sigma: \mathsf{Exp}_0(D)\to \prod_i \mathsf{Exp}_0(D_i).
\]
Throughout the discussion in this section, the choice of $\mathsf{Exp}(D)$ is fixed, and the results will hold for any fixed choice.

\subsection{Cohomology basis via partitions}

Choose an $r$-tuple $\mathbf{n}$ of non-negative integers.
If we also fix an $r$-tuple of partitions $\bm{\mu}$ of $\mathbf{n}$, there is a corresponding space of logarithmic stable maps
where the marked points map to the boundary with multiplicities indexed by $\bm{\mu}$. Although the mapping stack itself does not appear in this section, its evaluation space does.

The evaluation space $\mathsf{Ev}$ and the Hilbert scheme $\mathsf{Hilb}$ are the natural targets for boundary maps in GW theory and DT theory respectively. The main result of this section asserts that, after passing to an appropriate subspace of invariants under a group action, the cohomology groups of these two spaces coincide.

\begin{theorem}\label{thm: Nakajima-basis}
There exists a natural isomorphism
\[
\Gamma= \oplus_{\bm{\mu}} \Gamma_{\bm{\mu}}: \bigoplus H^\star(\mathsf{Sym}^{\bm{\mu}}(D)) \rightarrow H^\star( \mathsf{Hilb}^{\mathbf{n}}(D))
\]
that is compatible with log modification of $\mathsf{Exp}(D)$.  
In particular, if we pass to the direct limit, we have a corresponding isomorphism of log cohomology spaces.
\end{theorem}
Here, the summation is over $r$-tuples of partitions 
$\bm{\mu} = (\mu_1, \dots, \mu_r)$.

Thus far, our constructions have depended on a fixed subdivision $\sigma$ of the product of the stacks of expansions of $D_i$, ranging over all $i$. In order to obtain something that is independent of the choice of stack of expansions, we can consider the logarithmic cohomology of these spaces. These are obtained by taking the direct limit with respect to all modifications of the stack of expansions.

It will be convenient to pass directly from $\mathsf{Ev}^{\bm{\mu}}(D)$ to $\mathsf{Hilb}^{\mathbf{n}}(D)$, so we set
\[
\widetilde{\Gamma_{\bm{\mu}}} = \Gamma_{\bm{\mu}}\circ \rho_\star:  
H^\star(\mathsf{Ev}^{\bm{\mu}}(D)) \rightarrow H^\star( \mathsf{Hilb}^{\mathbf{n}}(D))
\]
where $\rho: \mathsf{Ev}(D) \rightarrow \mathsf{Sym}(D)$ is the finite quotient map.

Our proof of this theorem builds on the classical study of the cohomology groups of the Hilbert scheme of points using the decomposition theorem, see work of G\"ottsche--S\"orgel~\cite{GS93} and de Cataldo--Migliorini~\cite{dCM02,dCM04}.

\subsubsection{Semismall morphisms} We recall some basic facts about semismall maps and the decomposition theorem~\cite{BBD82}. See also~\cite{dCM04}. 
\begin{definition}
Given a proper morphism $f: Y \rightarrow Z$ with $Z$ irreducible, we say $f$ is \textit{semismall} with respect to a stratification of $Z$ if,
for every stratum $W$ and every point $z \in W$, we have
$$2\cdot\mathsf{dim}(f^{-1}(z)) \leq \mathsf{dim}(Z) - \mathsf{dim}(W).$$
A stratum is \textit{relevant} if equality holds in the above inequality.  We say $f$ is semismall if there exists a stratification as above.  In that case, the set of closures of relevant strata is independent of the choice of stratification.  

\end{definition}

Suppose that for each relevant stratum $W$ and $z \in W$, there is a unique irreducible component $F_z$ in $f^{-1}(z)$ where this equality is realized.  We will always be in this situation.
Given a relevant stratum $W$, we denote by $F_W$ the closure of this maximal irreducible component of the fibers over $W$.
The fundamental class of $F_W$ defines a correspondence between $\overline{W}$ and $Y$.
As we vary over relevant strata, a result of de Cataldo and Migliorini states that these correspondences define a decomposition of the cohomology of $Y$ as a direct sum of the cohomologies of the relevant strata closures, up to shifts:
\[
H^\star(Y) = \bigoplus H^\star(\overline{W})\left[2\cdot\mathsf{codim}(F_W,Y)\right].
\]

\begin{remark}
A special case of this result occurs for a smooth surface $S$, and the Hilbert--Chow morphism 
$$\mathsf{Hilb}^n(S) \rightarrow \mathsf{Sym}^n(S),$$ 
where the relevant strata
are precisely the strata $\mathsf{Sym}^\mu(S)$ associated to partitions of $n$.  In this case, the above decomposition recovers the Nakajima decomposition via cohomology-weighted partitions~\cite{dCM02}.
\end{remark}

Suppose we have a smooth family of surfaces $\pi: S \rightarrow B$, such that $B$ and the geometric generic fiber of $\pi$ are irreducible.
We again consider the relative Hilbert scheme of points, along with its Hilbert--Chow morphism
$$\pi^{[n]}:  \mathsf{Hilb}^{n}(S/B) \rightarrow \mathsf{Sym}^n(S/B).$$
Fix a base point $0 \in B$ and a zero cycle $z$ of length $n$ in $S_0$, which defines a point $[z] \in \mathsf{Sym}^n(S_0)\subset \mathsf{Sym}^n(S/B)$.
\begin{lemma}
After restriction to a formal (or analytic) neighborhood of $[z]$, the morphism $\pi^{[n]}$ is isomorphic to the restriction of 
$$\mathsf{Hilb}^{n}(S_0)\times B \rightarrow \mathsf{Sym}^n(S_0)\times B$$
to a formal neighborhood of the corresponding point $[z]\times 0$.
\end{lemma}
\begin{proof}
Since $\pi$ is smooth, the formal neighborhood of any point $x \in S_0$ inside $S$ is isomorphic to the product of the formal neighborhood of $x$ inside $S_0$ and $0 \in B$; a similar product structure holds for any finite subset of $S_0$.  By applying this to the support of the cycle $z$ inside $S_0$, the claim follows.
\end{proof}

Using this lemma, we have the following corollary.

\begin{corollary}\label{semismall-families}
The morphism $\pi^{[n]}$ has the following properties:
\begin{enumerate}[(i)]
\item The morphism is semismall.
\item The relevant strata closures are the irreducible subvarieties $\mathsf{Sym}^{\mu}(S/B)$.
\item Associated to each relevant stratum above, we have a correspondence $F^\mu(S/B)$ obtained by taking the closure of its preimage.
Both the relevant strata closures $\mathsf{Sym}^{\mu}(S/B)$ and the associated correspondence $F^\mu(S/B)$ defined above are flat over $B$
\item The varieties $\mathsf{Sym}^{\mu}(S/B)$ and $F^\mu(S/B)$ are preserved under base change with respect to morphisms $B' \rightarrow B$.
\end{enumerate}
\end{corollary}

\begin{proof}
Most of the statements follow directly from the lemma and the local product structure which, in particular implies that the closure of a stratum, when restricted to a fiber over a point in $B$, is the same as the closure of the restriction.  For the irreducibility of $\mathsf{Sym}^{\mu}(S/B)$, this follows from the irreducibility of the generic fiber of $\pi$.
\end{proof}

\subsubsection{Proof of Theorem~\ref{thm: Nakajima-basis}}

We apply the results of previous section to $\mathsf{Hilb}^{\mathbf{n}}(D)$.
It suffices to establish the following claims:
\begin{enumerate}[(i)]

\item The morphism
\[
\mathsf{Hilb}^{\mathbf{n}}(D) \rightarrow \mathsf{Sym}^{\mathbf{n}}(D)
\] 
is semismall.
\item The relevant strata closures are the irreducible subvarieties $\mathsf{Sym}^{\bm{\mu}}(D)$ 
\item The fiber over the generic point of each relevant stratum has a unique irreducible component of maximal dimension.

\item If $F^{\bm{\mu}}(D)$ is the closure of the preimage of this relevant stratum in $\mathsf{Hilb}^{\mathbf{n}}(D)$, both 
$$\mathsf{Sym}^{\bm{\mu}}(D)\ \   \textnormal{and}\ \ F^{\bm{\mu}}(D),$$ are flat over $\mathsf{Exp}(D)$ and are preserved by logarithmic modification.

\end{enumerate}

The first three items, combined with \cite{dCM02}, implies the claimed decomposition of cohomology.  The morphisms $\Gamma_{\bm{\mu}}$ are defined using the correspondences
$[F^{\bm{\mu}}(D)]$.  By the last item, since the correspondences are preserved under base change, this implies the decomposition is compatible with
log modification.

To show these claims, we first consider the case of a single boundary component $(E, \partial E)$.  All items can be checked after passing to a smooth cover of
$\mathsf{Exp}_0(E)$, in which case they follow from Corollary \ref{semismall-families}.
Since all claims are preserved under products, we then deduce the case of
\[
\prod \mathsf{Hilb}^{\mathbf{n}}(D_i) \rightarrow \prod \mathsf{Sym}^{\mathbf{n}}(D_i)
\]
over
$\prod \mathsf{Exp}_0(D_i)$.
Finally, again by Corollary \ref{semismall-families}, everything is preserved under base change, 
so we can prove the statement for an arbitrary 
logarithmic modification as well.

\qed

\subsection{Diagonal matching} 

We show that the correspondences $\Gamma_{\bm{\mu}}$ from Theorem \ref{thm: Nakajima-basis} are compatible with diagonal self-correspondences. In the traditional situation, without logarithmic structure, this is used to show compatibility with the Poincar\'{e} pairing, and this is in turn needed for compatibility of the curve/sheaf correspondence with degeneration~\cite{MNOP06b}.

We need it for the same reasons, but also require something more: 
compatibility with {\it strict transforms} of the diagonal under blowups of the product. Precisely, we study the diagonal $\Delta$ in a product of the form 
$\mathsf{Hilb}^{\mathbf{n}}(D)  \times \mathsf{Hilb}^{\mathbf{n}}(D)$ and analyze how $\Gamma$ interacts with the strict transform of $\Delta$ in log modifications of this product.

\begin{remark}
In general, there are correction terms in calculating the action of correspondences on strict transforms, see for instance~\cite{MR21}. However in the present setting, for our particular correspondence, no such terms arise. Roughly, one can view this as a consequence of the transversality of two different stratifications -- one coming from the logarithmic structure and the other coming from the stratification on the symmetric product by multiplicity type. We explain this rigorously below.
\end{remark}

As in the last section, let $D_1, \dots D_r$ be a collection of surfaces with boundary.\footnote{In practice, and in this paper, these may arise from a threefold as earlier, or may instead be the components of the double locus of the special fiber of a degeneration of a threefold.} Fix an $r$-tuple of non-negative integers $\mathbf{n} = (n_1, \dots, n_r)$.  Consider
a tropical model
\begin{equation}\label{eqn: H-blowup}
    \left(\mathsf{Hilb}^{\mathbf{n}}(D)  \times \mathsf{Hilb}^{\mathbf{n}}(D) \right)^{\dagger} \rightarrow \mathsf{Hilb}^{\mathbf{n}}(D)  \times \mathsf{Hilb}^{\mathbf{n}}(D)
\end{equation}
and denote
$$\Delta_H^{\dagger} \hookrightarrow \left(\mathsf{Hilb}^{\mathbf{n}}(D)  \times \mathsf{Hilb}^{\mathbf{n}}(D) \right)^{\dagger}$$
the strict transform of the diagonal.

Given a pair of $r$-tuple of partitions of $\mathbf{n}$, $\bm{\mu} = (\mu_1, \dots, \mu_r)$ and $\bm{\nu} = (\nu_1, \dots, \nu_r)$, we can take the 
corresponding modification of the relevant strata
$$
\left({\mathsf{Sym}^{\bm{\mu}}(D)} \times \mathsf{Sym}^{\bm{\nu}}(D) \right)^{\dagger}\rightarrow {\mathsf{Sym}^{\bm{\mu}}}(D) \times \mathsf{Sym}^{\bm{\nu}}(D).
$$
If $\bm{\mu} = \bm{\nu}$, we let $\Delta_{\bm{\mu}}^{\dagger}$ denote the strict transform of the diagonal of the right-hand side.

There is a corresponding log modification of the cycles
$$
\left(F^{\bm{\mu}}(D) \times F^{\bm{\nu}}(D)\right)^{\dagger} \rightarrow\left(\mathsf{Sym}^{\bm{\mu}}(D) \times \mathsf{Sym}^{\bm{\nu}}(D)\right)^{\dagger}.
$$
These define transpose operators on rational cohomology
$$\Gamma^{\mathsf{tr}}_{\bm{\mu},\bm{\nu}}:  H^\star\left(\left(\mathsf{Hilb}^{\mathbf{n}}(D)  \times \mathsf{Hilb}^{\mathbf{n}}(D) \right)^{\dagger}\right)
\rightarrow H^\star(\left(\mathsf{Sym}^{\bm{\mu}}(D) \times \mathsf{Sym}^{\bm{\nu}}(D)\right)^{\dagger}).$$

Given a partition $\mu = (\mu^{(1)}, \dots, \mu^{(\ell)})$ of $n$, we set
$$m_\mu =  \prod \mu^{(i)};$$
given an $r$-tuple of partitions $\bm{\mu}$, we set $m_{\bm{\mu}} = \prod_{i=1}^{r} m_{\mu_i}$.  Similarly, we set
$$(-1)^{\bm{\mu}}: = \prod_{i=1}^{r} (-1)^{n_i - \ell(\mu_i)}.$$

\begin{proposition}

If $\bm{\mu}= \bm{\nu}$, then 
\[
\Gamma^{\mathsf{tr}}_{\bm{\mu},\bm{\nu}}\left([\Delta_H^{\dagger}]\right) = (-1)^{\bm{\mu}}m_{\bm{\mu}} \cdot\left[\Delta_{\bm{\mu}}^{\dagger}\right];
\]
If $\bm{\mu} \neq \bm{\nu}$, then $\Gamma^{\mathsf{tr}}_{\bm{\mu},\bm{\nu}}([\Delta_H^{\dagger}]) = 0.$

\end{proposition}

\begin{proof}

The diagonal in the self-product $\mathsf{Hilb}^{\mathbf{n}}(D)\times \mathsf{Hilb}^{\mathbf{n}}(D)$ is identified with a copy of $\mathsf{Hilb}^{\mathbf{n}}(D)$. We have fixed a modification of this product. Choose a tropical model 
$$
\mathsf{Hilb}^{\mathbf{n}}(D)' \rightarrow \mathsf{Hilb}^{\mathbf{n}}(D)
$$ 
that maps onto the diagonal $\Delta_H^{\dagger}$. 


Let $\mathsf{Sym}^{\bm{\mu}}(D)'$ and $F^{\bm{\mu}'}$ denote the corresponding modifications of symmetric products relative to boundary; by construction, $\mathsf{Sym}^{\bm{\mu}}(D)'$ 
surjects onto the diagonal cycle $\Delta_{\bm{\mu}}$. 
For each choice of $\bm{\mu}$ and $\bm{\nu}$, we have a diagram
\[
\begin{tikzcd}
  F^{\bm{\mu}'} \cap F^{\bm{\nu}'}\arrow{d}\arrow{r} & \left(F^{\bm{\mu}} \times F^{\bm{\nu}}\right)^{\dagger}\arrow{d}\arrow{r} & \left(\mathsf{Sym}^{\bm{\mu}}(D) \times \mathsf{Sym}^{\bm{\nu}}(D)\right)^{\dagger}\\
 \mathsf{Hilb}^{\mathbf{n}}(D)'\arrow{r} &  \left(\mathsf{Hilb}^{\mathbf{n}}(D)  \times \mathsf{Hilb}^{\mathbf{n}}(D) \right)^{\dagger}&.
\end{tikzcd}
\]

We claim the commutative square is Cartesian.  This follows because the both vertical arrows are obtained by pullback of the product of the inclusions
$$F^{\bm{\mu}}\times F^{\bm{\nu}} \rightarrow \mathsf{Hilb}^{\mathbf{n}}(D) \times \mathsf{Hilb}^{\mathbf{n}}(D).$$
Also the composition of the two horizontal arrows of the top row factors through the intersection
$\mathsf{Sym}^{\bm{\mu}}(D)' \cap \mathsf{Sym}^{\bm{\nu}}(D)' $.

This intersection is a union of strata associated to $r$-tuples $\bm{\lambda}$ of partitions, obtained by combining parts of $\bm{\mu}$ and also by combining parts of $\bm{\nu}$.  Since length can only decrease when combining parts, we have
$\ell(\bm{\lambda}) \leq \ell(\bm{\mu}), \ell(\bm{\nu})$
and at least one of these inequalities must be strict if ${\bm{\mu}} \neq {\bm{\nu}}$.

The cycle class $\Gamma^{\mathsf{tr}}_{\bm{\mu},\bm{\nu}}([\Delta_H]^{\dagger})$ has (complex) homological degree
$$
\mathsf{deg}\left(\Gamma^{\mathsf{tr}}_{\bm{\mu},\bm{\nu}}([\Delta_H]^{\dagger})\right) = \ell(\bm{\mu}) + \ell(\bm{\nu}) \geq 2 \ell(\bm{\lambda}).
$$
If ${\bm{\mu}} \neq {\bm{\nu}}$ then this class vanishes since it is supported on a variety of smaller dimension.  
If $\bm{\mu} = \bm{\nu}$, then again by dimension-counting,  $\Gamma^{\mathsf{tr}}_{\bm{\mu},\bm{\nu}}([\Delta_H]^{\dagger})$ is a multiple of $[\Delta_{\bm{\mu}}]^{\dagger}$.

Finally to determine the multiple, one can restrict to the locus where the surfaces are all unexpanded.  In this case, the calculation is done in 
\cite[Proposition~5.1.4]{dCM02}.

\end{proof}

The following two corollaries are immediate consequences of this proposition.

\begin{corollary}
The inverse to $\Gamma = \oplus_{\bm{\mu}} \Gamma_{\bm{\mu}}$
is given by the operator
\[ \bigoplus_{\bm{\mu}} \frac{(-1)^{\bm{\mu}}}{m_{\bm{\mu}}}\Gamma_{\bm{\mu}}^{\mathsf{tr}}.
\]
\end{corollary}

In what follows, let 
\[
\Delta_{\mathsf{Ev}^{\bm{\mu}}}^{\dagger} \hookrightarrow
\left(\mathsf{Ev}^{\bm{\mu}}(D)  \times \mathsf{Ev}^{\bm{\mu}}(D) \right)^{\dagger}
\]
denote the strict transform of the diagonal for the evaluation spaces $\mathsf{Ev}^{\bm{\mu}}(D)$. This diagonal will later play the key role in our degeneration formula in both curve and sheaf theories. The two will be compared by the following:

\begin{corollary}\label{cor: diagonal-formula}
The cycle class of $\Delta_{H}^{\dagger}$ is given by the formula
\begin{align*}
[\Delta_H^{\dagger}] &= \sum_{\bm{\mu}} 
\frac{(-1)^{\bm{\mu}}}{m_{\bm{\mu}}}
\Gamma_{\bm{\mu},\bm{\mu}}
[\Delta_{\bm{\mu}}^{\dagger}] \\
&= 
\sum_{\bm{\mu}} \frac{(-1)^{\bm{\mu}}}{|\mathrm{Aut}(\bm{\mu})|m_{\bm{\mu}}}
\widetilde{\Gamma_{\bm{\mu},\bm{\mu}}}
[\Delta_{\mathsf{Ev}^{\bm{\mu}}}^{\dagger}].
\end{align*}
\end{corollary}

\section{The curve/sheaf conjectures}\label{sec:curve/sheaf-conjectures}

Equipped with the logarithmic Nakajima basis of Theorem~\ref{thm: Nakajima-basis}, we now state the logarithmic curve/sheaf conjectures. Schematically, the conjectures fill in the top arrow below:
\[
\begin{tikzcd}
    \textsf{GW -- maps from curves}\arrow[dash,dashed]{rr}\arrow[swap,dash]{d}{\color{ForestGreen}\mathsf{evaluation}}& &\textsf{DT -- scheme theoretic curves}\arrow[dash]{d}{\color{ForestGreen}\mathsf{evaluation}}\\
        \textsf{{Ev} -- maps from points}\arrow[swap,dash]{rr}{\color{ForestGreen}\textsf{log Nakajima}}& &\textsf{Hilb -- scheme theoretic points}
\end{tikzcd}
\]
The cohomology classes on the bottom row impose tangency conditions and boundary incidence conditions on the moduli spaces in the top row, and by Theorem~\ref{thm: Nakajima-basis} they can be matched. 

\subsection{Evaluation maps} Let us begin by recalling the evaluation maps. 

Let $X$ be our threefold, with simple normal crossings boundary $D$ and divisor components $D_1, \dots, D_r$.  
Fix a curve class $\beta \in H_2(X, \ZZ)$ and let ${\bf n}$ denote the collection of intersection multiplicities $n_i = D_i \cdot \beta$.
We fix contact orders by an $r$-tuple of partitions $\bm{\mu} = (\mu_1, \dots, \mu_r)$, where $\mu_i$ is a partition of $n_i$.

We first fix cone structures on the tropical moduli spaces so that we have evaluation maps
\[
T_{\beta}(\Sigma) \rightarrow (\prod P(\delta_i))^{\dagger}.
\]

After passing to geometric moduli spaces, we have, for each Euler characteristic $\chi$,
evaluation maps
\[
\mathsf{ev}\colon \mathsf{DT}_{\beta,\chi}(X)\to \mathsf{Hilb}^{\mathbf n}(D).
\]
and a parallel map for PT moduli spaces. 
On the Gromov-Witten side, for each genus $g$,
 we have maps:
\[
\mathsf{ev}_{\bm \mu} \colon \mathsf{GW}_{\beta,g}(X)_{\bm{\mu},k}\to \mathsf{Ev}^{\bm \mu}(D).
\]
In addition, we have, for each $i = 1,\ldots, k$ internal evaluation maps:
\[
\mathsf{ev}_i\colon \mathsf{GW}_{\beta,g}(X)_{\bm{\mu},k}\to X.
\]

\subsection{Generating functions}

For internal insertions, we specify 
cohomology classes
$$\zeta_1, \dots, \zeta_k \in H^\star(X)$$
as well as non-negative integers $a_1, \dots, a_k$ for descendants. 
For relative conditions, we use the map:
\[
\widetilde{\Gamma}= \oplus \widetilde{\Gamma_{\bm{\mu}}}\colon \bigoplus H_{\mathsf{log}}^\star(\mathsf{Ev}^{\bm{\mu}}(D)) \rightarrow H_{\mathsf{log}}^\star( \mathsf{Hilb}^{\mathbf{n}}(D)),
\]
obtained from Theorem~\ref{thm: Nakajima-basis} by taking a colimit over all logarithmic modifications. 

If we fix a log cohomology class
$$\delta \in H_{\mathsf{log}}^\star(\mathsf{Ev}^{\bm{\mu}}(D)),
$$
we have an associated cohomology class 
$$[\bm{\mu}(\delta)] := \frac{1}{m_{\bm{\mu}}}\widetilde{\Gamma_{\bm{\mu}}}(\delta)$$
on the right.

\subsubsection{The GW partition function} We consider the moduli space $\mathsf{GW}(X)^{\bullet}_{g, \beta, \bm{\mu}, k}$
associated to stable maps to expansions of $(X,D)$, where we allow disconnected domains for which each connected component is non-contracted. As we have noted, there are evaluation maps:
$$
\mathsf{ev}_i\colon \mathsf{GW}(X)^{\bullet}_{g, \beta, \bm{\mu}, k} \rightarrow X \ \ \textnormal{and} \ \  \mathsf{ev}_{\bm{\mu}}\colon \mathsf{GW}(X)^{\bullet}_{g, \beta, \bm{\mu}, k} \rightarrow \mathsf{Ev}^{\bm{\mu}}(D).
$$

We assume we have chosen suitable 
refinements so that the log cohomology class $\delta$ arises from a cohomology class on the tropical model $\mathsf{Ev}^{\bm{\mu}}(D)$. This gives descendent log invariants
$$\langle \prod_{i=1}^{k} \tau_{a_{i}}(\zeta_i) | \bm{\mu}(\delta) \rangle^{\bullet, \mathsf{GW}}_{g,\beta} = 
\mathrm{deg}\left(
\prod_{i=1}^{k} \psi_{i}^{a_{i}}\mathsf{ev}_{i}^{\star}(\zeta_i)\cdot \mathsf{ev}_{\bm{\mu}}^{\star}(\delta)\cap [\mathsf{GW}(X)]^{\mathsf{vir}}
\right),$$
where $\mathrm{deg}$ denotes the degree of the pushforward to a point.  
By summing over the genus, we have the partition function
\[
\mathsf{Z}_{\mathsf{GW}}\left(X,D; u | \prod_{i=1}^{k} \tau_{a_{i}}(\zeta_i)|\bm{\mu}(\delta)\right)_{\beta} = \sum_{g} \langle \prod_{i=1}^{k} \tau_{a_{i}}(\gamma_i) | \bm{\mu}(\delta) \rangle^{\bullet, \mathsf{GW}}_{g,\beta} u^{2g-2},
\]
which is a Laurent series in $u$.

\subsubsection{The DT/PT partition functions}
On the sheaf theory side, we proceed as follows.  
For each $\zeta \in H^\star(X)$, we define descendent operators for each $a \geq 0$
\[
\tau_a(\zeta)\colon H_\star(\mathsf{DT}_{\beta,\chi}(X), \mathbb{Q}) \rightarrow H_\star(\mathsf{DT}_{\beta,\chi}(X), \mathbb{Q})
\]
via the formula
\[
\tau_a(\zeta)(\bullet) \coloneqq \pi_{\mathsf{DT},\star}(\mathrm{ch}_{a+2}(\mathcal{I})\cdot \pi_X^\star(\zeta)\cap \pi_{\mathsf{DT}}^{\star}(\bullet)).
\]
Similarly, we have descendent operators $\tau_a(\gamma)$ for $\mathsf{PT}_{\beta,\chi}(X)$.

 Using these and the relative conditions described earlier, we have descendent invariants
\[
\langle \tau_{a_{1}}(\zeta_1)\dots \tau_{a_{k}}(\zeta_k)| \bm{\mu}(\delta) \rangle^{\mathsf{DT}}_{\beta,\chi}
= \mathrm{deg}\left( \prod_{i=1}^{k}\tau_{a_{i}}(\zeta_i)\mathsf{ev}^\star[\bm{\mu}(\delta)]\cap[\mathsf{DT}]^{\mathsf{vir}}  \right)
\]
and the corresponding invariants for stable pairs. If we sum over the possible values of the Euler characteristic $\chi$, we have the partition function
\[
\mathsf{Z}_{\mathsf{DT}}\left(X,D;q | \prod_{i=1}^{k} \tau_{a_{i}}(\zeta_i)|\bm{\mu}(\delta)\right)_{\beta}
= \sum_{\chi} \langle \prod_{i=1}^{k} \tau_{a_{i}}(\zeta_i)| \bm{\mu}(\delta) \rangle_{\beta,\chi} \cdot {q^{\chi}},
\]
and similarly for 
$\mathsf{Z}_{\mathsf{PT}}\left(X,D;q | \prod_{i=1}^{k} \tau_{a_{i}}(\zeta_i)|\bm{\mu}(\delta)\right)_{\beta}.$
Each of these is a Laurent series since for each fixed curve class $\beta$, the logarithmic DT and PT spaces are empty for $\chi\ll0$.

\subsubsection{Boundedness of relative conditions}

As stated above, our invariants depend on a log cohomology class 
$\delta$ in $H_{\mathsf{log}}^\star(\mathsf{Ev}^{\bm{\mu}}(D))
$, which is an infinite dimensional space.
However, it turns out for fixed $\beta$, this dependence factors through a finite-dimensional space.  We do not use this anywhere, but we feel it is an important philosophical point. We sketch how this works on the Gromov-Witten side; the same argument applies in the other cases.

Using Theorem~\ref{prop:flat-evaluation}, assume we have chosen tropical models
so that the map 
\[
T^{\circ}_{\beta} \rightarrow (\prod P(\delta_i))^{\dagger}
\]
is combinatorially flat with reduced fibers.
We pass to geometric moduli spaces
\[
\mathsf{ev}: \mathsf{GW}_{\beta,g}(X)_{\bm{\mu}}
\rightarrow \mathsf{Ev}^{\bm{\mu}}(D).
\]
We claim that log Gromov-Witten invariants with relative insertions in the finite-dimensional space $H^\star(\mathsf{Ev}^{\bm{\mu}}(D))$ determine all log GW invariants.

Indeed, given any log cohomology class $\delta$, choose a log modification $\pi: \mathsf{Ev}^{\bm{\mu}}(D)' \rightarrow \mathsf{Ev}^{\bm{\mu}}(D)$ for which $\delta$
arises from a cohomology class on the modification. We can extend $\pi$  to a Cartesian diagram
\[
\begin{tikzcd}
\mathsf{GW}(X)_{\bm{\mu},k}'\arrow{d}\arrow{r} & \mathsf{Ev}^{\bm{\mu}}(D)'\arrow{d}\\
\mathsf{GW}(X)_{\bm{\mu},k}\arrow{r}& \mathsf{Ev}^{\bm{\mu}}(D).
\end{tikzcd}
\]
As explained in the proof of Theorem \ref{thm: main-splitting-theorem} in Section \ref{gluing-in-dt} in a more general case, combinatorial flatness of $\mathsf{ev}$ implies that
\[
\pi^\star\mathsf{ev}_\star\left(\prod_{i=1}^{k} \psi_{i}^{a_{i}}\mathsf{ev}_{i}^{\star}(\zeta_i)\cap [\mathsf{GW}(X)]^{\mathsf{vir}}\right) =
\mathsf{ev}'_\star\left(\prod_{i=1}^{k} \psi_{i}^{a_{i}}\mathsf{ev}_{i}^{\star}(\zeta_i)\cap [\mathsf{GW}(X)']^{\mathsf{vir}}\right).
\]
Therefore by a push-pull argument,
we have
\[
\langle \prod_{i=1}^{k} \tau_{a_{i}}(\zeta_i) | \bm{\mu}(\delta) \rangle^{\bullet, \mathsf{GW}}_{g,\beta}
=
\langle \prod_{i=1}^{k} \tau_{a_{i}}(\zeta_i) | \bm{\mu}(\pi_\star\delta) \rangle^{\bullet, \mathsf{GW}}_{g,\beta}
\]
so only the projection of $\delta$ onto 
$H^\star(\mathsf{Ev}^{\bm{\mu}}(D))$ contributes.  A similar statement applies for DT and PT invariants.

Since this flattening applies uniformly over all values of $g$, we can also think of the generating functions 
$\mathsf{Z}_{\mathsf{GW}}(u)$, $\mathsf{Z}_{\mathsf{DT}}(q),$ and
$\mathsf{Z}_{\mathsf{PT}}(q)$
as Laurent series depending on only finitely many relative conditions.

\subsection{Conjectures}
The first conjecture concerns punctual evaluation of the DT series, the relationship the DT and PT series, and the rationality of the PT generating function. They were first stated in~\cite{MR20}.

\begin{conjecture}\label{conj: rationality}
\begin{enumerate}
\item[(i)]
The DT series for zero-dimensional subschemes is given by
\[
\mathsf{Z}_{\mathsf{DT}}\left(X,D;q)\right)_{\beta=0} = M(-q)^{\int_{X} c_{3}(T^\mathrm{log}_X\otimes K^{\mathrm{log}}_{X})},
\]
where
\[
M(q) = \prod_{n\geq1} \frac{1}{(1-q^n)^n}
\]
is the McMahon function.
\item[(ii)]
For any curve class $\beta$, and cohomology classes $\zeta_1, \dots, \zeta_k$, and integers $a_1, \dots, a_k$, the PT series
$$\mathsf{Z}_{\mathsf{PT}}\left(X,D;q | \prod_{i=1}^{k} \tau_{a_{i}}(\zeta_i)|\bm{\mu}(\delta)\right)_{\beta}$$
is the Laurent expansion of a rational function in $q$.

\item[(iii)]
If the insertions $\zeta_1, \dots, \zeta_k$ each have degree $\geq 2$, the normalized DT series 
\[
\mathsf{Z}'_{\mathsf{DT}}\left(X,D;q | \prod_{i=1}^{k} \tau_{a_{i}}(\zeta_i)|\mu\right)_{\beta}
:=
\frac{\mathsf{Z}_{\mathsf{DT}}\left(X,D;q | \prod_{i=1}^{k} \tau_{a_{i}}(\zeta_i)|\mu\right)_{\beta}}{\mathsf{Z}_{\mathsf{DT}}\left(X,D;q)\right)_{\beta=0}}
\]
is the Laurent expansion of a rational function in $q$ and
$$\mathsf{Z}'_{\mathsf{DT}} = \mathsf{Z}_{\mathsf{PT}}.$$

\end{enumerate}
\end{conjecture}

We state the simplest form of the curve/sheaf conjectures -- the primary correspondence. Fix the degree
$d_\beta = c_1(T_X)\cdot \beta \in \mathbb{Z}$.

\begin{conjecture}\label{conj: primary}
Under the identification of variables $-q = e^{iu}$, we have the equality
\[
(-q)^{-d_{\beta}/2} \mathsf{Z}_{\mathsf{PT}}\left(X,D;q | \prod_{i=1}^{k} \tau_{0}(\zeta_i)|\bm{\mu}(\delta)\right)_{\beta}=
(-iu)^{d_{\beta} + \sum \ell(\mu_j) - |\mu_j|} \mathsf{Z}_{\mathsf{GW}}\left(X,D; u|\prod_{i=1}^{k} \tau_{0}(\zeta_i)|\bm{\mu}(\delta)\right)_{\beta}.
\]
\end{conjecture}
\subsection{Descendent conjectures}

We can also formulate the descendent correspondence, in parallel with the absolute and relative situations explained in \cite{PP14,PP12}. The key idea for this correspondence is a certain universal matrix $\widetilde{\mathsf{K}}$ indexed by partitions $\alpha$ and $\widehat{\alpha}$
\[
\widetilde{\mathsf{K}}_{\alpha, \widehat{\alpha}} \in \QQ[i, c_1, c_2, c_3]((u)).
\]
We will substitute $c_i = c_i(T^{\mathrm{log}}_X)$ the Chern classes of the logarithmic tangent bundle of $(X,D)$ so each matrix entry is a Laurent series with coefficients in $H^\star(X)$.

Before stating the correspondence, we require some further notation.
As in the definition of evaluation spaces, let
$(X|D)^{\ell}$
denote the moduli space of ordered points in expansions of $X$, associated to a choice of tropical moduli space of ordered points in $\Sigma_X$.

After choosing subdivisions compatibly, if we have $k$ internal marked points, the evaluation maps $\mathsf{ev}_i\colon\mathsf{GW}(X) \rightarrow X$ for these points can be refined to give an evaluation map
$$\mathsf{ev}: \mathsf{GW}(X) \rightarrow (X|D)^{k}.$$
Let $\Delta^{\mathrm{st}} \subset (X|D)^{k}$ denote the strict transform of the small diagonal $X \subset X^{k}$.  Given $\zeta \in H^\star(X)$, pullback to $\Delta^{\mathrm{st}}$ followed by pushforward defines a cohomology class
$\zeta \cdot \Delta \in H^\star((X|D)^{k}).$

Given a partition $\alpha$ of length $k$ we define the descendent insertion
\[
\tau_{\alpha}(\zeta) = \prod \psi_i^{\alpha_i - 1} \cdot \mathsf{ev}^\star(\zeta\cdot \Delta).
\]
Suppose we are given a partition $\alpha$ of length $k$, and cohomology classes $\zeta_1, \dots, \zeta_k$.  We will sum over all set partitions $P$ of the index set $\{1, \dots, k\}$.  Given such a partition $P$ with an element $S \in P$ corresponding to a subset of $\{1, \dots, k\}$, we set $\alpha_S$ to be the corresponding sub-partition of $\alpha$ and $\zeta_S := \prod_{i \in S} \zeta_i$.

We now define the corrected descendent insertion 
\begin{align*}
\overline{\tau_{\alpha_{1}-1}(\zeta_{1})\cdots \tau_{\alpha_{k}-1}(\zeta_{k})} &= 
\sum_{P} (-1)^{\sigma(P)} \prod_{S \in P} \sum_{\widehat{\alpha}} \tau_{\widehat{\alpha}}(\widetilde{\mathsf{K}}_{\alpha_S, \widehat{\alpha}}\cdot \zeta_{S}).
\end{align*}

Here $(-1)^{\sigma(P)}$ is a sign convention to handle odd cohomology classes $\zeta_i$, which is explained in~\cite[Section~0.4]{PP12}).

We now state the full descendent correspondence for PT theory using the corrected descendents.

\begin{conjecture}\label{conj: descendent}
Under the identification of variables $-q = e^{iu}$, we have the equality
\begin{align*}
 (-q)^{-d_{\beta}/2} &\mathsf{Z}_{\mathsf{PT}}\left(X,D;q | \tau_{\alpha_{1}-1}(\zeta_{1})\cdots \tau_{\alpha_{k}-1}(\zeta_{k})|\bm{\mu}(\delta)\right)_{\beta}=\\
&(-iu)^{d_{\beta} + \sum \ell(\mu_j) - |\mu_j|} \mathsf{Z}_{\mathsf{GW}}\left(X,D; u| \overline{\tau_{\alpha_{1}-1}(\zeta_{1})\cdots \tau_{\alpha_{k}-1}(\zeta_{k})}|\bm{\mu}(\delta)\right)_{\beta}.
\end{align*}

\end{conjecture}

\newpage

{\Large \part{The logarithmic degeneration package}\label{part: 2}}

In the first part, we have studied the DT, PT, and GW theory of a pair $(X|D)$.  We explain how to use this theory to study traditional enumerative invariants using sufficiently nice degenerations, and in particular, establish the compatibility of the logarithmic curve/sheaf conjectures with degeneration.

\subsection*{The degeneration package} We begin with a simple normal crossings degeneration over a pointed curve $(B,0)$\footnote{Our precise assumptions on $\mathcal Y\to B$ are described in the next section.} 
\[
\mathcal Y\to B,
\]
that is smooth away from the point $0\in B$; the fiber over $0$ is a normal crossings variety denoted $Y_0$. We consider the GW and DT spaces to the fibers of the family $\mathcal Y/B$. On the GW side, this fits into the formalism of~\cite{GS13} and is studied in detail in~\cite{ACGS15}. For the DT side, the necessary foundations are provided in the next section. We are led to moduli spaces:
\[
\begin{tikzcd}
    \coprod_\chi \mathsf{DT}_{\beta,\chi}(\mathcal Y/B) \arrow{dr} & & \coprod_g \mathsf{GW}_{\beta,g}(\mathcal Y/B) \arrow{dl}\\
    & B.&
\end{tikzcd}
\]

The following picture has emerged for logarithmic GW invariants in a degeneration, from work of Abramovich, Chen, Gross, Siebert, and the second author. Specifically, we have three aspects of the package. 

\begin{enumerate}[(i)]
    \item {\bf Deformation invariance.} The virtual cycles of any two fibers of the map
    \[
    \mathsf{GW}_{\beta,g}(\mathcal Y/B)\to B,
    \]
    are homologically equivalent in the total space, see~\cite{GS13}.
    \item {\bf Decomposition.} The virtual class of $\mathsf{GW}_{\beta,g}(Y_0)$ is a sum of virtual classes indexed by distributions of the discrete data. The discrete data are indexed by rigid tropical stable maps. The tropical stable maps account for the topological type of the curve as it breaks over the components of $Y_0$, including the curve class distributions and the contact orders with the strata, see~\cite{ACGS15}.
    \item {\bf Splitting formula.} Given a rigid tropical map $\gamma$, the virtual class of the associated moduli space $\mathsf{GW}_{\gamma}(Y_0)$ can be expressed in terms of moduli spaces of the form $\mathsf{GW}(X_v)$ attached to the vertices of $\gamma$, together with cohomology classes pulled back along the evaluation maps, see~\cite{ACGS17,R19}. 
\end{enumerate}

We now establish the same package of three results for DT and PT theory, and also strengthen the final formula on both curve and sheaf sides.

\subsection*{Vertex-by-vertex numerical splitting} The strengthening referred to above concerns the splitting aspect of the package. One key feature of our splitting formula, as stated in Section~\ref{sec: splitting-compatibility} is that it gives both a cycle-level and numerical statement, split across vertices. On the GW side, this is stronger than the statement established in~\cite{R19}, so in Section \ref{sec: main-splitting-formula} we explain how to modify the argument here to obtain the analogous strong splitting property for GW invariants in the expanded formalism of~\cite{R19}. To emphasize the point, we state explicitly the following:

\noindent
{\bf Key Consequence.} {\it The numerical logarithmic GW/DT/PT invariants of the strata of $Y_0$ determine the logarithmic GW/DT/PT invariants of a general fiber $\mathcal Y_t$, with non-vanishing cohomology\footnote{In other words, insertions that are restricted from the total space of the degeneration.} insertions.}

The GW/DT/PT invariants of a stratum $W$ in $Y_0$ are, by definition, the invariants of the projectivization $\mathbb P_{W}(N_{W/\mathcal Y})$ of its normal bundle. 

To be explicit, we will show deformation invariance, expressing DT invariants of $\mathcal Y_t$ in terms of those of $Y_0$.  Second we prove a decomposition result, expressing the DT invariants of $Y_0$ as a sum of invariants indexed by rigid Hilbert $1$-complexes $\gamma$.  Finally, we prove a splitting formula, which expresses the contribution of $\gamma$ in terms of DT invariants associated to the vertices of $\gamma$. The first two steps, treated in the next two sections, are essentially formal. 

The splitting aspect is the heart of this part of the paper. We will proceed in several steps: a purely combinatorial version of the splitting is established in Section~\ref{sec: combinatorial-degeneration}. After this, a gluing formula parallel to the existing GW formulas is established in Section~\ref{sec: moving-gluing}, and finally the full formula is completed in all theories in Section~\ref{sec: main-splitting-formula}. The full statement of the formulas and the degeneration compatiblity of the conjectures is stated before the proofs in Section~\ref{sec: splitting-compatibility} and several basic examples are recorded in Section~\ref{sec: examples}.

\section{DT space over a base and deformation invariance}\label{sec: DT-over-base}

\subsection{Moduli of vertical $1$-complexes} Let $(B,0)$ be a smooth pointed curve and let
\[
\mathcal Y\to B
\]
be a {\it simple normal crossings degeneration} defined as follows. The space $\mathcal Y$ 
 will be a simple normal crossings pair as in the first part of the paper, but it is no longer required to be proper. It is equipped with a morphism
 \[
\cY\to B
 \]
such that the preimage of $0$ is contained in the boundary divisor of $\cY$. The map is requied to be flat with reduced fibers. Note that $\mathcal Y$ may contain divisors that dominate $B$, i.e. ``horizontal'' divisors: this allows the possibility that the generic fiber is equipped with a nontrivial divisorial boundary. We denote the special fiber by $Y_0$.

Since the family is logarithmic, there is an associated map of cone complexes:
\[
\pi: \Sigma_{\mathcal Y}\to \mathbb R_{\geq 0}
\]
We denote the fiber of $\pi$ over the point $1$ by $\Delta_{Y_0}$. It is a possibly non-compact polyhedral complex.

The irreducible components of $Y_0$ will be denoted as $X_v$ so we have
\[
Y_0 = \bigcup_{v} X_v,
\]
where $v$ runs over the vertices of $\Delta_{Y_0}$. Typically, the degeneration $\mathcal Y$ will be a smooth projective variety fibered over $B$ in varieties of dimension at most $3$, but this assumption is not necessary until the virtual class is involved. 

We construct the associated moduli stack of expansions, using the results of~\cite[Section~3]{MR20}. First, let us treat $\Sigma_{\mathcal Y}$ as a cone complex in absolute terms, and ignore the map to $\mathbb R_{\geq 0}$. The constructions in~\cite{MR20} give rise to a set of embedded $1$-complexes in $\Sigma_{\mathcal Y}$ and this set can be given the structure of a cone complex. Let us denote this $T'(\Sigma_{\mathcal Y})$. 

We say a $1$-complex in $\Sigma_{\mathcal Y}$ is \textit{vertical} if it is contracted by the map $\pi$ above. Let $|T(\Sigma_{\mathcal Y}/\RR_{\geq 0})|$ denote the subset of  $|T'(\Sigma_{\mathcal Y})|$ parametrizing vertical $1$-complexes. 

\begin{lemma}
    The subset $|T(\Sigma_{\mathcal Y}/\RR_{\geq 0})|$ is the set of points in a union of cones in $T'(\Sigma_{\mathcal Y})$, for any cone complex structure on the latter.
\end{lemma}

\begin{proof}
    Fix a cone complex structure on $T'(\Sigma_{\mathcal Y})$ and choose a point $p$ in the relative interior of a cone $\sigma$ that parameterizes a cone complex that fails to be vertical. If $\tau$ is a cone that contains $\sigma$ as a face, then for any point $q$ in the relative interior of $\tau$, the $1$-complex corresponding to $q$ is also not vertical. It follows that we can remove the locus of non-vertical $1$-complexes by passing to a union of cones, and so the result follows. 
\end{proof}

Let us fix a cone structure $T(\Sigma_{\mathcal Y}/\RR_{\geq 0})$ as well as a universal family $\mathbf G(\Sigma_{\mathcal Y}/\RR_{\geq 0})$ of $1$-complexes over it. The latter will be compressed to just $\mathbf G$ when there is no chance of confusion.

As a consequence of the lemma, there is a canonical map
\[
T(\Sigma_{\mathcal Y}/\RR_{\geq 0})\to \RR_{\geq 0}
\]
sending a vertical $1$-complex to the point to which it is contracted.

\subsection{Decorated tropical spaces}

As in Section \ref{sec:decorated}, we can consider vertical Chow or Hilbert $1$-complexes, or tropical maps depending on the context. In DT/PT theory, for fixed $\beta$ and $\chi$, we have a diagram of cone complexes
\[
T_{\beta,\chi}(\Sigma_{\mathcal Y}/\RR_{\geq 0})\rightarrow T_{\beta}(\Sigma_{\mathcal Y}/\RR_{\geq 0})\rightarrow T(\Sigma_{\mathcal Y}/\RR_{\geq 0}).
\]
We always choose the cone structures on these spaces compatibly, so that each of the forgetful maps is an isomorphism upon restriction to each cone. Note that without further restrictions these cone complexes and the maps between them are typically not of finite type.

\subsection{Geometric moduli spaces} From the construction in the previous section, we can convert $T(\Sigma_{\mathcal Y}/\RR_{\geq 0})\to \RR_{\geq 0}$ into a map of Artin fans
\[
\mathsf A(T(\Sigma_{\mathcal Y}/\RR_{\geq 0})]\to[\mathbb A^1/\mathbb G_m].
\]
By also base changing along the canonical map $B\to[\mathbb A^1/\mathbb G_m]$ associated to the Cartier divisor $0$ in $B$, we obtain {\it the stack of expansions of a degeneration}:
\[
\mathsf{Exp}(\mathcal Y/B)\to B,
\]
which carries a universal family of expansions of $\mathcal Y$, determined by the combinatorics of $\mathbf G(\Sigma_{\mathcal Y}/\RR_{\geq 0})$.  As before, we have decorated variants $\mathsf{Exp}_{\beta}(\mathcal Y/B)$ and $\mathsf{Exp}_{\beta,\chi}(\mathcal Y/B)$ as well, and \'etale maps between them.

We can use this to build the DT space, over $B$, in this setting. Given a scheme $S\to B$, we can consider families of expansions given by commutative triangles:
\[
\begin{tikzcd}
    S\arrow{rr}\arrow{dr} & & \mathsf{Exp}(\mathcal Y/B)\arrow{dl}\\
    & B.&
\end{tikzcd}
\]
By pulling back the universal family, we obtain a family $\mathcal Y_S\to B$. The fiber over a point $s\in S$, that lies over $b\in B$ is an expansion over the fiber of $\mathcal Y$ over $b$. As before, at each point, there is well-defined notion of tube component -- a distinguished subset of the $\mathbb P^1$-bundle components. Note that components of the special fiber of $\mathcal Y$ are \textit{never} tubes. 

We can similarly study subschemes
\[
\begin{tikzcd}
    Z_S\arrow{rr}\arrow{dr} & & \mathcal Y_S\arrow{dl}\\
    & S,&
\end{tikzcd}
\]
in expansions over a fixed $S\to B$. At each geometric point of $S$, we can demand the DT stability condition, as before. This gives rise to a moduli functor
\[
\mathsf{DT}(\mathcal Y/B)\to\mathsf{Exp}(\mathcal Y/B)\to B.
\]
If we study PT stable pairs in expansions over a fixed $S$, we obtain the moduli functor $\mathsf{PT}(\mathcal Y/B)$.

The structure results from our earlier paper~\cite{MR20} carry over to this context. We summarize the outcome. 

\begin{theorem}
    The space $\mathsf{DT}(\mathcal Y/B)$ is universally closed and separated. For each fixed choice of curve class $\beta$ and holomorphic Euler characteristic $\chi$, the subspace parameterizing subschemes with these data is bounded. The map
    \[
    \mathsf{DT}(\mathcal Y/B)\to\mathsf{Exp}(\mathcal Y/B)
    \]
    is equipped with a relative perfect obstruction theory.
    Parallel results hold for $\mathsf{PT}(\mathcal Y/B).$
\end{theorem}

\begin{proof}
    The functor $\mathsf{DT}(\mathcal Y/B)$ can be viewed as a closed subfunctor of the space $\mathsf{DT}(\mathcal Y)$, where the latter parameterizes proper, at most $1$-dimensional, closed subschemes of the total space. Precisely, we demand that at every geometric point, the subscheme $Z_S\hookrightarrow \mathcal Y_S$ is contained in a fiber of the projection to $\A^1$. This establishes the universal closedness, separatedness, and boundedness after fixing numerical data. For the perfect obstruction theory, we note that the universal expansion is always a double point degeneration, i.e. there are no triple points. The arguments of~\cite[Section~3]{MPT10} therefore carry over without change, as in our first paper~\cite[Section~5]{MR20}.  
\end{proof}

There are decorated variants $\mathsf{Exp}_{\beta}(\mathcal Y/B)$ and $\mathsf{Exp}_{\beta,\chi}(\mathcal Y/B)$ as well, and the forgetful map from $\mathsf{DT}_{\beta,\chi}(\cY/B)$ factors over these.

\subsection{Deformation invariance} We state the deformation invariance property of logarithmic DT theory. 

\begin{theorem}
    Let $\mathcal Y\to B$ be a simple normal crossings degeneration of a threefold. The space $\mathsf{DT}_{\beta,
    \chi}(\mathcal Y/B)$ has a relative obstruction theory over $B$. As a result, there is an associated virtual class in Borel--Moore homology:
    \[
    [\mathsf{DT}_{\beta,
    \chi}(\mathcal Y/B)]^{\mathsf{vir}}\in \mathsf{H}_\star(\mathsf{DT}_{\beta,
    \chi}(\mathcal Y/B)).
    \]
    Furthermore, if $j_b:\{b\}\hookrightarrow B$ denotes the inclusion, and $Y_b$ denotes the corresponding fiber of $\mathcal Y$, then
    \[
    j_b^![\mathsf{DT}_{\beta,
    \chi}(\mathcal Y/B)]^{\mathsf{vir}} = [\mathsf{DT}_{\beta,
    \chi}(Y_b)]^{\mathsf{vir}}.
    \]
Parallel results hold for $[\mathsf{PT}(\mathcal Y/B)]^{\mathsf{vir}}.$
\end{theorem}

The same construction also gives virtual classes in Chow groups, but the Borel--Moore statement is more useful to us here.

\begin{proof}
    We have already noted that there is a relative obstruction theory for the morphism
    \[
    \mathsf{DT}_{\beta,
    \chi}(\mathcal Y/B)\to \mathsf{Exp}(\mathcal Y/B).
    \]
    If we pullback this obstruction theory along the inclusion $j_b$, we recover the obstruction theory on the smooth fibers of $\mathcal Y$ from~\cite{MR20}, and the obstruction theory of the previous section for the fiber over $0$. We can define Borel--Moore virtual classes by applying virtual pullback defined by Kapranov--Vasserot~\cite[Section~3.2]{KV19}. Alternatively, we can use the Chow theoretic virtual pullback~\cite{Mano12} and then use the cycle class map. The statement of the theorem now follows from formal properties of the pullback. 
\end{proof}


\section{Decomposition into virtual strata}\label{sec: decomposition-step}

\subsection{Virtual stratification} We have constructed a moduli space equipped with structure maps:
\[
\mathsf{DT}_{\beta,\chi}(\mathcal Y/B)\to \mathsf{Exp}(\mathcal Y/B)\to B.
\]
The second map has relative dimension $0$, but the fiber over $0$ will typically be reducible. Let us now pass to the fibers over $0$:
\[
\mathsf{DT}_{\beta,\chi}(Y_0)\to \mathsf{Exp}(Y_0)\to \Spec(\mathbb C). 
\]
By the discussion in the previous section, the first arrow is virtually smooth, i.e. equipped with a relative perfect obstruction theory. However, the $0$-dimensional stack $\mathsf{Exp}(Y_0)$ is \textit{reducible}. Locally in the smooth topology, it is isomorphic to collection of toric varieties glued along their boundary divisors. In fact, since we are free to choose the polyhedral structure on $\mathsf T(\Sigma_{\mathcal Y}/\mathbb R_{\geq 0})$ so all its cones are smooth, locally in the smooth topology, this stack has the structure a normal crossings variety. 

Since $\mathsf{Exp}(Y_0)$ is stratified, by pulling back along $\mathsf{DT}_{\beta,\chi}(\mathcal Y/B)\to \mathsf{Exp}(\mathcal Y/B)$, we obtain a ``virtual'' stratification on the moduli space of subschemes\footnote{The virtual stratification is a stratification in the ``weak'' sense: a decomposition into locally closed subschemes, such that the closure of a stratum is \textit{contained} in a union of strata. We use the term ``virtual'' rather than ``weak'' to remind the reader that it is the pullback stratification along the map with the perfect obstruction theory. Indeed, this makes each virtual stratum virtually smooth.}. 

The discussion above has a combinatorial parallel. The fiber over the point $1$ in
\[
\mathsf T(\Sigma_{\mathcal Y})\to \mathbb R_{\geq 0}
\]
is a polyhedral (and not necessarily conical) complex. We denote it $\mathsf T(\Delta_{Y_0})$.

\begin{proposition}\label{Exp-decomposition}
    The strata of $\mathsf{Exp}(Y_0)$ are in bijection with the vertices of $\mathsf T(\Delta_{Y_0})$.
\end{proposition}

\begin{proof}
    The components of the fiber over $0$ in $\mathsf{Exp}(\mathcal Y/B)$ are in bijection with those rays of $\mathsf T(\Sigma_{\mathcal Y})$ that map surjectively onto $\mathbb R_{\geq 0}$. These are precisely in bijection with the vertices of $\mathsf T(\Delta_{Y_0})$.
\end{proof}

\begin{definition}
A \textit{rigid $1$-complex} are the $1$-complexes corresponding to the vertices of $\mathsf T(\Delta_{Y_0})$.  Similarly, \textit{rigid} Chow $1$-complexes are those corresponding to vertices of  $\mathsf{T}_{\beta}(\Delta_{Y_0})$; rigid Hilbert $1$-complexes correspond to vertices of $\mathsf{T}_{\beta,\chi}(\Delta_{Y_0})$.
\end{definition}

By the choice of our cone structures on decorated spaces, a Chow or Hilbert $1$-complex is rigid if and only if its underlying $1$-complex is rigid.

It follows that the virtual strata of $\mathsf{DT}_{\beta,\chi}(Y_0)$ are in bijection with a subset of these rigid $1$-complexes. Note that this can be a strict subset, as we have no way of guaranteeing non-emptiness of any particular fiber of the map $\mathsf{DT}_{\beta,\chi}(Y_0)\to \mathsf{Exp}(Y_0)$. Nevertheless, our arguments will typically be insensitive to this, so allowing the possibility of a finite set of distinct, empty strata, it is harmless to think about the virtual strata as being in bijection with these vertices.

\subsection{Decomposition} By the deformation invariance above, in order to understand the virtual fundamental class of the general fiber of the DT space of $\mathcal Y/B$, it suffices to understand the fiber over $0$. There is a natural morphism:
\[
\mathsf{DT}_{\beta,\chi}(Y_0)\to \mathsf{Exp}_{\beta,\chi}(Y_0).
\]
As noted previously, the stack $\mathsf{Exp}_{\beta,\chi}(Y_0)$ has several irreducible components corresponding to rigid Hilbert $1$-complexes.

Let us now fix one such rigid Hilbert $1$-complex $\gamma$. There is an irreducible component
\[
\mathsf{Exp}_{\gamma}(Y_0)
\]
associated to $\gamma$. Denote the fiber product of the inclusion of this component with $\mathsf{DT}_{\beta,\chi}(Y_0)$ by $\mathsf DT_\gamma(Y_0)$. As before, the map
\[
\mathsf{DT}_{\gamma}(Y_0)\to \mathsf{Exp}_{\gamma}(Y_0)
\]
is equipped with a perfect obstruction theory. By pulling back the fundamental class, we obtain a virtual class $[\mathsf{DT}_{\gamma}(Y_0)]^{\mathsf{vir}}$. 

We now state the decomposition theorem for the virtual class.

\begin{theorem}
There is an equality of homology classes on $\mathsf{DT}_{\beta,\chi}(Y_0)$ given by
\[
\mathsf{DT}_{\beta,\chi}(Y_0) = \sum_\gamma [\mathsf{DT}_{\gamma}(Y_0)]^{\mathsf{vir}},
\]
where $\gamma$ ranges over the set of rigid Hilbert $1$-complexes.
\end{theorem}

\begin{proof}
    The result follows from the irreducible component decomposition of $\mathsf{Exp}_{\gamma}(Y_0)$ and the formal properties of the virtual pullback. Again, we can either use~\cite{KV19}, or Manolache's virtual pullback~\cite{Mano12}. In the latter case, we simply note that the class descends from cycles to Chow groups. The statement in Borel--Moore homology follows by applying the cycle class map. 
\end{proof}

\begin{remark}\label{rem: decomposition-gw}
Analogous results in the GW case can be found in~\cite[Theorem~3.11]{ACGS15}, though we note that much of loc. cit. concerns itself with exposing the virtual stratification of the space of logarithmic stable maps in terms of tropical curves. Our formalism essentially builds this into the definition, rendering it something of a tautology.  A \textit{rigid} tropical stable map
to $\Delta_{Y_0}$ is a tropical stable map whose image $1$-complex is rigid; as in the DT/PT setting,
virtual irreducible components of
$\mathsf{GW}(Y_0)$ are indexed by rigid tropical stable maps. 

In GW theory, there is an additional multiplicative factor. It is referred to in~\cite{ACGS15} as the {\it base order} and also appears in~\cite{AF11,CheDegForm}. If $\gamma$ is the rigid tropical stable map with induced Chow $1$-complex $\overline \gamma$, then define $\ell(\gamma)$, or simple $\ell$ if $\gamma$ is clear, to be the product, taken over all edges $e$ in the Chow $1$-complex $\overline \gamma$, of the least common multiplies of the dilation factors $m_{e'}$ of edges $e'$ in $\gamma$ that map to $e$.

The decomposition formula in GW theory states that the virtual class associated to the general fiber is equal to the sum of virtual classes of spaces associated to the rigid tropical maps, weighted by $\ell(\gamma)$.
\end{remark}

\subsection{Uniform rigid $1$-complexes}

If we fix $\beta$ but allow $\chi$ to vary, then a priori there are infinitely many rigid $1$-complexes appearing in the decomposition formula.  This would make the degeneration formula ill-suited for studying generating functions.  Fortunately, for the appropriate choice of cone structure, we have the following uniformity in which rigid complexes arise.

\begin{proposition}

Fix a curve class $\beta$.  There is a finite set $\overline{\gamma_1}, \dots, \overline{\gamma_r}$ of rigid $1$-complexes such that every rigid Chow (or Hilbert) $1$-complex with curve class $\beta$ has underlying $1$-complex in this set.
\end{proposition}

In other words, all rigid Hilbert $1$-complexes arise from choosing a $1$-complex from this finite list and distributing curve class and Euler characteristic decorations among the vertices/edges.  In particular, there are only finitely many vertices in $\mathsf{T}_{\beta}(\Delta_{Y_0})$.

\begin{proof}

This follows from a similar argument to Theorem~\ref{prop:flat-evaluation} using pure $1$-complexes and combinatorial flatness.
Recall from Section~\ref{sec: pure-complexes} the notions of pure $1$-complexes, spaces $\mathsf{T}^\mathsf{pure}$ of pure $1$-complexes, and the realization of the forgetful map
$\varphi\colon\mathsf{T} \rightarrow \mathsf{T}^\mathsf{pure}$ as the relative space of $0$-complexes, of some complete subdivision extending the universal $1$-complex.  In our setting, by taking fibers over $1$, we have polyhedral (as opposed to conical) versions of these constructions.

For fixed $\beta$, only a bounded subcomplex of pure complexes arise, so after choosing complex structures appropriately, we have a composition of combinatorially flat maps
\[
\mathsf{T}_{\beta}(\Delta) \rightarrow \varphi^{-1}(\mathsf{T}^\mathsf{pure}_{\sf bd}(\Delta)) \xrightarrow{\varphi} \mathsf{T}^\mathsf{pure}_{\sf bd}(\Delta).
\]

If we fix a vertex in the last space here, the vertices in its preimage with respect to $\varphi$ correspond to $0$-complexes obtained by distributing points to the vertices of a complete subdivision of $\Delta$.
In particular there are only finitely many such vertices.  These give us the finitely many $1$-complexes in the statement of the proposition.

\end{proof}

\begin{remark}
In fact, one can show that, in the statement of the proposition, we can ignore rigid $1$-complexes associated to $\beta$ that are not pure. By the splitting formula, if there are any free or linear $2$-valent vertices, they must support a non-minimal Hilbert polynomial and the corresponding invariant will vanish.  In fact, one can go further and show that the only rigid $1$-complexes that matter are those which are \textit{topologically} rigid:  any nearby $1$-complex has different combinatorial type.  The details for both statements will appear elsewhere.
\end{remark}


\section{Splitting formulas and compatibility of the conjectures}\label{sec: splitting-compatibility}

\subsection{Preliminaries} 

Let $\overline{\gamma}$ be a fixed rigid Chow $1$-complex, indexed by a vertex of the polyhedral complex $T_{\beta}(\Delta_{\cY})$.  In turn, each vertex $v$ of $\overline{\gamma}$ corresponds to an irreducible component $X_v$ which is a torus-bundle over a stratum of $Y_0$. The goal of the splitting formula is to calculate the contribution of $\overline{\gamma}$ 
in the decomposition formula for GW/DT/PT invariants in terms of the corresponding logarithmic invariants associated to its vertices.  More precisely, we will split, among vertices, the contribution to DT/PT invariants of each rigid Hilbert $1$-complex $\gamma$ lying over $\overline{\gamma}$ (and similarly with GW invariants and rigid tropical stable maps).

We may modify $\cY$ by a subdivision and root construction in such a way that the vertices of $\overline{\gamma}$ are a subset of the vertices of the height $1$ slice $\Delta_{\mathcal Y}$ of $\Sigma_{\mathcal Y}$, and similarly that the edges of $\overline{\gamma}$ are a subset of the edges of this slice. In particular, we can assume that $\overline{\gamma}$ has no tube vertices.

Given a vertex $v$ of $\overline{\gamma}$, let $\beta(v) \in H_2(X_v)$ denote the curve class associated to $\overline{\gamma}(v)$.
Let $e_1(v),\ldots, e_m(v)$ be the edges incident to $v$, which are decorated with integers $n_1, \dots, n_m$ from the labelling of $\gamma$.  

Using Theorem~\ref{prop:flat-evaluation}, we have a uniform combinatorial flattening of the evaluation maps
\[
T_{\beta(v)}(\Sigma_v) \rightarrow \left(\prod P(e_i(v)\right)^{\dagger}.
\]
We stress again that this map is not of finite type, and infinitely many cones of the domain might map to a single cone in the codomain and therefore the specific flattening construction of Theorem~\ref{prop:flat-evaluation} must be applied. 

\subsection{Cycle-theoretic splitting: sheaf theory}

Fix a rigid Hilbert $1$-complex $\gamma$ lying over $\overline{\gamma}$; concretely, this means
we have assigned $\chi$-labels to each of the vertices.
We use $\gamma(v)$ as the symbol used to indicate the induced discrete data for each vertex $v$ coming from $\gamma$.

Associated to the vertex $v$, we have the product of Hilbert schemes
\[
\mathsf{Hilb}^{n(v)}(D_v):=\prod_{i=1}^m \mathsf{Hilb}^{n_i}(D_{e_i(v)}). 
\]
Each factor on the right is the logarithmic Hilbert scheme of points on the appropriate divisor, with the number of points determined by the intersection number. 
From the flattened tropical evaluation map recalled above, we can take the corresponding tropical models, so there is a combinatorially flat evaluation map
\[
\mathsf{ev}_v\colon  \mathsf{DT}_{\gamma(v)}(X_v) \to (\mathsf{Hilb}^{n(v)}(D_v))^{\dagger},
\]
where $(-)^\dagger$ denotes a log modification of the product over $e_i(v)$.
Consolidating these over vertices, we have 
\[
\mathsf{ev}:\prod_v \mathsf{DT}_{\gamma(v)}(X_v) \to \prod_v (\mathsf{Hilb}^{n(v)}(D_v))^{\dagger}
\]
Note that this morphism is an external product, i.e. it is obtained by a product of maps, one attached to each vertex $v$. This will be important in splitting contributions across vertices.

To formulate our splitting formula, we require the following construction, further details about which appear in Section~\ref{sec: comb-splitting}. 

\begin{construction}
There is a cutting map $$\kappa: \mathsf{DT}_\gamma(Y_0)\to \prod_{v\in \gamma} \mathsf{DT}_{\gamma(v)}(X_v).$$
This map is constructed first at the combinatorial level. Given a $1$-complex $\gamma'$ that retracts onto $\gamma$, a $1$-complex on the factor corresponding to a vertex $v$ of $\gamma$ is determined by a neighborhood of $v$. This gives rise to a cutting map of cone complexes, and consequently a map of stacks
\[
\mathsf{Exp}_\gamma(Y_0)\to \prod_v \mathsf{Exp}_{\gamma(v)}(X_v).
\]
In practice, this amounts to taking an expansion of $Y_0$ and performing a partial normalization prescribed by $\gamma$. Intersecting with the subscheme gives rise to a commutative diagram involving DT spaces and stacks of expansions
\[
\begin{tikzcd}
\mathsf{DT}_\gamma(Y_0)\ar{r}{\kappa}\ar{d} & \prod_{v} \mathsf{DT}_{\gamma(v)}(X_v)\ar{d}\\
\mathsf{Exp}_\gamma(Y_0)\ar{r}& \prod_v \mathsf{Exp}_{\gamma(v)}(X_v).
\end{tikzcd}
\]
\end{construction}

For every edge $e$ of $\gamma$, the Hilbert scheme for the corresponding stratum $D_e$ occurs twice in the product $\prod_v \mathsf{Hilb}(D_v)$, once for each adjacent vertex.  As a result, we can consider the diagonal $$\Delta_H \subset \prod_v \mathsf{Hilb}_{\beta}(D_v).$$  
We obtain a cycle
\[
\Delta_H^\dagger\subset\prod_v \mathsf{Hilb}^\dagger_{\beta(v)}(D_v)
\]
by taking strict transform.

We now precisely state the strongest form of the logarithmic degeneration formula.

\begin{theorem}\label{thm: main-splitting-theorem}
 There is an equality of homology classes
\[
\kappa_\star\left[\mathsf{DT}_\gamma(Y_0) \right]^{\mathsf{vir}} = \left(\mathsf{ev}\right)^\star\Delta^\dagger\cap \left[\prod_v \mathsf{DT}_{\gamma(v)}(X_v) \right]^{\mathsf{vir}}.
\]
The analogous formula holds for $\kappa_\star\left[\mathsf{PT}_\gamma(Y_0) \right]^{\mathsf{vir}}$.
\end{theorem}

\subsection{Cycle-theoretic splitting: curve theory} 


A parallel cycle theoretic splitting holds in GW theory, although some of the combinatorial aspects are a little different. 

\subsubsection{Combinatorial organization} We continue to fix a Chow $1$-complex $\overline\gamma$, but now parameterize it using a tropical stable map $\gamma$. Precisely, $\gamma$ is a pair $(\mathbf G,f)$, where $\mathbf G$ is a decorated tropical curve in the sense of Section~\ref{sec:decorated} and $f$ is a map to the polyhedral complex $\Delta_{Y_0}$. A map
\[
f:\mathbf G\to \Delta_{Y_0}
\]
is called {\it rigid} if, similarly to the Chow and Hilbert cases, if it corresponds to a vertex of the moduli space of tropical stable maps to $\Delta_{Y_0}$ with the chosen cone complex structure. As we have noted in Section~\ref{sec:decorated}, the tropical stable map $\gamma$ determines a Chow $1$-complex $\overline \gamma$. 

Let us spell out the ways in which these two could differ:
\begin{enumerate}[(i)]
\item For each vertex of $\overline \gamma$, there could be several corresponding vertices of $\mathbf G$ of the tropical map that map to the given vertex. If this happens, the degree labels on $\gamma$ are required to sum to $\overline \gamma$, so this can only happen in finitely many ways. 
\item Given vertices $u$ and $u'$ connected by an edge $e$ then, for each lift $v$ and $v'$ of these vertices to $\mathbf G$, there could be several edges between $v$ and $v'$. Recall that $e$ is decorated by a multiplicity $n_e$, and each edge of $\mathbf G$ is comes with a dilation factor associated to the map $f$. The factor is necessarily nonzero, and the sum of multiplicities of all edges $e'$ over $e$ must sum to $n_e$. 
\item Similarly, for each leg $e$ of $\overline \gamma$ there can be finitely many legs of $\mathbf G$ that map to $e$. Again, the multiplicities are necessarily positive, and must sum to $n_e$. 
\end{enumerate}

\begin{remark}[Automorphisms]\label{rem: automorphism}
Another complication in GW theory is that the rigid objects $\gamma$ may have automorphisms. This stems from point (ii) above. Precisely, suppose we have $u$ and $u'$ in $\overline\gamma$ with an edge $e$ between them, and also lifts $v$ and $v'$ in $\mathbf G$ with multiple edges between them lying over $e$. If the intersection multiplicities along these edges are the same, then there there are automorphisms of $\mathbf G$ that commute with the map to $\Delta_{Y_0}$. 
\end{remark}

For each vertex $u$ in $\overline \gamma$, we have a well-defined target $X_u$ which is a component of $Y_0$. We also have a curve class $\beta_u$ associated to this vertex. Moreover, for each edge or leg $e$ incident to $u$, we have an intersection multiplicity $n_e$ with this divisor. By choose a lift $\gamma$ of $\overline \gamma$, we enhance these data. Precisely, for each vertex, we have chosen a number of connected components to break this curve class into, indexed by vertices $v$ in $\mathbf G$ over $u$. For each vertex $v$, i.e. for each such connected component, we have chosen a genus, an intersection multiplicity with the boundary of $X_u = X_v$, and a partition of each intersection multiplicity at an edge. 

\subsubsection{The formula} With the organization clarified, we can state the formula. For each vertex $v$ of $\gamma$, we set
$$\mathsf{Ev}_{{\bm \mu}(v)} := \prod_{i=1}^{m} \mathsf{Ev}_{\mu_{e_{i}(v)}}(D_{e_i(v)})$$
to be the product of the GW evaluation spaces from Section~\ref{sec: geometric-evaluations}.

Using the flattened tropical evaluation
\[
T_{\beta(v)}(\Sigma_v) \rightarrow \left(\prod P(e_i(v)\right)^{\dagger},
\]
which is the same one we used for the DT side, we have tropical models of the evaluation space for for each vertex $v$:
$$\mathsf{Ev}_{{\bm \mu}(v)}^{\dagger} := \left(\prod_{i=1}^{m} \mathsf{Ev}_{\mu_{e_{i}(v)}}(D_{e_i(v)})\right)^{\dagger}.$$
These are the target of evaluation maps from the GW spaces, and consolidating over all vertices, we have:
\[
\mathsf{ev}:\prod_v \mathsf{GW}_{\gamma(v)}(X_v)\to \prod_v \mathsf{Ev}_{{\bm \mu}(v)}^{\dagger}.
\]

Finally, observe that associated to $\gamma$ we have a component of the space of maps to $Y_0$, denoted $\mathsf{GW}_\gamma(Y_0)$. Recall the automorphisms noted in Remark~\ref{rem: automorphism}. We can partially rigidify this stack by taking a cover, e.g. by labelling all the edges in $\mathbf G$. Let $\mathsf{GW}'_\gamma(Y_0)$ denote the corresponding stack. See~\cite[Section~2.6]{ACGS15} for a detailed treatment. Finally, as in Remark~\ref{rem: decomposition-gw} let $\ell(\gamma)$ denote the product, taken over all edges $e$ in $\overline \gamma$, of the least common multiplies of the dilation factors $m_{e'}$ of edges $e'$ in $\gamma$ that map to $e$.\footnote{Note that this factor was mistakenly dropped in~\cite{R19}. It appears in the cycle theoretic splitting formula for each rigid tropical curve but cancels out in the final numerical splitting formula due to the base change order in the decomposition part~\cite{ACGS15}. }

As in the construction above for the DT moduli spaces, there is a cutting map
\[
\kappa: \mathsf{GW}'_\gamma(Y_0)\to\prod_{v\in \gamma} \mathsf{GW}_{\gamma(v)}(X_v).
\]
If we let 
 $$\Delta_{\mathsf{Ev}_{\bm \mu}}^{\dagger} \subset \prod_{v} \mathsf{Ev}_{{\bm \mu}(v)}^{\dagger}$$
 denote the strict transform of the diagonal, we can state the GW splitting formula:
 
\begin{theorem}\label{thm: main-splitting-theorem-GW}
 There is an equality of homology classes
\[
\kappa_\star\left[\mathsf{GW'}_\gamma(Y_0) \right] =\prod_e m_e\cdot \frac{1}{\ell(\gamma)} \left(\mathsf{ev}\right)^\star\Delta^\dagger\cap \left[\prod_v \mathsf{GW}_{\gamma(v)}(X_v) \right]^{\mathsf{vir}},
\]
where the indicated product is the product of all dilation factors of all bounded edges $e$ in the rigid tropical map $\gamma$. 
\end{theorem}


\subsection{Numerical degeneration formulas}\label{sec: numerical-formulas}

We can specialize the cycle-theoretic result to obtain numerical degeneration formulas.

We fix some additional notation.
For each rigid Chow $1$-complex $\gamma$, we will sum over collections of partitions $\bm{\mu}:= (\mu_e)$ of the edge labels $n_e$.  For each vertex $v$, we denote by $\bm{\mu}_v$ the subset of $\bm{\mu}$ given by the edges
adjacent to $v$.
The strict transform of the diagonal
$\Delta_{\mathsf{Ev}^{\bm{\mu}}}^{\dagger}$ lives in a product, so we can take its K\"unneth decomposition:

\[
[\Delta_{\mathsf{Ev}_{\bm{\mu}}}^{\dagger}] = \sum_{j} \otimes_{v} \delta^{(j)}_{v} \in \bigotimes H^\star(\mathsf{Ev}_{{\bm{\mu}}(v)}^{\dagger})
\]

\noindent
Given an index set $S$, for each $i \in S$, we fix an insertion $\zeta_i \in H^\star(\mathcal Y)$ and integer $a_i$.  We also denote by $\zeta_i$ the restriction of this insertion to either a fiber $Y_t$ or an irreducible component $X_v$ of an expansion of $Y_0$.  

\begin{corollary}\label{cor: DT-deg-form}
The following numerical degeneration formulas hold for DT/PT invariants in the setting above.
\[
\langle \prod_{i\in S} \tau_{a_i}(\zeta_i)\rangle^{DT, Y_t}_{\beta, \chi} 
=
\sum_{\gamma} \sum_{\bm{\mu}} \sum_{S = \coprod S_v} \sum_{(j)} \frac{(-1)^{\bm{\mu}}}{\mathrm{Aut}(\bm{\mu})}\cdot m_{\bm{\mu}}
\prod_{v}  \langle \prod_{i \in S_v} \tau_{a_i}(\zeta_i)| \bm{\mu}_{v}(\delta^{(j)}_{v}) \rangle^{DT, X_v}_{\gamma'(v)}.
\]
and
\[
\langle \prod_{i\in S} \tau_{a_i}(\zeta_i)\rangle^{PT, Y_t}_{\beta, \chi} 
=
\sum_{\gamma} \sum_{\bm{\mu}} \sum_{S = \coprod S_v} \sum_{(j)} \frac{(-1)^{\bm{\mu}}}{\mathrm{Aut}(\bm{\mu})}\cdot m_{\bm{\mu}}
\prod_{v}  \langle \prod_{i \in S_v} \tau_{a_i}(\zeta_i)| \bm{\mu}_{v}(\delta^{(j)}_{v}) \rangle^{PT, X_v}_{\gamma'(v)}.
\]

\end{corollary}

In this corollary, the first sum is over rigid Hilbert $1$-complexes $\gamma$ with discrete invariants $\beta$ and $\chi$; the second sum is over all collections of partitions of the edge labels.  The third sum is over all possible distribution of insertions among the vertices; the last sum is over the terms of the K\"unneth decomposition for the diagonal.

\begin{proof}

Since the cohomology classes $\zeta_i$ are restricted from $\mathcal Y$, using the deformation invariance and decomposition parts of our degeneration package, 
we have the equality
\[
\langle \prod_{i\in S} \tau_{a_i}(\zeta_i)\rangle^{DT, Y_t}_{\beta, \chi} 
=
\sum_{\gamma} 
\mathrm{deg}\left( \prod_{i=1}^{k}\tau_{a_{i}}(\zeta_i)[\mathsf{DT}_{\gamma}]^{\mathsf{vir}}  \right)
\]
where $\gamma'$ ranges over rigid Hilbert $1$-complexes with total discrete invariants $\beta$ and $\chi$.

To split the invariant associated to $\gamma$, we pushforward along the cutting map:
\[
 \kappa: \mathsf{DT}_{\gamma}(Y_0)\to \prod_{v\in \gamma} \mathsf{DT}_{\gamma(v)}(X_v).
\]

Theorem~\ref{thm: main-splitting-theorem} tells us how the virtual class pushes forward, and Corollary~\ref{cor: diagonal-formula} gives us a decomposition of the diagonal in terms of the boundary insertions on each factor of the product.  It remains to analyze the descendent operators.

By construction, we have the partial normalization of the universal target over $\mathsf{DT}_{\gamma'}(Y_0)$
\[
p: \coprod \mathcal{X}_v \rightarrow \mathcal{Y_0}
\]
where $\mathcal{X}_v$ denotes the pullback of the universal target from $\mathsf{DT}_{\gamma'(v)}(X_v)$.
.
By algebraic transversality, we have an equality of ideal sheaves
$$p^\star(\mathcal{I}_{\gamma}) = \oplus \mathcal{I}_v$$
as well as an equality of Chern characters
$$p^\star\mathrm{ch}_{a+2}(\mathcal{I}_{\gamma}) = \sum \mathrm{ch}_{a+2}(\mathcal{I}_v).$$

The push-pull formula gives 
\[
\kappa_\star(\tau_{a}(\zeta_i)\circ ) = \sum_{v} \tau_{a}(\zeta_i)\circ \kappa_\star
\]
As we iterate over our insertions, we obtain the statement of the corollary.
\end{proof}

We state the numerical GW degeneration formula for the disconnected invariants.  

\begin{corollary}\label{cor: GW-deg-form}
The following numerical degeneration formula holds for GW invariants in the setting above. 
\[
\langle \prod_{i\in S} \tau_{a_i}(\zeta_i)\rangle^{\bullet, GW, Y_t}_{\beta, g} 
=
\sum_{\overline{\gamma}}  \sum_{g(v), \bm{\mu}} \sum_{S = \coprod S_v}  \sum_{(j)} \frac{(-1)^{\bm{\mu}}}{\mathrm{Aut}(\bm{\mu})}\cdot m_{\bm{\mu}}
\prod_{v}  \langle \prod_{i \in S_v} \tau_{a_i}(\zeta_i)| \bm{\mu}_{v}(\delta^{(j)}_{v}) \rangle^{\bullet, GW, X_v}_{\gamma(v)}.
\]
\end{corollary}
In this statement, the first summation is over rigid Chow $1$-complexes.  The second summation is over genus decorations for vertices and partition decorations for edges so that the total genus agrees with $g$.  For connected invariants, one can refine this equation to a summation over rigid tropical maps.  The two formulas are related by taking account of how labelling affects automorphism factors.

\begin{proof}
The formula follows analogously, by applying the K\"unneth decomposed formula for the strict diagonal. We note two features that are different from the DT side. First, the automorphism contribution, which is slightly different from the DT/PT side, is accounted for by Remark~\ref{rem: automorphism}. Second, the least common multiple multiplicity $\ell(\gamma)$ from Remark~\ref{rem: decomposition-gw} does not appear in the final numerical formula. Indeed, a factor of $\frac{1}{\ell(\gamma)}$ appears in Theorem~\ref{thm: main-splitting-theorem-GW}. However, for each rigid tropical map $\gamma$, the ``base order'' in the decomposition formula~\cite{ACGS15} produces a multiple of $\ell(\gamma)$ as in Remark~\ref{rem: decomposition-gw}. These cancel in the final formula, leading to the claimed result. 
\end{proof}

\begin{remark}
As noted in the introduction, a feature of these corollaries is that the decomposition is ``vertex by vertex'', which allows to use the K\"unneth decomposition to give a  significantly improved splitting as compared to the statement in~\cite{R19}. The cost, as compared to the situation of smooth pairs, is the need for the exotic evaluation classes discussed in the introduction. Precisely, the exotic insertions arise from the blowup step discussed in Section~\ref{sec: Hilb-cohomology}.
\end{remark}

\subsection{Generating function corollaries}

As a consequence of the numerical degeneration formula, we obtain degeneration formulas
for the GW/DT/PT partition functions.  
For these formulas to make sense, it is important that we are summing over a fixed 
set of rigid Chow $1$-complexes that only depends on $\beta$ and not the other discrete invariants (genus or Euler characteristic).

For DT and PT partition functions, the degeneration formulas are as follows. In order to display the equations compactly, we compress $\mathsf{Z}_{\mathsf{DT}}\left(Y_t;q | \prod_{i\in S} \tau_{a_{i}}(\zeta_i)\right)_\beta$ to $\mathsf{Z}_{\mathsf{DT}}(q)$, and similarly for PT and GW series, with the variable for the GW series being $u$. 
\[
\mathsf{Z}_{\mathsf{DT}}(q)
= 
\sum_{\overline{\gamma}} \sum_{\bm{\mu}} \sum_{S = \coprod S_v} \sum_{(j)} \frac{(-1)^{\bm{\mu}}m_{\bm{\mu}}}{\mathrm{Aut}(\bm{\mu})}\frac{1}{q^{|\bm{\mu}|}}
\prod_{v}  
\mathsf{Z}_{\mathsf{DT}}\left(X_v,D_v;q |
\prod_{i \in S_v} \tau_{a_i}(\zeta_i)| \bm{\mu}_{v}(\delta^{(j)}_{v})
\right)_{\beta(v)}
\]
and
\[
\mathsf{Z}_{\mathsf{PT}}(q)
= 
\sum_{\overline{\gamma}} \sum_{\bm{\mu}} \sum_{S = \coprod S_v} \sum_{(j)} \frac{(-1)^{\bm{\mu}}m_{\bm{\mu}}}{\mathrm{Aut}(\bm{\mu})}\frac{1}{q^{|\bm{\mu}|}}
\prod_{v}  
\mathsf{Z}_{\mathsf{PT}}\left(X_v,D_v;q |
\prod_{i \in S_v} \tau_{a_i}(\zeta_i)| \bm{\mu}_{v}(\delta^{(j)}_{v})
\right)_{\beta(v)}.
\]
For the GW partition function, we have
\[
\mathsf{Z}_{\mathsf{GW}}(u)
= 
\sum_{\overline{\gamma}} \sum_{\bm{\mu}} \sum_{S = \coprod S_v} \sum_{(j)} \frac{(-1)^{\bm{\mu}}m_{\bm{\mu}}}{\mathrm{Aut}(\bm{\mu})}\cdot u^{2\ell(\bm{\mu})}
\prod_{v}  
\mathsf{Z}_{\mathsf{GW}}\left(X_v,D_v;u |
\prod_{i \in S_v} \tau_{a_i}(\zeta_i)| \bm{\mu}_{v}(\delta^{(j)}_{v})
\right)_{\beta(v)}.
\]

In each of these expressions, the first summation is over the finite set of rigid Chow $1$-complexes.  

As a result, we have the compatibility of all of the logarithmic curve/sheaf conjectures from Section \ref{sec:curve/sheaf-conjectures}.
As above, we fix a simple normal crossings degeneration $\mathcal Y$ and, for each stratum $Z_w$ of $Y_0$, let $X_w$ denote the projectivization of its normal bundle in $\mathcal Y$, viewed, up to logarithmic modification, as a smooth variety with simple normal crossings boundary. We fix insertions $\zeta_i \in H^\star(\mathcal Y)$.

\begin{corollary}
If each of Conjectures \ref{conj: rationality}, \ref{conj: primary}, and \ref{conj: descendent}
hold for all components $X_w$ in $Y_0$ with insertions $\zeta_i$, then they also hold for $Y_t$ with insertions $\zeta_i$.
\end{corollary}

\subsection{A roadmap to the proof}\label{sec: roadmap} The proof of the degeneration formula requires a number of steps. The first task is to establish a parallel statement in a purely combinatorial context, effectively describing the decomposition of $\mathsf{Exp}_{\beta,\chi}(Y_0)$, and this is done in Section~\ref{sec: combinatorial-degeneration}. Next, we prove a weak version of the formula, involving $0$-dimensional subschemes on the universal divisor expansion over the moduli space. This is done in Section~\ref{sec: moving-gluing} and brings the DT side of the story to essentially the same level of development as the existing logarithmic GW gluing result~\cite[Theorem~B]{R19}. Finally, we combine these weaker statements, on both GW and DT sides, with logarithmic intersection theory arguments to upgrade them to the claimed ``vertex-by-vertex'' formula above. This latter step is done on both GW and DT sides, and we emphasize that it is new even on the GW side. 

\section{First examples}\label{sec: examples}

In this section, we explore the splitting formula in some basic examples, and examine its consequences for the logarithmic curve/sheaf correspondence. 

\subsection{Degree zero invariants} The DT invariants counting zero-dimensional ideal sheaves is already very interesting. For a pair $(X|D)$, the series for zero-dimensional subschemes conjecture to be:
\[
\mathsf{Z}_{\mathsf{DT}}\left(X,D;q)\right)_{\beta=0} = M(-q)^{\int_{X} c_{3}(T^\mathsf{log}_X\otimes K^{\mathsf{log}}_{X})},
\]
where
\[
M(q) = \prod_{n\geq1} \frac{1}{(1-q^n)^n}
\]
is the MacMahon function. In particular, if $X$ is a toric threefold and $D$ is its toric boundary the conjecture collapses to the statement
\[
\mathsf{Z}_{\mathsf{DT}}\left(X,D;q)\right)_{0} = 1.
\]
While the general conjecture is presently open, this special case is a pleasant consequence of the degeneration formalism. Consider the degeneration of $X$ to the normal cone of a torus invariant divisor:
\[
\mathcal Y\to \A^1.
\]
The special fiber $Y_0$ consists of two components $X_1$ and $X_2$. Both of them are toric, and in fact, the precise toric structure will not matter in what follows. 

Since we are dealing with $0$-dimensional subschemes, there is no step in the formula for flattening of the evaluation map. There are also no tube subschemes, so the formula simplifies as follows. Fix a value $\chi$ for the holomorphic Euler characteristic. We have an induced degeneration of the moduli space
\[
\mathsf{DT}_{0,n}(\mathcal Y)\to \A^1.
\]
The special fiber decomposes into a union:
\[
\mathsf{DT}_{0,n}(\mathcal Y_0) = \coprod_{n_1+n_2 = n} \mathsf{DT}_{0,n_1}(X_1)\times \mathsf{DT}_{0,n_2}(X_2).
\]
The decomposition is compatible with virtual classes. We can take degree, and then pass to generating functions to get:
\[
\mathsf{Z}_{\mathsf{DT}}\left(X,D;q)\right)_{0} = \mathsf{Z}_{\mathsf{DT}}\left(X_1,D_1;q)\right)_{0}\cdot \mathsf{Z}_{\mathsf{DT}}\left(X_2,D_2;q)\right)_{0}
\]
Now, we recognize that for the appropriate choice of stack of expansion, there is an identification of spaces
\[
\mathsf{DT}_{0,n}(X) \cong \mathsf{DT}_{0,n}(X_1) \cong \mathsf{DT}_{0,n}(X_2), 
\]
compatible with virtual structures. Therefore, writing $F(q)$ for the generating series $\mathsf{Z}_{\mathsf{DT}}(X,D;q)$, we observe that
\[
F(q) = F(q)^2.
\]
The terms are concentrated in non-negative degrees and the constant term is equal to $1$, so the series is invertible and we conclude that $F(q) = 1$. 

\subsection{Conics in $\mathbb P^2\times\PP^1$}\label{sec: conics-example} In this example, we explain what the degeneration formula looks like in a simple example -- counting conics through $5$ points in $\PP^2$. In the following example, we consider a more sophisticated version of this geometry. 

\subsubsection{The invariants} We now consider the threefold $X$ to be $\mathbb P^2\times\PP^1$. We give it the toric logarithmic structure and think of the point on the second factor as being the {\it height}. Let $h$ denote the curve class of the form $\ell \times \{\sf pt\}$, where $\ell$ is a line in $\PP^2$. Set $\beta = 2h$.  Note in particular that this means the intersection numbers with the divisors $\PP^2\times\{0\}$ and $\PP^2\times\{\infty\}$ are $0$. 

Let us describe a few of the subschemes that can appear in this space. The minimum possible value of the Euler characteristic is $\chi = 1$, and this is the case when we have a pure curve contained in a $\PP^2$ slice. On the other hand, the union of two lines at different heights appears in the moduli space where $\beta = 2h$ but $\chi = 2$. Higher Euler characteristic subschemes can also appear by taking a conic together with an embedded point. 

The virtual dimension of the moduli space is equal to $6$ for all values on $\chi$. We can view these roughly as $5$ parameters to choose a conic in $\PP^2$ and another parameter for the height direction. We impose conditions to extract an invariant. The simplest invariants are obtained by forcing point incidences in the $\mathbb P^2$-direction and then imposing a height. We can do this precisely by taking $4$ primary insertions of the 
\[
\{\mathsf{pt}\}\times\PP^1\subset\PP^2\times\PP^1.
\]
and one point insertion.

\begin{remark}[Convention on conditions on $S\times\PP^1$]\label{rem: point-conditions}
 In the discussion, as well as the examples of the form $S\times\PP^1$ that follow, for brevity, the phrase ``$k$ point conditions'' will mean imposing $(k-1)$ conditions of the form $\mathsf{pt}\times\PP^1$ and one genuine point condition. 
\end{remark}

\subsubsection{The target degeneration} We set the degeneration up as follows. Choose five $\mathbb C[\![t]\!]$-points of $\PP^2$, given in homogeneous coordinates by:
\[
p_i(t) = (1\colon z^i_1(t)\colon z^i_2(t)), \ \ \textsf{for}\ \  z^i_j(t)\in \mathbb C[\![t]\!].
\]
We choose the elements $z^i_j(t)$ so that their $t$-valuations are generic, e.g. they give $10$ distinct positive integers. We denote the valuation vectors of these points by $p_i^{\trop}$. We now choose any toric degeneration $\mathcal P^2\to \A^1$ of $\PP^2$ over $\A^1_t$, such that the closures of these points lie in strata of codimension $0$ in the special fiber. 

\subsubsection{The degeneration of the moduli space} We now consider the moduli space $\mathsf{PT}_{2h,\chi}(\cP^2\times\PP^1)$ over $\A^1_t$. The virtual irreducible components of this moduli space in the special fiber are in bijection with rigid Hilbert $1$-complexes. Let us describe some of these. 

The simplest type of rigid $1$-complex is actually a virtually birational copy of the original moduli space $\mathsf{PT}_{2h,\chi}(\PP^2\times\PP^1)$. The rigid $1$-complex is depicted in the figure below
\begin{figure}[h!]

\tikzset{every picture/.style={line width=0.75pt}} 

\tikzset{every picture/.style={line width=0.75pt}} 

\begin{tikzpicture}[x=0.75pt,y=0.75pt,yscale=-1,xscale=1]

\draw [line width=0.75]    (180,362) ;
\draw [shift={(180,362)}, rotate = 0] [color={rgb, 255:red, 0; green, 0; blue, 0 }  ][fill={rgb, 255:red, 0; green, 0; blue, 0 }  ][line width=0.75]      (0, 0) circle [x radius= 2.34, y radius= 2.34]   ;
\draw [line width=0.75]    (232,432) ;
\draw [shift={(232,432)}, rotate = 0] [color={rgb, 255:red, 0; green, 0; blue, 0 }  ][fill={rgb, 255:red, 0; green, 0; blue, 0 }  ][line width=0.75]      (0, 0) circle [x radius= 2.34, y radius= 2.34]   ;
\draw [line width=0.75]    (64,392) ;
\draw [shift={(64,392)}, rotate = 0] [color={rgb, 255:red, 0; green, 0; blue, 0 }  ][fill={rgb, 255:red, 0; green, 0; blue, 0 }  ][line width=0.75]      (0, 0) circle [x radius= 2.34, y radius= 2.34]   ;
\draw [line width=0.75]    (239,535) ;
\draw [shift={(239,535)}, rotate = 0] [color={rgb, 255:red, 0; green, 0; blue, 0 }  ][fill={rgb, 255:red, 0; green, 0; blue, 0 }  ][line width=0.75]      (0, 0) circle [x radius= 2.34, y radius= 2.34]   ;
\draw [line width=0.75]    (118.5,515) ;
\draw [shift={(118.5,515)}, rotate = 0] [color={rgb, 255:red, 0; green, 0; blue, 0 }  ][fill={rgb, 255:red, 0; green, 0; blue, 0 }  ][line width=0.75]      (0, 0) circle [x radius= 2.34, y radius= 2.34]   ;
\draw    (137,356) -- (137,467) ;
\draw    (255,467) -- (137,467) ;
\draw    (48,535) -- (137,467) ;

\draw (38,398) node [anchor=north west][inner sep=0.75pt]    {$p^{\trop}_{2}$};
\draw (185,366) node [anchor=north west][inner sep=0.75pt]    {$p^{\trop}_{1}$};
\draw (93.5,518.5) node [anchor=north west][inner sep=0.75pt]    {$p^{\trop}_{3}$};
\draw (244,538) node [anchor=north west][inner sep=0.75pt]    {$p^{\trop}_{4}$};
\draw (248,418) node [anchor=north west][inner sep=0.75pt]    {$p^{\trop}_{5}$};
\draw (116,358) node [anchor=north west][inner sep=0.75pt]    {$2$};
\draw (223,473) node [anchor=north west][inner sep=0.75pt]    {$2$};
\draw (52,544) node [anchor=north west][inner sep=0.75pt]    {$2$};

\end{tikzpicture}
\caption{A rigid Chow $1$-complex corresponding to a component in the moduli space $\mathsf{PT}_{2h,\chi}(\cP^2_0\times\PP^1)$. The component is virtually birational to $\mathsf{PT}_{2h,\chi}(\PP^2\times\PP^1)$. The vertices corresponding to $p_i^{\trop}$ are components of the expansion that contain the special fibers of $p_i(t)$.}
\end{figure}
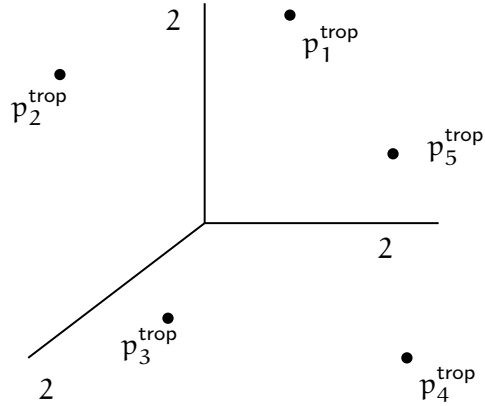
This rigid $1$-complex does not interact with the new components that contain the points $p_i(0)$. We will see shortly that the contributions from such vertices vanish. 

In order to force some interaction with the new points, consider the same $1$-complex as the one above, but translated to the point $p_1^\trop$. If we had not subdivided to include these points, this $1$-complex would not be rigid -- we see that small translations of the $1$-complex would have the same combinatorial type. Indeed, the combinatorial type includes this incidence information, and as a result, the combinatorial type changes under deformation. 

More interestingly, we consider the much more degenerate rigid $1$-complex pictured in Figure~\ref{fig: full-conic}. In this case, the $5$ of the edges/legs of $1$-complex pass through points $p_i^\trop$. 
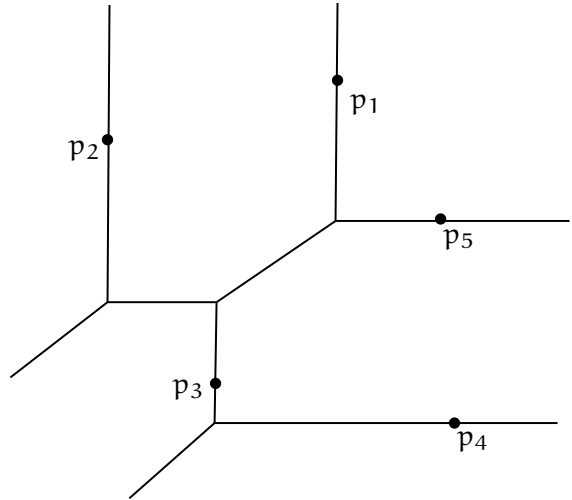
\begin{figure}[h!]

\tikzset{every picture/.style={line width=0.75pt}} 

\begin{tikzpicture}[x=0.75pt,y=0.75pt,yscale=-1,xscale=1]

\draw [line width=0.75]    (495,69) ;
\draw [shift={(495,69)}, rotate = 0] [color={rgb, 255:red, 0; green, 0; blue, 0 }  ][fill={rgb, 255:red, 0; green, 0; blue, 0 }  ][line width=0.75]      (0, 0) circle [x radius= 2.34, y radius= 2.34]   ;
\draw [line width=0.75]    (547,139) ;
\draw [shift={(547,139)}, rotate = 0] [color={rgb, 255:red, 0; green, 0; blue, 0 }  ][fill={rgb, 255:red, 0; green, 0; blue, 0 }  ][line width=0.75]      (0, 0) circle [x radius= 2.34, y radius= 2.34]   ;
\draw [line width=0.75]    (379,99) ;
\draw [shift={(379,99)}, rotate = 0] [color={rgb, 255:red, 0; green, 0; blue, 0 }  ][fill={rgb, 255:red, 0; green, 0; blue, 0 }  ][line width=0.75]      (0, 0) circle [x radius= 2.34, y radius= 2.34]   ;
\draw [line width=0.75]    (554,242) ;
\draw [shift={(554,242)}, rotate = 0] [color={rgb, 255:red, 0; green, 0; blue, 0 }  ][fill={rgb, 255:red, 0; green, 0; blue, 0 }  ][line width=0.75]      (0, 0) circle [x radius= 2.34, y radius= 2.34]   ;
\draw [line width=0.75]    (433.5,222) ;
\draw [shift={(433.5,222)}, rotate = 0] [color={rgb, 255:red, 0; green, 0; blue, 0 }  ][fill={rgb, 255:red, 0; green, 0; blue, 0 }  ][line width=0.75]      (0, 0) circle [x radius= 2.34, y radius= 2.34]   ;
\draw    (495,30) -- (494,140) ;
\draw    (612,140) -- (494,140) ;
\draw    (434,181) -- (494,140) ;
\draw    (434,181) -- (433,242) ;
\draw    (434,181) -- (379,181) ;
\draw    (330,219) -- (379,181) ;
\draw    (380,31) -- (379,181) ;
\draw    (433,242) -- (390,280) ;
\draw    (433,242) -- (606,242) ;

\draw (358,98) node [anchor=north west][inner sep=0.75pt]    {$p_{2}$};
\draw (500,75) node [anchor=north west][inner sep=0.75pt]    {$p_{1}$};
\draw (410.5,218.5) node [anchor=north west][inner sep=0.75pt]    {$p_{3}$};
\draw (554,245) node [anchor=north west][inner sep=0.75pt]    {$p_{4}$};
\draw (547,143) node [anchor=north west][inner sep=0.75pt]    {$p_{5}$};

\end{tikzpicture}
\caption{A rigid Chow $1$-complex corresponding for conics in the degenerate target $\cP^2_0\times\PP^1$.}\label{fig: full-conic}
\end{figure}

Again note that it is only because we have subdivided to include $p_i^\trop$ that these are rigid. For future reference, we refer to this particular combinatorial type where $\gamma$ meets all these points as the {\it tropically distinguished} one. Note that in this geometry, for general valuations as above, there will be a {\it unique} rigid $1$-complex that passes through all the points $p_i^\trop$, also satisfies the balancing condition (all non-balanced ones can be ignored). This follows from standard considerations in tropical geometry, see for instance~\cite{GM07c,Mi03}. 

\subsubsection{Vanishing arguments} Now let $\gamma$ be any rigid Hilbert $1$-complex. We know from the main splitting formula that the virtual class decomposes as:
\[
\left[\mathsf{PT}_{2h,\chi}(\cP_0^2\times\PP^1)\right]^{\sf vir} = \sum_\gamma [\mathsf{PT}_\gamma]^{\sf vir}.
\]
We now impose the point conditions. One can easily see that if a point constraint $p_i(0)$ is contained a component $X_v$ of the special fiber that is not also a vertex of the rigid tropical curve $\gamma$, then the subschemes in the space $\mathsf{PT}_\gamma$ are all disjoint from this point. Therefore, admitting the tropical argument above, after imposing the point conditions, the only relevant summand in the equation is above occurs when $\gamma$ is the tropically distinguished one. 

\subsubsection{The splitting} We now compute the invariant by taking apart the special fiber. In this case, no blowups are necessary to flatten the evaluation map. 

Running the degeneration formula requires splitting the diagonal. After doing this, we find that the DT invariant for conics is the fourth power of the PT invariant for lines in $\PP^2\times\PP^1$ through two points, so we find
\[
\mathsf{Z}_{\mathsf{PT}}(X,D;q)_{2h} = (\mathsf{Z}_{\mathsf{PT}}(X,D;q)_h)^4.
\]
The invariant $\mathsf{Z}_{\mathsf{PT}}(X,D;q)_h$ can in turn be calculated from known invariants. The underlying cycle of a stable pair with curve class $h$ is necessarily a smooth $\PP^1$. The moduli space is therefore isomorphic to the PT moduli space of the local geometry $\mathcal O_{\PP^1}\oplus\mathcal O_{\PP^1}(1)$ over $\PP^1$, relative to three fibers. The formulas (via the normalized DT series) can be found in~\cite{OP10}. A similar argument applies for the usual DT series, to give the same relation. 

\begin{remark}[A better way of imposing point conditions]
In the setting above, when we impose point condition, we have taken a ``by the book'' approach. A more direct approach would be to work with a cut down moduli space $\mathsf{PT}^p{\beta,\chi}(\mathcal Y)$ which imposes the incidence conditions to the points $p_i(t)$ on the whole family. One can then just use an adapted degeneration formula in this setting that directly relates the cut down moduli stacks. If set up in this way, {\it all} $1$-complexes are required to pass through the points $p_i^\trop$. The arguments of the paper carry over this cut down setting with cosmetic changes. 
\end{remark}

\subsection{Flattening and exotic invariants} In the discussion above, a very simple form of the gluing formula holds. Let us explain where one should ``look'' for correction terms. Along the way, we will see the need for exotic insertions. 

Recall the setup of the degeneration package. We have a degeneration $\cY\to B$ with special fiber $Y_0$. There is an associated degeneration of the DT moduli space
\[
\mathsf{DT}_{\beta,\chi}(\cY)\to B,
\]
and the special fiber $\mathsf{DT}_{\beta,\chi}(Y_0)$ consists of a number of {\it virtual irreducible components} indexed by rigid Hilbert $1$-complexes $\gamma$. We denote them $\mathsf{DT}_\gamma(Y_0)$. We then have a map
\[
\prod_v \mathsf{DT}_{\gamma(v)}(X_v)\to \prod_e \mathsf{Hilb}^{\gamma(e)}(D_e)^2 = \prod_v \mathsf{Hilb}^{\gamma(e)}(D_v).
\]
The equality on the right is a regrouping according to flags incident to vertices. The cycle-level degeneration formula then states that if we choose tropical models so that the induced map on cone spaces is combinatorially flat, then the pullback of the strict transform of the diagonal locus on the right is equal to the virtual class $[\mathsf{DT}_{\beta,\chi}(Y_0)]^{\sf vir}$, up to explicit combinatorial multiplicities. To summarize, we note that:
\begin{enumerate}[(i)]
\item The tropical model of the product of Hilbert schemes {\it over edges} on the right may not be a product of tropical models, however
\item the tropical model can be chosen such that the product of Hilbert schemes {\it over vertices} on the right is still a product. 
\end{enumerate}

We now give two combinatorial examples of evaluation maps that exhibit (i) the non-flatness of evaluations, and (ii) the need to flatten ``jointly'', i.e. that the edge-by-edge product structure cannot be maintained. The latter phenomenon is responsible for the exotic invariants. 

\subsubsection{Single divisor phenomena} We exposit and example due to Kennedy-Hunt, in his study of PT spaces on toric surfaces relative to the full toric boundary~\cite{KH21}. We can again work in dimension $2$ to see this phenomenon, and in fact, we can assume that there are no embedded points. 

Let $S$ be the toric surface whose rays are given by the positive real spans:
\[
\rho_1 = \RR_{\geq 0}(1,0), \ \ \rho_2 = \RR_{\geq 0}(-1,-1), \ \ \rho_3 = \RR_{\geq 0}(1,-1).
\]
We can determine a curve class $\beta$ by imposing intersection numbers with the three associated divisors
\[
\beta\cdot D_1 = 4, \ \ \beta\cdot D_2 = \beta\cdot D_3 = 2.
\]
We let $\chi$ be the minimum possible value of the holomorphic Euler characteristic, and then study $\mathsf{PT}_{\beta,\chi}(S)$. There is an evaluation map
\[
-\cap D_1\colon \mathsf{PT}_{\beta,\chi}(S)\to \mathsf{Hilb}^4(D_1).
\]
Kennedy-Hunt observes that this map is not combinatorially flat. Indeed, the cone complex associated to the source contains a cone parameterizing tropical curves of type shown in Figure~\ref{non-flat-evaluation} below:
\begin{figure}

\tikzset{every picture/.style={line width=0.75pt}} 

\tikzset{every picture/.style={line width=0.75pt}} 

\begin{tikzpicture}[x=0.75pt,y=0.75pt,yscale=-1,xscale=1]

\draw    (108,502) -- (182,429.5) ;
\draw    (109,359) -- (182,429.5) ;
\draw    (263,430) -- (182,429.5) ;
\draw [line width=0.75]    (182,429.5) ;
\draw [shift={(182,429.5)}, rotate = 0] [color={rgb, 255:red, 0; green, 0; blue, 0 }  ][fill={rgb, 255:red, 0; green, 0; blue, 0 }  ][line width=0.75]      (0, 0) circle [x radius= 2.34, y radius= 2.34]   ;
\draw [line width=0.75]    (150.5,460) ;
\draw [shift={(150.5,460)}, rotate = 0] [color={rgb, 255:red, 0; green, 0; blue, 0 }  ][fill={rgb, 255:red, 0; green, 0; blue, 0 }  ][line width=0.75]      (0, 0) circle [x radius= 2.34, y radius= 2.34]   ;
\draw [line width=0.75]    (151,399.75) ;
\draw [shift={(151,399.75)}, rotate = 0] [color={rgb, 255:red, 0; green, 0; blue, 0 }  ][fill={rgb, 255:red, 0; green, 0; blue, 0 }  ][line width=0.75]      (0, 0) circle [x radius= 2.34, y radius= 2.34]   ;
\draw    (151,399.75) -- (150.5,460) ;
\draw    (151,399.75) -- (263,400) ;
\draw    (265,459) -- (150.5,460) ;
\draw    (291,430) -- (356,430) ;
\draw [shift={(358,430)}, rotate = 180] [color={rgb, 255:red, 0; green, 0; blue, 0 }  ][line width=0.75]    (10.93,-3.29) .. controls (6.95,-1.4) and (3.31,-0.3) .. (0,0) .. controls (3.31,0.3) and (6.95,1.4) .. (10.93,3.29)   ;
\draw    (388,430.5) ;
\draw [shift={(388,430.5)}, rotate = 0] [color={rgb, 255:red, 0; green, 0; blue, 0 }  ][fill={rgb, 255:red, 0; green, 0; blue, 0 }  ][line width=0.75]      (0, 0) circle [x radius= 3.35, y radius= 3.35]   ;
\draw    (388,400) ;
\draw [shift={(388,400)}, rotate = 0] [color={rgb, 255:red, 0; green, 0; blue, 0 }  ][fill={rgb, 255:red, 0; green, 0; blue, 0 }  ][line width=0.75]      (0, 0) circle [x radius= 3.35, y radius= 3.35]   ;
\draw    (388,459) ;
\draw [shift={(388,459)}, rotate = 0] [color={rgb, 255:red, 0; green, 0; blue, 0 }  ][fill={rgb, 255:red, 0; green, 0; blue, 0 }  ][line width=0.75]      (0, 0) circle [x radius= 3.35, y radius= 3.35]   ;
\draw    (388,505.5) -- (388,355.5) ;

\draw (298,439) node [anchor=north west][inner sep=0.75pt]    {$\cap \ D_{1}$};

\end{tikzpicture}
\caption{A combinatorial type exhibiting the non-flatness of the evaluation map from the PT moduli space.}\label{non-flat-evaluation}
\end{figure}
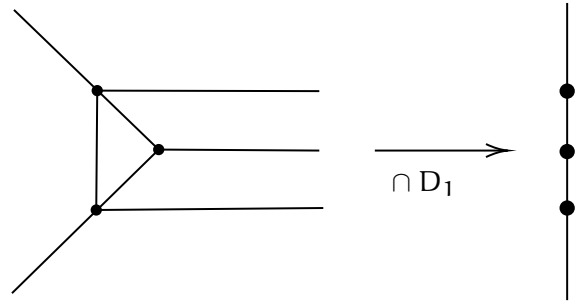
This cone maps, in codomain, to the induced tropical subdivision given by placing a vertical line far to the right of the picture and intersecting with $3$ rays parallel to the $x$-axis. This induces an expansion on $\mathbb R$. The target cone is $2$-dimensional. However, if for example, we fix the middle of the $3$ rays, the position of one of the rays determines the position of the other, because the base vertices of these rays form a cycle. 

\subsubsection{Coupling phenomena} While the above non-flatness illustrates the need to blow up, it does not show why exotic insertions are necessary. For that we consider the following example. We let $X = \PP^2\times\PP^1$ with its toric logarithmic structure. Building on the conics case we looked at before, we consider cubics in the $\PP^2$-direction, i.e. $\beta = h\otimes {\sf pt}$. 

Let $H_0$, $H_1$, and $H_2$ be the pullbacks of the toric lines in $\PP^2$ and $D_0$ and $D_\infty$ the pullbacks of the toric points on $\PP^1$. We impose maximal tangency conditions, i.e. the partition $(3)$ along the $H_0$ and $H_1$, simple tangency, i.e. the partition $(2+1)$ along $H_2$ and of course have tangency $0$ along the $D_j$. 

The typical member of this moduli space is a plane cubic, at fixed height. The Euler characteristic of such a subscheme is $0$. Higher Euler characteristic subschemes can be obtained, for example by adding embedded points. The virtual dimension of the moduli space is $9$. The tangency imposes $5$ conditions. We then impose two bulk incidences of the form ${\sf pt}\otimes\PP^1$ and one point insertion to fix the ``height''. 

With these conditions imposed, we can choose a degeneration of the $\PP^2$ as we did for conics, and specialize the points into different components. Doing this, the following rigid Chow $1$-complex appears, shown in Figure~\ref{fig: triangle-example}
\begin{figure}

\tikzset{every picture/.style={line width=0.75pt}} 

\begin{tikzpicture}[x=0.75pt,y=0.75pt,yscale=-1,xscale=1]

\draw [line width=0.75]    (106,520.5) ;
\draw [shift={(106,520.5)}, rotate = 0] [color={rgb, 255:red, 0; green, 0; blue, 0 }  ][fill={rgb, 255:red, 0; green, 0; blue, 0 }  ][line width=0.75]      (0, 0) circle [x radius= 2.34, y radius= 2.34]   ;
\draw    (348,469) -- (298,488.5) ;
\draw    (79,554) -- (79.5,449) ;
\draw    (79,554) -- (161,554) ;
\draw    (79.5,449) -- (161,554) ;
\draw    (79,553) -- (106,520.5) ;
\draw    (79.5,448) -- (106,520.5) ;
\draw    (161,554) -- (106,520.5) ;
\draw [line width=0.75]    (79,553) ;
\draw [shift={(79,553)}, rotate = 0] [color={rgb, 255:red, 0; green, 0; blue, 0 }  ][fill={rgb, 255:red, 0; green, 0; blue, 0 }  ][line width=0.75]      (0, 0) circle [x radius= 2.34, y radius= 2.34]   ;
\draw [line width=0.75]    (161,554) ;
\draw [shift={(161,554)}, rotate = 0] [color={rgb, 255:red, 0; green, 0; blue, 0 }  ][fill={rgb, 255:red, 0; green, 0; blue, 0 }  ][line width=0.75]      (0, 0) circle [x radius= 2.34, y radius= 2.34]   ;
\draw [line width=0.75]    (79.5,449) ;
\draw [shift={(79.5,449)}, rotate = 0] [color={rgb, 255:red, 0; green, 0; blue, 0 }  ][fill={rgb, 255:red, 0; green, 0; blue, 0 }  ][line width=0.75]      (0, 0) circle [x radius= 2.34, y radius= 2.34]   ;
\draw    (348,469) -- (326,516) ;
\draw    (326,516) -- (298,488.5) ;
\draw    (298,488.5) -- (216,488.5) ;
\draw    (326,516) -- (326,585) ;
\draw    (348,469) -- (394,420.5) ;
\draw [line width=0.75]    (298,488.5) ;
\draw [shift={(298,488.5)}, rotate = 0] [color={rgb, 255:red, 0; green, 0; blue, 0 }  ][fill={rgb, 255:red, 0; green, 0; blue, 0 }  ][line width=0.75]      (0, 0) circle [x radius= 2.34, y radius= 2.34]   ;
\draw [line width=0.75]    (326,516) ;
\draw [shift={(326,516)}, rotate = 0] [color={rgb, 255:red, 0; green, 0; blue, 0 }  ][fill={rgb, 255:red, 0; green, 0; blue, 0 }  ][line width=0.75]      (0, 0) circle [x radius= 2.34, y radius= 2.34]   ;
\draw [line width=0.75]    (348,469) ;
\draw [shift={(348,469)}, rotate = 0] [color={rgb, 255:red, 0; green, 0; blue, 0 }  ][fill={rgb, 255:red, 0; green, 0; blue, 0 }  ][line width=0.75]      (0, 0) circle [x radius= 2.34, y radius= 2.34]   ;

\draw (390,434) node [anchor=north west][inner sep=0.75pt]    {$( 2+1)$};
\draw (36,603) node [anchor=north west][inner sep=0.75pt]   [align=left] {Newton subdivision};
\draw (238,601) node [anchor=north west][inner sep=0.75pt]   [align=left] {Dual polyhedral decomposition};
\draw (338,565) node [anchor=north west][inner sep=0.75pt]    {$( 3)$};
\draw (212,455) node [anchor=north west][inner sep=0.75pt]    {$( 3)$};
\draw (126,478) node [anchor=north west][inner sep=0.75pt]    {$H_{2}$};
\draw (36,492) node [anchor=north west][inner sep=0.75pt]    {$H_{1}$};
\draw (101,558) node [anchor=north west][inner sep=0.75pt]    {$H_{0}$};

\end{tikzpicture}
\caption{The picture on the left is the dual Newton subdivision and the picture on the right is the dual Chow $1$-complex. The picture has been decorated also with the choice of partition along the legs.}\label{fig: triangle-example}
\end{figure}
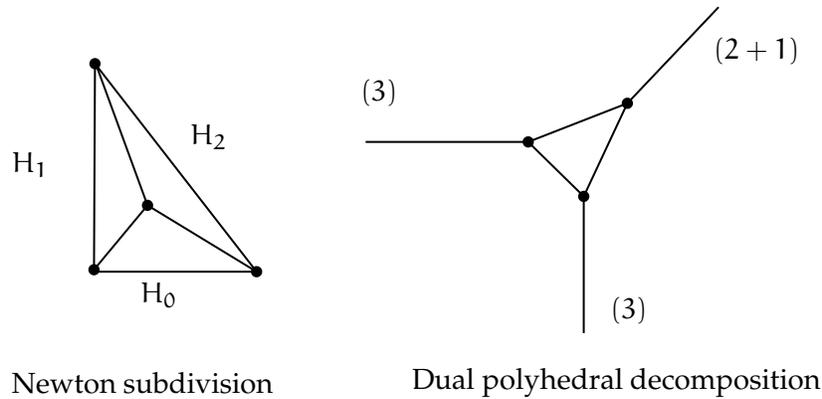
In the splitting formula, we start with evaluation maps
\[
\prod_v\mathsf{DT}_{\beta,\chi}(X_v)\to \prod_e {\sf Hilb}(D_e)^2,
\]
with the products taken over the three vertices in the domain, and three edges in the codomain. The individual tropical evaluation maps
\[
\mathsf{DT}_{\beta,\chi}(X_v)\to {\sf Hilb}(D_e)
\]
for $v$ incident to $e$ are easily seen to be flat. Indeed in this case the cone complex of the codomain has dimension $1$ so the flatness is automatic. However, the consolidated map
\[
\mathsf{DT}_{\beta,\chi}(X_v)\to {\sf Hilb}(D_{e_1})\times {\sf Hilb}(D_{e_2})
\]
is not flat. This forces us, at least in our proof, to consider blowups of this product that are not products of blowups. This is the source of the exotic evaluations. 

\begin{remark}[Relevance of cycles in the $1$-complex]
If the rigid $1$-complex is a tree and the logarithmic rank of $X$ is $2$, even if the consolidated evaluation maps are not flat, one has a fairly simple degeneration formula, where blowups of the evaluation space can essentially be avoided. In GW theory this is recorded in~\cite[Section~6]{R19} and a similar formula holds in DT theory. More generally, if the rigid $1$-complex is a tree in any logarithmic rank, then one can preserve the product structure on the evaluation space, although blowups of the factors must be allows. Both these facts are based on the fact that for a tree, one can choose an orientation on the $1$-complex in such a way that all unbounded edges are directed towards their incident vertex, and the orientation can be extended such that at every vertex there is exactly one outgoing edge. By flowing along such an orientation, gluing can be performed ``one edge at a time''. Once cycles are present in the $1$-complex this is no longer possible.
\end{remark}





\subsection{Trivalent vertices and constraints via degeneration} We introduce a basic class of DT/PT invariants of toric threefolds relative to the full toric boundary, which call the {\it trivalent vertices}. Based on known results in the Gromov--Witten theory of exploded manifolds, it seems reasonable to expect that the trivalent vertices determine {\it all} primary DT/PT invariants in this geometry, see for instance~\cite{Par17}. In particular, it may be possible to reduce the primary GW/DT/PT equivalence for toric threefold pairs to matching the trivalent vertices. We pursue this elsewhere. 


\subsubsection{Setup} In what follows, we specify a projective toric surface $S$ be giving the rays of its fan $\Sigma_S$, and specify a curve class $\beta$ by giving its intersection multiplicity with each invariant divisor, i.e. an integer for each ray. 

We study DT/PT invariants on $X = S\times\PP^1$, with the boundary coming from the toric divisors. We will mainly be interested in classes of the form $\beta\otimes {\sf pt}$.

To fix numerical data for DT/PT invariants we also fix a cohomology weighted partition for each divisor $D_\rho$ in the fan, as well as ``bulk'' primary insertions from $S\times\PP^1$. We are interested in point insertions in the sense of Remark~\ref{rem: point-conditions}, namely finitely many conditions of the form $\mathsf{pt}\times\PP^1$ and one point condition in the threefold $S\times\PP^1$.  

\begin{remark}[Connected vs. disconneced]
If the insertions above are imposed in logarithmic Gromov--Witten theory, the connected and disconnected theories coincide by elementary arguments. Therefore our conjectures, which concern the disconnected theory, relate the invariants considered in~\cite{Bou17}, which a priori the connected theory, with the logarithmic DT invariants defined in~\cite{MR20}. 
\end{remark}

\subsubsection{The trivalent vertices} An interesting class of DT (or PT) invariants of this kind arise when $S$ is a weighted projective plane. Given positive integers $a$ and $b$, we can take $S_{a.b}$ to be the toric surface whose fan has rays:
\[
\rho_1 = \RR_{\geq 0}(a,0), \ \ \rho_2 = \RR_{\geq 0}(b,0), \ \ \rho_3 = \RR_{\geq 0}(-a,-b).
\] 
Note that different $a$ and $b$ can determine the same surface. The intersection multiplicities with the corresponding divisors $D_1,D_2,D_3$ are now taken to be $a$, $b$, and $\mathsf{gcd}(a,b)$. This determines $\beta$.

We define a DT invariant $N_{(a,b)}(q)$ in class $\beta$ with maximal tangency order along each $D_i$, two point insertions (with the notational convention above), and summing over all possible Euler characteristics. 

\subsubsection{A proposed relation among trivalent vertices} We have explained why the trivalent vertices are interesting. We now illustrate how the degeneration formula can constrain and calculate the invariants. The basic principle is that because different degenerations of incidence conditions must give the same answers, one can extract relations by comparing them. 

We propose and examine the following equality:
\[
N_{(2,2)}(q) \stackrel{?}{=} N_{(4,1)}(q).
\]
The $q = 1$ specializations can be checked to be equal, e.g. this is one consequence of Mikhalkin's correspondence theorem~\cite{Mi03}. 

A feature of the proposal is that the moduli spaces of stable pairs in the respective curve classes are very different. There is no obvious {\it geometric} reason why these invariants should be the same. 

\subsubsection{Constraining the proposal via degeneration} We proceed by computing another -- non-trivalent invariant -- in two different ways using the splitting formula, so the equality forces the relation above. This method of using consistency of the degeneration formula appears prominently in~\cite{Bou17,Par17}. 

Consider the toric surface $S_{(2,3)}$. In the notation above, the intersection multiplicities are
\[
\beta\cdot D_1 = 2, \ \ \beta_2\cdot D_2 = 3, \ \ \beta\cdot D_3 = 1.
\]
We impose $3$ point conditions, and in tangency conditions as follows: take maximal tangency, i.e. the partition (2) with $D_1$, simple tangency, i.e. the partition (2+1) with $D_2$, and necessarily tangency $1$ with $D_3$. 

As in the example from Section~\ref{sec: conics-example}, for different configurations of the three points, we get different rigid Chow $1$-complexes. The two possibilities are shown below:
\begin{figure}

\tikzset{every picture/.style={line width=0.75pt}} 

\begin{tikzpicture}[x=0.75pt,y=0.75pt,yscale=-1,xscale=1]

\draw    (159.5,678) -- (160,740) ;
\draw    (219.5,740) -- (160,740) ;
\draw    (160,740) -- (130.5,771) ;
\draw    (130.5,678) -- (130.5,771) ;
\draw    (130.5,771) -- (107.5,812) ;
\draw    (372.5,677) -- (373,739) ;
\draw    (432.5,739) -- (373,739) ;
\draw    (373,739) -- (332.5,761) ;
\draw    (332.5,678) -- (332.5,761) ;
\draw    (332.5,761) -- (309.5,802) ;
\draw [line width=0.75]    (159.75,709) ;
\draw [shift={(159.75,709)}, rotate = 0] [color={rgb, 255:red, 0; green, 0; blue, 0 }  ][fill={rgb, 255:red, 0; green, 0; blue, 0 }  ][line width=0.75]      (0, 0) circle [x radius= 2.34, y radius= 2.34]   ;
\draw [line width=0.75]    (205,740) ;
\draw [shift={(205,740)}, rotate = 0] [color={rgb, 255:red, 0; green, 0; blue, 0 }  ][fill={rgb, 255:red, 0; green, 0; blue, 0 }  ][line width=0.75]      (0, 0) circle [x radius= 2.34, y radius= 2.34]   ;
\draw [line width=0.75]    (416,739) ;
\draw [shift={(416,739)}, rotate = 0] [color={rgb, 255:red, 0; green, 0; blue, 0 }  ][fill={rgb, 255:red, 0; green, 0; blue, 0 }  ][line width=0.75]      (0, 0) circle [x radius= 2.34, y radius= 2.34]   ;
\draw [line width=0.75]    (372.75,708) ;
\draw [shift={(372.75,708)}, rotate = 0] [color={rgb, 255:red, 0; green, 0; blue, 0 }  ][fill={rgb, 255:red, 0; green, 0; blue, 0 }  ][line width=0.75]      (0, 0) circle [x radius= 2.34, y radius= 2.34]   ;
\draw [line width=0.75]    (332.5,719.5) ;
\draw [shift={(332.5,719.5)}, rotate = 0] [color={rgb, 255:red, 0; green, 0; blue, 0 }  ][fill={rgb, 255:red, 0; green, 0; blue, 0 }  ][line width=0.75]      (0, 0) circle [x radius= 2.34, y radius= 2.34]   ;
\draw [line width=0.75]    (130,719) ;
\draw [shift={(130,719)}, rotate = 0] [color={rgb, 255:red, 0; green, 0; blue, 0 }  ][fill={rgb, 255:red, 0; green, 0; blue, 0 }  ][line width=0.75]      (0, 0) circle [x radius= 2.34, y radius= 2.34]   ;

\draw (416,746) node [anchor=north west][inner sep=0.75pt]    {$( 2,0)$};
\draw (198,746) node [anchor=north west][inner sep=0.75pt]    {$( 2,0)$};
\draw (166,677) node [anchor=north west][inner sep=0.75pt]    {$( 0,2)$};
\draw (378.5,676) node [anchor=north west][inner sep=0.75pt]    {$( 0,1)$};
\draw (78.5,678) node [anchor=north west][inner sep=0.75pt]    {$( 0,1)$};
\draw (106.5,817) node [anchor=north west][inner sep=0.75pt]    {$( -2,-3)$};
\draw (311.5,812) node [anchor=north west][inner sep=0.75pt]    {$( -2,-3)$};
\draw (283,676) node [anchor=north west][inner sep=0.75pt]    {$( 0,2)$};

\end{tikzpicture}
\label{Contributing rigid Chow $1$-complexes to the invariant above on the toric surface $S_{(2,3)}$. }
\end{figure}
The invariants are defined independently of any degeneration, so they are obviously equal. In fact, this already provides a constraint between $N_{(2,2)}(q)$ and $N_{(4,1)}(q)$, though at this stage it is not clear how concrete this is. We argue that it is, in fact, extremely concrete. 

For each of the two rigid $1$-complexes in the figure above, we can try to apply the splitting formula of Corollary~\ref{cor: DT-deg-form} by cutting along the unique bounded edge. In particular, one has to sum over partitions of the multiplicity at the bounded edge, coming from the decomposition of the strict diagonal -- let us call this the {\it gluing partition}. 

In the first case on the left, this weight is $2$ so we need to consider the case where this partition is $(2)$ and where it is $(1+1)$. Note that in the case where the gluing partition encodes maximal tangency, the gluing formula holds without corrections. It turns out that when the gluing partition is $(1+1)$ the contribution to the degeneration formula vanishes. This is essentially due to the presence of the $\mathbb P^1$-direction. A detailed argument will appear in a future paper. For now we explore the consequences, so the reader can take the rest of the discussion as conditional if they wish. 


With this at hand, the gluing formula can be applied without any correction terms, because only maximal tangency along the divisor contributes. Since the two degenerations are calculating the same thing. Using the $GL_2(\ZZ)$ invariance of the data (i.e. applying toric automorphisms) we have:
\[
N_{(2,2)}(q)\cdot N_{(2,1)}(q) = N_{(2,1)}(q)\cdot N_{(4,1)}(q).
\]
The terms $N_{(2,1)}(q)$ are nonzero; these can be seen by explicit calculation in the minimal Euler characteristic case. We have the claimed equality. 

As noted, the quantity $N_{(4,1)}$ can be calculated -- it is a PT invariant of a local curve, relative to three fibers. The explicit formula for the normalized DT series can again be found in~\cite{OP10}. The GW/PT equivalence holds for $N_{4,1}(q)$ because no cycle in the moduli space of curves/maps interacts with the codimension $2$ strata -- it follows from toric smooth pair arguments. This verifies the conjecture for $N_{2,2}(q)$, which is a genuinely new logarithmic invariant. 

\begin{remark}
One can ask whether the quantity $N_{a,b}(q)$ depends only on $ab$. We expect that this is true, and it gives a highly nontrivial relation among PT invariants, and for example, would completely determine the trivalent vertices. In GW theory this was done by Parker (and a closely related problem was considered by Bousseau). A different approach to the trivalent vertices in Gromov--Witten theory, which can also handle the more generally higher valency vertices with descendants, has been carried out by Kennedy-Hunt, Shafi, and Urundolil Kumaran using double ramification cycle techniques~\cite{KHSUK}. While their ability to handle descendants is appealing, it is less clear how to transfer this to DT/PT theory. On the other hand, if the correspondence can be proved independently, these techniques can in principle be used to compute descendant PT invariants using moduli of curves techniques. 
\end{remark}


%

\section{Combinatorial splitting formula}\label{sec: combinatorial-degeneration}

The goal of the rest of the paper is devoted to proving the splitting formulas, following the roadmap of Section~\ref{sec: roadmap}.

In this section, we give a purely combinatorial version of the formula. Under the equivalence between Artin fans and cone stacks, it describes the components of the special fiber of the stack $\mathsf{Exp}(\mathcal Y/B)$ in terms of stacks of the form $\mathsf{Exp}(X_v)$. The argument for combinatorial splitting is insensitive to the labelling of vertices and edges with discrete data of curve classes or Euler characteristics.  We will explain the argument for $1$-complexes, noting that the decorated cases are essentially identical. 

\subsection{Setup and notation}\label{sec:setupandnotation} We have a map of cone complexes
\[
\pi: \Sigma_{\mathcal Y}\to \RR_{\geq 0}.
\]
Let $r_1,\ldots, r_k$ be the rays of $\Sigma_{\mathcal Y}$ that map surjectively onto $\RR_{\geq 0}$.  These correspond to the irreducible components of the normal crossings variety $Y_0$.

We have fixed a cone structure on the space of vertical $1$-complexes and a universal family $\mathbf G$ over it. Given a point $p$ in $T(\Sigma_{Y})$ we will write $\mathbf G_p$ for the associated $1$-complex. 

In our discussion of the degeneration package, we have fixed a ray $\rho$ of $T(\Sigma_{\mathcal Y})$ that maps surjectively onto the base $\RR_{\geq 0}$.  Furthermore, we have already subdivided $\Sigma_{mathcal Y}$ so that the associated $1$-complex has no tube vertices and that every vertex occurring in this $1$-complex is associated to one of the rays $r_i$.

We record a few observations:

\begin{enumerate}[(i)]
\item If $p$ is a point on $\rho$ lying over $t$ in $\RR_{\geq 0}$, the $1$-complex at $p$ can be identified with a union of vertices and edges in the $1$-skeleton of the fiber $\pi^{-1}(t)$ of $\Sigma_{\mathcal Y}$ over $t$. 
\item For any $t\in \RR_{>0}$, the partially ordered sets of faces of $\pi^{-1}(t)\subset \Sigma_{\mathcal Y}$ is independent of $t$. 
\item For any point $p$ in $\rho$ lying over $t$ in $\RR_{>0}$, the subset of faces $\pi^{-1}(t)$ whose union is $\mathbf G_p$ is independent of $t$, under the identification above. 
\end{enumerate}

Over different points of $\RR_{>0}$ the fibers of $\Sigma_{\mathcal Y}$ differ by uniformly dilating the positions of the vertices.  Similarly, over a ray of $T(\Sigma_{\mathcal Y})$, the corresponding $1$-complexes differ by dilation.

\begin{figure}[h!]

\tikzset{every picture/.style={line width=0.75pt}} 

\begin{tikzpicture}[x=0.75pt,y=0.75pt,yscale=-0.6,xscale=0.6]

\draw    (340,189) -- (428,189) ;
\draw    (428,189) -- (428,270) ;
\draw    (428,189) -- (497,134) ;
\draw    (650,82) -- (562,82) ;
\draw    (562,82) -- (562,1) ;
\draw    (562,82) -- (493,137) ;
\draw    (193.57,157) -- (237,157) ;
\draw    (237,157) -- (237,199.46) ;
\draw    (237,157) -- (268.51,130.29) ;
\draw    (346,101.46) -- (302.57,101.46) ;
\draw    (302.57,101.46) -- (302.57,59) ;
\draw    (302.57,101.46) -- (268.51,130.29) ;
\draw    (40,341) -- (639,338.01) ;
\draw [shift={(641,338)}, rotate = 179.71] [color={rgb, 255:red, 0; green, 0; blue, 0 }  ][line width=0.75]    (10.93,-3.29) .. controls (6.95,-1.4) and (3.31,-0.3) .. (0,0) .. controls (3.31,0.3) and (6.95,1.4) .. (10.93,3.29)   ;
\draw    (-2,122.54) -- (41.43,122.54) ;
\draw    (40.57,122.46) -- (40.57,164.91) ;
\draw    (84,122.46) -- (40.57,122.46) ;
\draw    (40.57,122.46) -- (40.57,80) ;
\draw  [loosely dotted]  (41.43,122.54) -- (657,73) ;
\draw  [loosely dotted]  (40.57,122.46) -- (625,223) ;
\draw  [line width=3] [line join = round][line cap = round] (41,122) .. controls (41,122) and (41,122) .. (41,122) ;
\draw  [line width=2.25] [line join = round][line cap = round] (303,101) .. controls (303,101) and (303,101) .. (303,101) ;
\draw  [line width=2.25] [line join = round][line cap = round] (237,157) .. controls (237,157) and (237,157) .. (237,157) ;
\draw  [line width=2.25] [line join = round][line cap = round] (428,189) .. controls (428,189) and (428,189) .. (428,189) ;
\draw  [line width=2.25] [line join = round][line cap = round] (562,82) .. controls (562,82) and (562,82) .. (562,82) ;
\draw  [line width=2.25] [line join = round][line cap = round] (40,341) .. controls (40,341) and (40,341) .. (40,341) ;
\draw  [line width=2.25] [line join = round][line cap = round] (238,340) .. controls (238,340) and (238,340) .. (238,340) ;
\draw  [line width=2.25] [line join = round][line cap = round] (429,339) .. controls (429,339) and (429,339) .. (429,339) ;

\draw (595,231) node [anchor=north west][inner sep=0.75pt]    {$r_{1}$};
\draw (604,48) node [anchor=north west][inner sep=0.75pt]    {$r_{2}$};

\end{tikzpicture}
\caption{A schematic of the slices of of $\Sigma_{\mathcal Y}$ over different points of the base $\RR_{\geq 0}$.}
\end{figure}
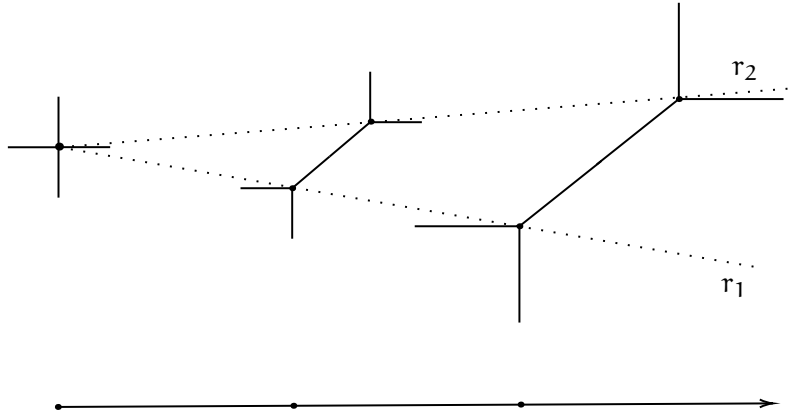

\subsection{Star of a cone} 
Let $\Sigma$ be a cone complex and $\sigma$ a cone in it.  We recall from Section \ref{sec:star-of-cone}
that the \textit{star} of $\sigma$, denoted $\Sigma(\sigma)$, is a cone complex obtained by taking the relative interiors of all cones adjacent to $\sigma$ and quotienting by $\sigma$.

The motivation for considering the star fan construction is the following.  If $V$ is a simple normal crossings pair with cone complex $\Sigma$ and $W\subset V$ is a closed stratum associated to a cone $\sigma$, then $W$ carries a generically trivial logarithmic structure associated to the divisor on $W$ obtained by intersection with the divisorial boundary of $V$.   The star fan $\Sigma(\sigma)$ is precisely the tropicalization of $W$ with respect to this log structure.

We are interested in two instances of this star construction.  First, given $\Sigma_{\mathcal Y}$ and its $r_1,\ldots, r_k$ lying over $\RR_{\geq 0}$, the stars
$\Sigma_{\mathcal Y}(r_i)$ are the cone complexes of the irreducible components of $Y_0$ with the induced generically trivial logarithmic structure. 

The second involves the DT moduli space itself.  The standard degeneration formula involves (i) choosing a stratum and (ii) describing it terms of smaller moduli spaces. We describe the logarithmic structure on a component of the special fiber of the DT moduli space degeneration. However, we would like to work with the generically trivial logarithmic structure, and so again use the star fan of a ray. 

More precisely, given our fixed ray $\rho$ of 
$T(\Sigma_{\mathcal Y})$ that maps surjectively onto the base $\RR_{\geq 0}$, we consider the star fan $T(\Sigma_{\mathcal Y})(\rho)$.  This is a cone space which will keep track of the tropical data associated to a virtual irreducible component of the DT moduli space.

\subsection{Cutting map} 

Our goal is to describe the star fan $T(\Sigma_{\mathcal Y})(\rho)$ in terms of the moduli space of $1$-complexes on the various star fans $\Sigma_{\mathcal Y}(r_i)$.
From now on, we use $\Sigma_1,\ldots \Sigma_k$ to denote these fans, i.e. we have set 
\[ \Sigma_i = \Sigma_{\mathcal Y}(r_i).\] 
We will construct cutting maps
\[
\kappa_i: |T(\Sigma_{\mathcal Y})(\rho)| \to |T^\circ(\Sigma_i)| \subset |T(\Sigma_i)|
\]
from the star of $\rho$ to a certain subspace $T^\circ$ of the moduli space of $1$-complexes in the star fans $\Sigma_i$, which we will specify in the course of the construction. 

In order to do this, we first discuss some preliminaries.
Fix a cone $\tau$ of the space $T(\Sigma_{\mathcal Y})$. We restrict the universal family to $\tau$ to obtain a family of $1$-complexes, compatible with the map to $\RR_{\geq 0}$:
\[
\begin{tikzcd}
\mathbf G_\tau\arrow{r}\arrow{d} & \Sigma_{\mathcal Y}\arrow{d}\\
\tau\arrow{r} & \RR_{\geq 0}. 
\end{tikzcd}
\]
Fix a point $p$ in the interior of $\tau$. The fiber of $p$ in $|\mathbf G_\tau|$ can be identified with a $1$-complex. 

\begin{lemma}\label{lem:combtype}
The combinatorial type of the polyhedral complex over $p$ is independent of the choice of $p$, provided $p$ lies in the interior of $\tau$. 
\end{lemma}

\begin{proof}
The combinatorial structure can be described independently of $p$ as follows. Since $\mathbf G_\tau$ is combinatorially flat over $\tau$, its cones either map isomorphically onto the base or drop dimension by $1$. The vertices in the preimage are in bijection with the former, while the edges and legs are in bijection with the latter. The partially ordered set structure of the cones of $\mathbf G_\tau$ coincides with the graph structure of the $1$-complex over $p$, so we conclude.
\end{proof}

We denote by $G_\tau$ the combinatorial $1$-complex defined by the above lemma.

Fix a ray $r_i$ of $\Sigma_{\mathcal Y}$ and the corresponding star $\Sigma_i$. Given a point in $T(\Sigma_{\mathcal Y})$ that is contained in the overstar of $\rho$, we will take the corresponding $1$-complex in $\Sigma_{\mathcal Y}$ and construct a $1$-complex in $\Sigma_i$. Each vertex of the combinatorial $1$-complex $G_\rho$ is assigned a cone in $\Sigma_{\mathcal Y}$. By Section \ref{sec:setupandnotation}, this cone is in fact a ray of $\Sigma_{\mathcal Y}$ that maps surjectively onto $\mathbb R_{\geq 0}$. Note that there is at most one vertex $u$ of $G_\rho$ that is assigned to $r_i$ since the $1$-complex is embedded.  If no such vertex exists, the cutting map sends everything to the empty $1$-complex, so from now on we assume $u$ exists.

Given a cone $\tau$ of $T(\Sigma_{\mathcal Y})$ containing $\rho$ as  face,with universal family $\mathbf G_\tau \rightarrow \tau$ as above.  The inclusion $\rho \subset \tau$ determines an inclusion of complexes $\mathbf G_\rho \subset \mathbf G_\tau$; moreover the choice of vertex $u$ determines a ray $u_\rho$ in $\mathbf G_\rho$ and thus a ray in $\mathbf G_\tau$ as well.
We then have a diagram of pairs of cone complexes and rays
\[
\begin{tikzcd}
\mathbf (G_{\tau}, u_\rho) \arrow{d}\arrow{r} & (\Sigma_{\mathcal Y}, r_i)\\
(\tau, \rho). & 
\end{tikzcd}
\]

By taking the star fan of each pair in this diagram, we obtain a family of complexes in $\Sigma_i$, and so obtain a map
\[
\tau(\rho)\to |T(\Sigma_i)|.
\]
The construction is easily seen to be compatible with specialization of combinatorial $1$-complexes, and therefore glues across cones. We therefore obtain the cutting morphism
\[
\kappa_i:|T(\Sigma_{\mathcal Y})(\rho)|\to |T(\Sigma_i)|. 
\]

Given a point $p \in \tau \subset T(\Sigma_{\mathcal Y})(\rho)$, it follows from this construction that the only unbounded edges of $\kappa_i(p)$ are parallel to the rays of $\Sigma_i$, since $\mathbf G_\tau$ contains $\mathbf G_\rho$ as a subcomplex.  
Let $T^{\circ}(\Sigma_i)$ denote the subcomplex of $T(\Sigma_i)$ consisting of $1$-complexes with this property; it is an open and closed subcomplex, and it contains the image of $\kappa_i$.

By combining these, we have a \textbf{total cutting map}
\[
\kappa: |T(\Sigma_{\mathcal Y})(\rho)|\to \prod_i |T^\circ(\Sigma_i)|.
\]
Here we are only writing these as maps of underlying spaces, but of course, with appropriate subdivision, they can be lifted to maps of cone complexes. Indeed, the arguments above apply without change to any fixed cone structure, due to the functoriality of the star constructions.

\subsection{Description via metric geometry}\label{sec: alt-description}

We give an equivalent description of the cutting map which can be useful when working with explicit examples.

We begin with two constructions.
First, we construct a retraction relating the combinatorial $1$-complexes $G_\rho$ and $G_\tau$:
\[
\psi: G_\tau\to G_\rho,
\] 
defined as follows.
Given a vertex $v$ of $G_\tau$, as in the proof of Lemma \ref{lem:combtype}, it defines a cone $\sigma_v$ in $\mathbf G_\tau$ which maps isomorphically onto $\tau$.  There is a unique ray $u_\rho$ in $\mathbf G_\rho$ that lies over $\rho \subset \tau$, which corresponds to a unique vertex $u \in G_\rho$ which is $\psi(v)$. The map on edges and rays can be constructed similarly (and is determined by the map on vertices).

Second, we define {\it slice neighborhoods} associated to the map $\pi:  \Sigma_{\mathcal Y} \rightarrow \RR_{>0}$.
Given any point $t$ in $\RR_{>0}$, the fiber $\pi^{-1}(t)$ is a polyhedral complex $\mathscr P_t$; its vertices are in natural bijection with those rays $r_1,\ldots, r_k$ of $\Sigma_{\mathcal Y}$ that are not contracted by $\pi$. Fix a point $t$ and let $v_i$ be a vertex of $\mathscr P_t$. Let $\mathscr P_t^\circ(v_i)$ be the connected component containing $v_i$, of the complement of faces in $\mathscr P_t$ that \textit{do not} contain $v_i$. This set is contained in the overstar of $r_i$ in $\Sigma_{\mathcal Y}$, and therefore we may examine its image in $\Sigma_i$. The projection from the overstar to the star maps this set $\mathscr P_t^\circ(v_i)$ isomorphically onto its image in $\Sigma_i$. We refer to this as the slice neighborhood of $0$ in $\Sigma_i$ at the point $t$. It follows easily from the definitions that upon increasing $t$, the slice neighborhoods get larger, and eventually exhaust $\Sigma_i$.

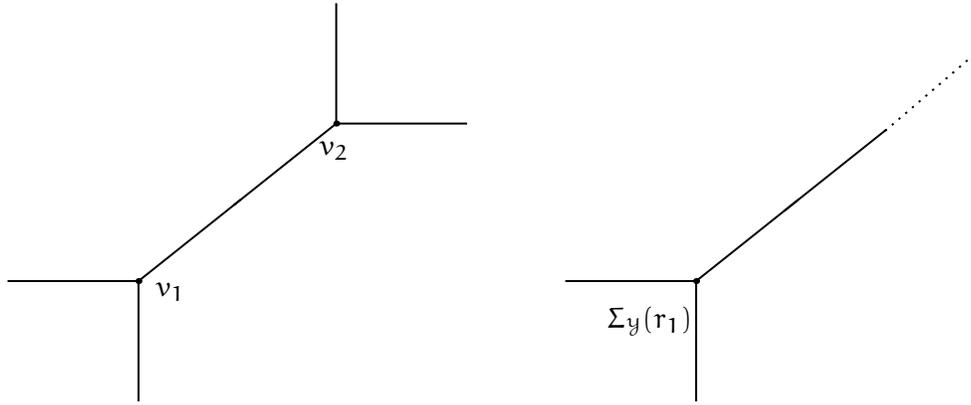
\begin{figure}[h!]
\tikzset{every picture/.style={line width=0.75pt}} 

\tikzset{every picture/.style={line width=0.75pt}} 

\tikzset{every picture/.style={line width=0.75pt}} 

\begin{tikzpicture}[x=0.75pt,y=0.75pt,yscale=-0.75,xscale=0.75]

\draw    (369,192) -- (457,192) ;
\draw    (457,192) -- (457,273) ;
\draw    (457,192) -- (526,137) ;
\draw  [line width=2.25] [line join = round][line cap = round] (457,192) .. controls (457,192) and (457,192) .. (457,192) ;
\draw    (-6,192) -- (82,192) ;
\draw    (82,192) -- (82,273) ;
\draw    (82,192) -- (151,137) ;
\draw    (303,86) -- (215,86) ;
\draw    (215,86) -- (215,5) ;
\draw    (215,86) -- (146,141) ;
\draw  [line width=2.25] [line join = round][line cap = round] (82,192) .. controls (82,192) and (82,192) .. (82,192) ;
\draw  [line width=2.25] [line join = round][line cap = round] (215,86) .. controls (215,86) and (215,86) .. (215,86) ;
\draw    (585,90) -- (516,145) ;
\draw  [dash pattern={on 0.84pt off 2.51pt}]  (644,39.5) -- (584,91) ;

\draw (92,191) node [anchor=north west][inner sep=0.75pt]    {$v_{1}$};
\draw (202,96) node [anchor=north west][inner sep=0.75pt]    {$v_{2}$};
\draw (395,207) node [anchor=north west][inner sep=0.75pt]    {$\Sigma_{\mathcal Y} ( r_{1})$};

\end{tikzpicture}

\caption{A typical slice neighborhood of the star of $r_1$ is depicted on the right. The neighborhood extends only as far as the dotted line in the $(1,1)$ direction, while the star $\Sigma(r_1)$ itself is a fan, and is unbounded in that direction.}
\end{figure}

Given a point $\overline{p} \in  |T(\Sigma_{\mathcal Y})(\rho)|$, we can describe $\kappa_i(\overline{p}) \in T(\Sigma_i)$ as follows.  Let $u$ denote the vertex in $G_\rho$ associated to the ray $r_i$.  Choose a lift $p \in |T(\Sigma_{\mathcal Y})|$ lying over $t \in \RR_{>0}$.  Consider the $1$-complex $\mathbf G_p$; the vertices $v$ of $\mathbf G_p$ such that $\psi(v) = u$ define a subcomplex.  It follows from the definition that these vertices $\psi^{-1}(u)$ lie inside the slice
$\mathscr P_t^\circ(v_i) \subset \mathscr P_t$, and therefore give a finite set $S$ of points inside $\Sigma_i$.  Any edges in $G_\tau$ connecting vertices in $\psi^{-1}(u)$ define corresponding edges connecting points of $S$; similarly, any edges or rays with one boundary vertex inside $\psi^{-1}(u)$ define unbounded rays for the corresponding point of $S$.  The result is a $1$-complex in $\Sigma_i$, which is precisely $\kappa_i(\overline{p})$.

\subsection{Evaluations} Our next task is to characterize the image of the total cutting map. In parallel with geometric degeneration formulas, we will identify it with the subset of the product space where a natural matching condition holds, given by a diagonal condition for an evaluation map. 

If $F$ is any cone complex, recall that $P(F)$ is the subcomplex of $T(F)$ parameterizing $0$-complexes, i.e. $1$-complexes without edges or rays. 

If we fix a ray $\delta$ in the fan of $\Sigma_i$, let $\Sigma_i(\delta)$ denote the star fan of the ray $\delta$. Recall from Section~\ref{sec: points} the evaluation map
\[
\mathsf{ev}_\delta: |T^\circ(\Sigma_i)|\to |P(\Sigma_i(\delta)).|
\]
Informally, it is described as follows. Given a $1$-complex, outside of a compact neighborhood of the origin of $\Sigma_i$, it consists of a disconnected union of rays, each parallel to one of the rays of $\Sigma_i$. In particular, there is a finite set of rays parallel to $\delta$, and by taking the quotient in the direction of $\delta$, we obtain a collection of points on $\Sigma_i(\delta)$. 




We return to our fixed ray $\rho$ in $T(\Sigma_{\mathcal Y})$ with $G_\rho$ be the combinatorial $1$-complex over this ray. Let $e$ be an edge in $G_\rho$ between vertices $v_1$ and $v_2$. The flags $(v_i,e)$ determine rays $\delta_1$ and $\delta_2$ in star fans $\Sigma_1$ and $\Sigma_2$ respectively. Finally, by our assumption the edge $e$ picks out an edge in each fiber of $\Sigma_{\mathcal Y}$ over $\RR_{>0}$. Let $\delta_e$ be the $2$-dimensional cone in $\Sigma_{\mathcal Y}$ corresponding to $e$. 

The following lemma is immediate from the definitions. 

\begin{lemma}\label{lem: gluing divisors}
The following stars are canonically identified:
\[
\Sigma_{\mathcal Y}(\delta_e) = \Sigma_1(\delta_1) = \Sigma_2(\delta_2).
\]
\end{lemma}

Given a ray $\rho$ in $T(\Sigma)$, we have exhibited cutting maps
\[
\kappa: |T(\Sigma_{\mathcal Y})(\rho)|\to \prod_i |T^\circ(\Sigma_i)|
\]
where the latter product is the space of $1$-complexes in $\Sigma_1,\ldots, \Sigma_k$. Let $G_\rho$ be the combinatorial $1$-complex along $\rho$. Fix an edge $e$ with incident vertices $v_1$ and $v_2$. Given this, and the canonical identification of Lemma~\ref{lem: gluing divisors} above, we can construct two evaluation maps to $\Sigma_{\mathcal Y}(\delta_e)$:
\[
\mathsf{ev}_{\delta_i}\circ \kappa_i: |T(\Sigma_{\mathcal Y})(\rho)|\to |P(\Sigma_{\mathcal Y}(\delta_e))|. 
\]
Packaging them, we have a map
\[
\mathsf{ev}_e: |T(\Sigma_{\mathcal Y})(\rho)|\to |P(\Sigma_{\mathcal Y}(\delta_e))|\times |P(\Sigma_{\mathcal Y}(\delta_e))|.
\]

\begin{lemma}
The image of the map $\mathsf{ev}_e$ is contained in the diagonal. 
\end{lemma}

\begin{proof}
Given a cone $\tau$ containing $\rho$ as a ray, the edge $e$ defines a two-dimensional cone $\delta_e \subset \mathbf G_\rho \subset \mathbf G_\tau$, and as in the definition of the cutting map, its star defines a map $\kappa_e: T(\Sigma_{\mathcal Y})(\rho)\to P(\Sigma_{\mathcal Y}(\delta_e))$.

On the other hand, for each $i=1,2$, the vertex $v_i$ defines a ray $\rho_i \subset \delta_e \subset \mathbf G_\tau$, and we have a diagram of triples
\[
\begin{tikzcd}
\mathbf (G_{\tau}, \delta_e, \rho_i) \arrow{d}\arrow{r} & (\Sigma, \delta_e, \rho_i)\\
(\tau, \rho, \rho) & 
\end{tikzcd}
\]
By applying the compatibility of the star construction with respect to triples, we see that $$\mathsf{ev}_{\delta_1}\circ \kappa_1 = \kappa_e = \mathsf{ev}_{\delta_2}\circ \kappa_2.$$ The result follows.   
\end{proof}

\subsection{The degeneration formula} Our final task is to characterize the image of the total cutting map $\kappa$ via a matching condition given by the diagonal in the tropical evaluation space. 

Consider the cutting morphism
\[
\kappa\colon|T(\Sigma_{\mathcal Y})(\rho)|\to \prod_i |T^\circ(\Sigma_i)|. 
\]
There is an evaluation morphism
\[
\mathsf{ev}\colon \prod_i |T^\circ(\Sigma_i)|\to \prod_{e} |P(\Sigma_{\mathcal Y}(\delta_e))|\times |P(\Sigma_{\mathcal Y}(\delta_e))|.
\]
Let $\Delta$ denote the product over all edges $e$ of the diagonal loci in $|P(\Sigma_{\mathcal Y}(\delta_e))|\times |P(\Sigma_{\mathcal Y}(\delta_e))|$. 

\begin{theorem}
The cutting map induces an equality of sets $\kappa\colon |T(\Sigma_{\mathcal Y})(\rho)| \xrightarrow{\sim} \mathsf{ev}^{-1}(\Delta)$. 
\end{theorem}

\begin{proof}

To prove injectivity and surjectivity, it is convenient to use the description of the cutting map via slice neighborhoods, as in Section~\ref{sec: alt-description}.

\noindent
{\bf I. Injectivity.}
Given $\overline{p}, \overline{p}' \in T(\Sigma_{\mathcal Y})(\rho)$, we can choose lifts $p, p' \in T(\Sigma_{\mathcal Y})$ lying over the same point $t \in \RR_{>0}$.
For each ray $r_i$ of $\Sigma_{\mathcal Y}$ with corresponding vertex $v_i$ of $\mathscr P_t$, the data characterizing $\kappa_i(\overline{p})$ is equivalent to specifying the open subset of $\mathbf G_p \cap \mathscr P_t^\circ(u_i)$ determined by the vertices of $G_p$ which retract to $v_i$.  As $i$ varies, these open subsets cover $\mathbf G_p$.  
In particular, if $\kappa_i(\overline{p}) = \kappa_i(\overline{p}')$ for all $i$, the corresponding open subset of the slice neighborhoods agree and, by taking the union over $i$, we must have $\mathbf G_p = \mathbf G_{p'}$ so that $p = p'$.

To prove surjectivity onto the pullback of the diagonal we give a reconstruction procedure: given points of each $T^\circ(\Sigma_i)$ that evaluate into the diagonal, we will construct, for any sufficiently large $t\in \RR_{>0}$, an embedded $1$-complex in $\Sigma_{\mathcal Y}$ that lies in the fiber of over $t$ 

\noindent
{\bf II. Exporting $1$-complexes to the slice neighborhood.} Fix a point in $T(\Sigma_i)$ and let $G$ be the embedded $1$-complex. On the complement of a compact neighborhood of the origin, the $1$-complex $\mathbf G$ consists of a disjoint union of unbounded rays, each parallel to one of the rays of $\Sigma_i$. For future use, we refer to the subcomplex that excludes these unbounded rays as a \textit{skeleton}. In particular, there exists some sufficiently large value of $t$ such that this slice neighborhood contains the skeleton and intersects all the unbounded rays. We now lift the $1$-complex parameterized by a point $T(\Sigma_i)$ to a non-compact, $1$-dimensional polyhedral subset $Q_i$ of $\mathscr P_t$, by identifying the slice neighborhood with a subset of $\mathscr P_t$ as explained above.

\noindent
{\bf III. Preparing the lifts.} Now fix a point $p$ in the product $\prod_i T(\Sigma_i)$, which determines $1$-complexes in each $\Sigma_i$. We have described a way to produce, for any sufficiently large $t$, a lift of each of these $1$-complexes to polyhedral subsets of $\mathscr P_t$. If the $t$ coordinate is sufficiently large, we can also guarantee the following separation criterion: there is a compact neighborhood $K_i$ of each vertex $v_i$, containing the skeleton of $Q_i$, such that $K_i\cap K_j$ is empty for distinct $i$ and $j$. In preparation for gluing the lifts $Q_i$, we adjust them as follows. For each $Q_i$, we truncate the non-compact edges that lie outside the skeleton such that these edges remain disjoint from $K_j$ for $i$ and $j$ distinct. 

\noindent
{\bf IV. Gluing the lifts.} We now specialize to the case where $p$ is a point in $\mathsf{ev}^{-1}(\Delta)$. By applying the procedure above, we obtain lifts $Q_i$ of $1$-complexes, to non-compact polyhedral subsets of $\mathscr P_t$, with $t$ sufficiently large as by the preceding step. For each edge $e$ incident to a vertex $v_i$, the polyhedral subset $Q_i$ has (a potentially empty) set of non-compact edges parallel to $e$. The lengths of these edges may be increased or decreased, maintaining the property that the edges belonging to $Q_i$ are disjoint from the neighborhood $K_j$ for $i$ and $j$ distinct. We now use the diagonal condition to glue the lifts together to a $1$-complex in $\Sigma_{\mathcal Y}$, living over the sufficiently large value of $t$. Consider two vertices $v_1$ and $v_2$ of $\mathscr P_t$, sharing an edge $e$ and the lifts $Q_1$ and $Q_2$ described above. The complement of the skeleton in each includes a set non-compact edges parallel to $e$. In both cases, there is a natural projection from the cone containing these non-compact edges to $\Sigma_{\mathcal Y}(\delta_e)$. By construction, the image is a $0$-complex, i.e. a point of $P(\Sigma_{\mathcal Y}(\delta_e)))$. The diagonal condition ensures that this $0$-complex is the same for both $Q_1$ and $Q_2$. By linear geometry, it follows that these non-compact edges in the direction of $e$ are open subsets of the preimage of a set of points in $\Sigma_{\mathcal Y}(\delta_e)$, where the preimage is taken under the natural projection from the overstar of $e$ to the quotient. it follows that if the edge lengths of these non-compact edges are adjusted appropriately, the lifts $Q_1$ and $Q_2$ may be glued. Ranging over all vertices, we now obtain a $1$-complex in $\mathscr P_t$ whose image under the cutting map is the chosen point $p$. 

It follows from the steps in the construction that the natural 
\[
\kappa\colon|T(\Sigma_{\mathcal Y})(\rho)|\to \prod_i |T(\Sigma_i)|
\]
is an isomorphism onto $\mathsf{ev}^{-1}(\Delta)$ as claimed. 
\end{proof}

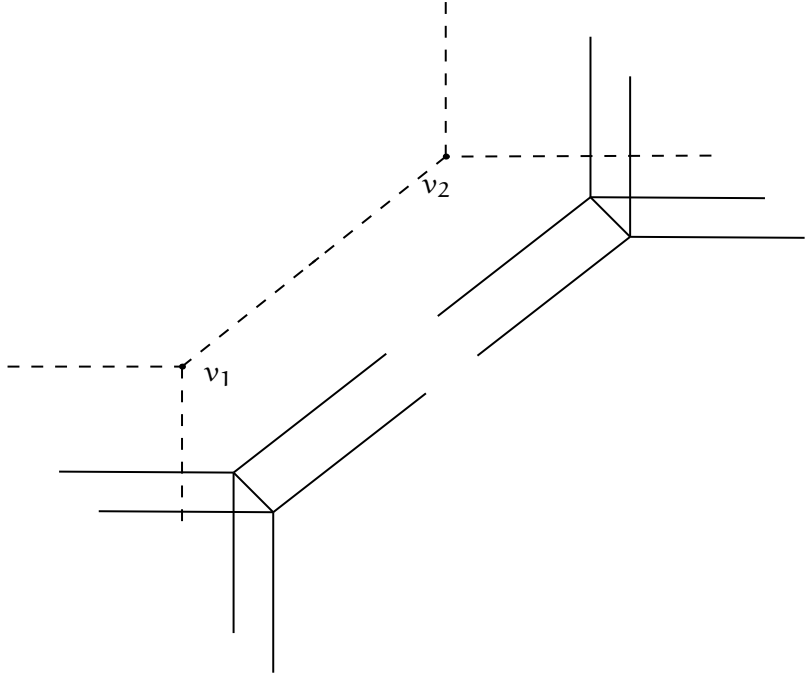
\begin{figure}

\tikzset{every picture/.style={line width=0.75pt}} 

\begin{tikzpicture}[x=0.75pt,y=0.75pt,yscale=-1,xscale=1]

\draw  [dash pattern={on 4.5pt off 4.5pt}]  (153,227) -- (241,227) ;
\draw  [dash pattern={on 4.5pt off 4.5pt}]  (241,227) -- (241,308) ;
\draw  [dash pattern={on 4.5pt off 4.5pt}]  (241,227) -- (310,172) ;
\draw  [dash pattern={on 4.5pt off 4.5pt}]  (508,120.5) -- (374,121) ;
\draw  [dash pattern={on 4.5pt off 4.5pt}]  (374,121) -- (374,40) ;
\draw  [dash pattern={on 4.5pt off 4.5pt}]  (374,121) -- (305,176) ;
\draw  [line width=2.25] [line join = round][line cap = round] (241,227) .. controls (241,227) and (241,227) .. (241,227) ;
\draw  [line width=2.25] [line join = round][line cap = round] (374,121) .. controls (374,121) and (374,121) .. (374,121) ;
\draw    (179,280) -- (267,280.5) ;
\draw    (267,280.5) -- (267,361.5) ;
\draw    (267,280.5) -- (344,220.5) ;
\draw    (199,300) -- (287,300.5) ;
\draw    (287,300.5) -- (287,381.5) ;
\draw    (287,300.5) -- (364,240.5) ;
\draw    (555,162) -- (467,161.5) ;
\draw    (467,161.5) -- (467,80.5) ;
\draw    (467,161.5) -- (390,221.5) ;
\draw    (535,142) -- (447,141.5) ;
\draw    (447,141.5) -- (447,60.5) ;
\draw    (447,141.5) -- (370,201.5) ;
\draw    (267,280.5) -- (287,300.5) ;
\draw    (447,141.5) -- (467,161.5) ;

\draw (251,226) node [anchor=north west][inner sep=0.75pt]    {$v_{1}$};
\draw (361,131) node [anchor=north west][inner sep=0.75pt]    {$v_{2}$};

\end{tikzpicture}
\caption{The dashed structure represents the polyhedral complex $\mathscr P_t$, while the solid structure represents the lifts of the $1$-complex from slice neighborhoods }
\end{figure}

\section{A first DT gluing formula}\label{sec: moving-gluing}

The decomposition aspect of the degeneration package above isolates virtual components $\mathsf{DT}_\gamma(Y_0)$ of the special fiber $\mathsf{DT}(Y_0)$. We explain how to construct this space out of the spaces $\mathsf{DT}(X_v)$ associated to the vertices of the rigid tropical curve $\gamma$. 

The basic idea is to express a subscheme in an expansion of the broken target $Y_0$ by gluing together subschemes in expansions of the components $X_v$. There are two aspects of this problem. First, we construct a moduli space of subschemes in expansions of the disjoint union of all the $X_v$, but with the additional property that these expansions glue together to form an expansion of $Y_0$. Next, we impose the condition that the subschemes themselves also glue together. We deal with these in turn. 

\subsection*{Conventions on terminology} For readability, in this section we drop the decoration of moduli spaces by discrete data whenever they are clear from context. For example, we will use $\mathsf{DT}(X_v)$ to mean the appropriate component of the DT spaces associated to $X_v$ determined by the rigid tropical curve $\gamma$. 

{We also recall a piece of terminology from the first part of the paper, as it has been a while since we've used it. Our moduli spaces $\mathsf{DT}(Y_0)$, $\mathsf{DT}(X_v)$, and $\mathsf{Hilb}^{\mathbf{n}}(D)$ have natural strict maps to Artin fans -- respectively $\mathsf{Exp}(Y_0)$, $\mathsf{Exp}(X_v)$, and $\mathsf{Exp}_0(D)$. Given these maps, a distinguished class of logarithmic modifications arise as pullbacks of modifications of this Artin fan. We called such modifications {\it tropical models}. }

\subsection{Gluing the target families} The notation is maintained from the previous section: we have a degeneration $\mathcal Y\to B$, with special fiber $Y_0$, and components $X_v$. We examine the product $\prod_v \mathsf{DT}(X_v)$. Over this moduli space, we have a disjoint union of subschemes in target expansions:
\[
\begin{tikzcd}
\coprod \mathcal Z_v \arrow{dr} \arrow{rr} & & \coprod \mathcal X_v \arrow{r} \arrow{dl} & \coprod X_v\\
&  \prod_v \mathsf{DT}(X_v). & &
\end{tikzcd}
\]
Consider the locus 
\[
\prod_v \mathsf{DT}^\circ(X_v)\subset \prod_v \mathsf{DT}(X_v)
\]
parameterizing subschemes where the universal target, in each factor, is {\it unexpanded}. On this locus, the subschemes $X_v$ glue together to a universal target
\[
Y_0\times \prod_v \mathsf{DT}^\circ(X_v)\to \prod_v \mathsf{DT}^\circ(X_v).
\]
However, there may exist points of $\prod_v \mathsf{DT}(X_v)$ where there is no way to glue the expanded target families $\mathcal X_v$ to form a family of expansions of $Y_0$. Indeed, as a simple example, consider a point of the product where the universal expansion is nontrivial in one factor but nontrivial at another factor. If the corresponding components in $Y_0$ meet, the expansions should be glued but this is impossible. Therefore it is not clear how to extend the gluing 
\[
Y_0 = \coprod_v X_v/\sim
\]
to the expansions over the product. 

The solution can be seen tropically. We pass to the locus in the tropical space of $1$-complexes where the targets {\it can} be glued -- indeed, this was already considered in the preceding Section~\ref{sec: combinatorial-degeneration}, as the locus of $1$-complexes in the fans associated to vertices of the rigid curve $\gamma$ that glue along the maps induced by the edge directions of the initial $1$-complex $\gamma$.

\subsubsection{Combinatorial splitting}\label{sec: comb-splitting} Let us recall the combinatorial splitting formula of Section~\ref{sec: combinatorial-degeneration}. First, we fixed a ray $\rho$ in the space $T(\Sigma_{\mathcal Y})$ corresponding to a rigid $1$-complex $\gamma$, with the assumption that the general $1$-complex on this ray was supported on the fan of $\Sigma_{\mathcal Y}$. 

We then constructed a total cutting map:
\[
\kappa: |T(\Sigma_{\mathcal Y})(\rho)|\to \prod_i |T^\circ(\Sigma_i)|,
\]
from the star of that ray to the product of spaces of $1$-complexes associated to stars of the rays of $\Sigma_{\mathcal Y}$. We also constructed an evaluation:
\[
\mathsf{ev}: \prod_i |T^\circ(\Sigma_i)|\to \prod_{e} |P(\Sigma_{\mathcal Y}(\delta_e))|\times |P(\Sigma_{\mathcal Y}(\delta_e))|.
\]
Finally, we identified a locus
\[
\Delta\subset \prod_{e} |P(\Sigma_{\mathcal Y}(\delta_e))|\times |P(\Sigma_{\mathcal Y}(\delta_e))|
\]
parameterizing the diagonal among pairs of $0$-complexes, one pair corresponding to each edge $e$. 

The combinatorial splitting formula asserts:
\[
|T(\Sigma_{\mathcal Y})(\rho)|\to\mathsf{ev}^{-1}(\Delta) 
\]
is an isomorphism. 

To use this geometrically, we fix cone structures on the spaces involved. In particular, we fix a cone structure on $\prod_i |T^\circ(\Sigma_i)|$ such that $\mathsf{ev}^{-1}(\Delta)$ is a union of cones; we will denote this cone structure
by $\bigtimes_i T^\circ(\Sigma_i)$. We may also assume that we have chosen a cone structure on $|T(\Sigma_{\mathcal Y})(\rho)|$ in such a way that 
\[
T(\Sigma_{\mathcal Y})(\rho)\to\mathsf{ev}^{-1}(\Delta) \subset \bigtimes_i T^\circ(\Sigma_i),
\]
becomes an isomorphism of cone complexes onto the preimage of the diagonal. Note this may involve changing integral structures. 

The cone structure $\bigtimes_i T^\circ(\Sigma_i)$ on the product space need not -- and typically is not -- a product of cone structures on the factors. Nevertheless, it gives rise to a modification
\[
\bigtimes_v \mathsf{DT}(X_v)\to \prod_v \mathsf{DT}(X_v).
\]
The subcomplex structure on $\mathsf{ev}^{-1}(\Delta)$ alluded to above is an {\it incomplete subdivision}, and so it gives us an open immersion:
\[
{\bigtimes}_v^\circ \mathsf{DT}(X_v)\hookrightarrow \bigtimes_v \mathsf{DT}(X_v).
\]

\subsubsection{Geometric relevance} We briefly explain the geometric role of this open subspace ${\bigtimes}_v^\circ \mathsf{DT}(X_v)$ in the next proposition and following discussion. This discussion is not logically necessary and the reader can skip to Section~\ref{sec: evaluation-sections} without loss of continuity. 

Recall that we have fixed a rigid $1$-complex $\gamma$.

\begin{proposition}
    Let $e$ be an edge between vertices $u$ and $v$ in $\gamma$. On the locus ${\bigtimes}_v^\circ \mathsf{DT}(X_v)$, the universal expansion of $D_e$ induced by the universal expansion $\mathcal X_u$ and by $\mathcal X_v$ are isomorphic. More precisely, the resulting families $\mathcal D_{e,u}$ and $\mathcal D_{e,v}$ are isomorphic as logarithmic modifications of ${\bigtimes}_v^\circ \mathsf{DT}(X_v)\times D_e$.
\end{proposition}

Let us express the phrase ``isomorphic as logarithmic modifications'' slightly differently. The universal expansion of $D_e$ is a naturally a subfunctor of the functor of points on logarithmic schemes of ${\bigtimes}_v^\circ \mathsf{DT}(X_v)\times D_e$. The proposition asserts that the two expansions $\mathcal X_u$ and $\mathcal X_v$ induce the {\it same} subfunctor. 

\begin{proof}
    The induced expansion of $D_e$ associated to $\mathcal X_v$ can be calculated as follows.  We take the universal $1$-complex, which is a family of $1$-complexes in $\Sigma_v$ over the subcomplex (whose support is) $\mathsf{ev}^{-1}(\Delta)$, and take asymptotics in the direction determined by the flag $(v,e)$. Since it is calculated at this tropical level, the result follows formally from Section~\ref{sec: combinatorial-degeneration} and in particular, the definition of $\mathsf{ev}^{-1}(\Delta)$: this is, by definition, the locus where the $0$-complexes induced at $u$ and $v$ coincide. 
\end{proof}

As a consequence of the proposition, for each edge $e$ between vertices $u$ and $v$ we have a diagram \textit{over} ${\bigtimes}_v^\circ \mathsf{DT}(X_v)$
\[
\begin{tikzcd}
    &\mathcal X_v&\\
    \mathcal D_e\arrow{ur}\arrow{dr}&&\\
    &\mathcal X_u&
\end{tikzcd}
\]
whose pushout glues the expansion of $X_v$ and $X_u$ along the expansion of $D_e$. More generally, by ranging over the edges in $\gamma$, we can take the colimit of all the spaces $\{\mathcal X_v\}_v\cup\{\mathcal D_e\}_e$ under the natural inclusions of the spaces $\mathcal D_e$ into the spaces $\mathcal X_v$ whenever $e$ is incident to $v$. The result is a scheme
\[
\cX^{\Join} \to {\bigtimes}_v^\circ \mathsf{DT}(X_v).
\]
Thus, over the space ${\bigtimes}_v^\circ \mathsf{DT}(X_v)$ we can extend the gluings of $X_v$ to $Y_0$ to gluings of expansions $\cX_v$ to form expansions of $Y_0$.

\subsection{Evaluation sections and the family Hilbert scheme}\label{sec: evaluation-sections} If we fix an edge $e$ in a rigid $1$-complex $\gamma$, we have constructed a flat family $\cD_e$ over ${\bigtimes}^\circ_v \mathsf{DT}(X_v)$, which captures the induced expansion of $D_e$ from both vertices adjacent to $e$.

\begin{construction}[Boundary evaluation maps]
Let $\mathcal D_e$ be the double divisor family above. The natural morphism
\[
\mathcal D_e\to {\bigtimes}^\circ_v \mathsf{DT}(X_v)
\]
is smooth. We may now pass to the traditional relative Hilbert scheme of $(D_e\cdot\beta)$-many points of the family $\mathcal D_e$, denoted
\[
\mathsf H(\mathcal D_e):=\mathsf{Hilb}\left(\mathcal D_e/{\bigtimes}^\circ_v \mathsf{DT}(X_v)\right) \to {\bigtimes}^\circ_v \mathsf{DT}(X_v).
\]
We let $\mathsf H^{\{2\}}(\mathcal D_e)$ be the fiber square of $\mathsf H(\mathcal D_e)$, subject to the following stability condition: a point of this fiber square determines two $0$-dimensional subschemes on the same expansion of $D_e$, and we demand that for every component of the expansion, at least one of these subschemes has nontrivial support. In keeping with the notation in the section, we drop the number of points, which is understood to be $D_e\cdot\beta$. 

Let $v_1$ and $v_2$ be the vertices adjacent to $e$ and let $\cZ_{v_i}$ be the pullbacks of the corresponding universal subschemes from $\mathsf{DT}(X_v)$ to ${\bigtimes}^\circ_v \mathsf{DT}(X_v)$. Intersecting $\mathcal D_e$ with either $\mathcal Z_{v_1}$ and $\mathcal Z_{v_2}$ gives rise to points in $\mathsf H(\mathcal D_e)$. As a result, we have evaluation maps:
\[
\mathsf{ev}_{v_i,e}\colon {\bigtimes}^\circ_v \mathsf{DT}(X_v)\to \mathsf H(\mathcal D_e), \ \ \textnormal{for $i=1,2$.}
\]
We can consolidate these to get a map:
\[
\mathsf{ev_e}:=\mathsf{ev}_{v_1,e}\times \mathsf{ev}_{v_2,e}\colon {\bigtimes}^\circ_v \mathsf{DT}(X_v) \to \mathsf H^{\{2\}}(\mathcal D_e).
\]
We refer to it as the \textit{boundary evaluation map at $e$}. 
\end{construction}

\begin{lemma}\label{lem: diag-intersection}
Let $\mathsf H^\Delta(\mathcal D_e)$ be the relative diagonal locus in $\mathsf H^{\{2\}}(\mathcal D_e)$ over ${\bigtimes}^\circ_v \mathsf{DT}(X_v)$. The inclusion is a regular embedding. 
\end{lemma}

\begin{proof}
The Hilbert scheme of points on a surface is smooth, and as a consequence $\mathsf{H}(\mathcal D_e)$ is a smooth fibration; the diagonal is therefore a regular embedding. 
\end{proof}

Let $\mathsf{DT}_\gamma(Y_0)$
denote the tropical model associated to the chosen cone complex structure
$T(\Sigma_{\mathcal Y})(\rho)$.
By construction, we have a map
\[
\mathsf{DT}_{\gamma}(Y_0)\to {\bigtimes}_v \mathsf{DT}(X_v)
\]
associated to the map of cone complexes $T(\Sigma_{\mathcal Y})(\rho) \rightarrow \bigtimes_i T^\circ(\Sigma_i)$. Examining the morphism, target space contains ${\bigtimes}_v^\circ \mathsf{DT}(X_v)$ as an open subspace. Indeed, the map $T(\Sigma_{\mathcal Y})(\rho) \rightarrow \bigtimes_i T^\circ(\Sigma_i)$ factors through the subspace $\mathsf{ev}^{-1}(\Delta)$. We can therefore write a more refined map
\[
\mathsf{DT}_{\gamma}(Y_0)\to {\bigtimes}^\circ_v \mathsf{DT}(X_v)
\]


\begin{proposition}\label{prop: expanded-formula}

There is an equality of Chow homology classes
\[
q_\star\left[\mathsf{DT}_\gamma(Y_0)\right]^{\mathsf{vir}} = \bigcap_e \mathsf H^\Delta(\mathcal D_e)\cap \left[{\bigtimes}^\circ_v \mathsf{DT}(X_v)\right]^{\mathsf{vir}},
\]
where the intersection product on the right is defined by the regular embedding in Lemma~\ref{lem: diag-intersection}.
\end{proposition}

\begin{proof}
We split the statement into two claims. First, an equality of spaces, and then a compatibility with virtual classes. The scheme theoretic statement is as follows.

\noindent
{\bf Schematic claim.} Our goal is to compare the spaces
\[
\mathsf{DT}_\gamma(Y_0) \ \ \text{and} \ \ \ \mathsf H^\Delta(\mathcal D_e)\cap {\bigtimes}^\circ_v \mathsf{DT}(X_v).
\]
Earlier in this section, we have chosen a cone complex structure on the product space $\prod_i |T^\circ(\Sigma_i)|$ and on $\mathsf{ev}^{-1}(\Delta)$. We consequently get Artin fans, which we denote:
\[
{\bigtimes}_v \mathsf{Exp}(X_v) \ \ \text{and} \ \ \mathsf{Exp}_\Delta.
\]
The space on the left is a blowup of the product of Exp-spaces associated to vertices of $\gamma$, while the space on the right is an open substack of this modification. 

We now make three observations. First, there are structure maps
\[
\mathsf{DT}_\gamma(Y_0)\to \mathsf{Exp}_\gamma(Y_0), \ \ {\bigtimes}_v \mathsf{DT}(X_v)\to {\bigtimes}_v\mathsf{Exp}(X_v), \ \ \text{and} \ \ {\bigtimes}^\circ_v \mathsf{DT}(X_v)\to \mathsf{Exp}_\Delta.
\]
Second, by the combinatorial degeneration formula, the cutting map at the tropical level gives rise to an arrow
\[
\mathsf{Exp}_\gamma(Y_0)\to {\bigtimes}_v\mathsf{Exp}(X_v),
\]
and by choosing cone structures compatibly, we can assume this is an isomorphism onto $\mathsf{Exp}_\Delta$. And third, we have a natural cutting map 
\[
\begin{tikzcd}
\mathsf{DT}_\gamma(Y_0)\arrow{d}\arrow{r} &{\bigtimes}_v \mathsf{DT}(X_v)\arrow{d}\\
\mathsf{Exp}_\gamma(Y_0)\arrow{r} & \bigtimes_v\mathsf{Exp}(X_v)
\end{tikzcd}
\]
lifting the combinatorial cutting map on Artin fans, in Section~\ref{sec: comb-splitting}. The map on Artin fans is given by partial normalization of the expansion; the top vertical arrow is then defined by pulling back the subscheme to the components. As we have already noted before the proposition, the map 
factorizes through 
\[
 \bigcap_e \mathsf H^\Delta(\mathcal D_e)\cap {\bigtimes}^\circ_v \mathsf{DT}(X_v)\subset {\bigtimes}_v \mathsf{DT}(X_v).
\]
Together, we therefore have a diagram
\[
\begin{tikzcd}
\mathsf{DT}_\gamma(Y_0) \ar{r}\ar{d}&  \bigcap_e \mathsf H^\Delta(\mathcal D_e)\cap {\bigtimes}^\circ_v \mathsf{DT}(X_v) \ar{d}\\
\mathsf{Exp}_\gamma(Y_0) \ar[equal]{r} & \mathsf{Exp}_\Delta.
\end{tikzcd}
\]
We claim the top horizontal is an isomorphism. Since the two bases are naturally identified, we can verify that the functors of points {\it relative to the base} coincide. We can see this from definitions: if we fix an $S$-valued point of $\mathsf{Exp}_\gamma(Y_0)$ we have fixed a family of expansions over $S$ of each component $X_v$, that glue. A point of the space $\mathsf{DT}_\gamma(Y_0)$ is precisely the data of subschemes in each irreducible component that are algebraically transverse to the strata and agree at the pairwise intersections.

The virtual structure statement now remains. The schematic claim gives an isomorphism:
\[
q: \mathsf{DT}_\gamma(Y_0) \to \bigcap_e \mathsf H^\Delta(\mathcal D_e)\cap {\bigtimes}_v^\circ \mathsf{DT}(X_v),
\]
The right hand side carries a natural virtual class, obtained by pulling back the virtual class of $\bigtimes_v \mathsf{DT}(X_v)$ along a regular embedding. We therefore make the following:

\noindent
{\bf Virtual claim.} Pushforward along $q$ gives an identification of Chow homology classes:
\[
q_\star\left[\mathsf{DT}_\gamma(Y_0)\right]^{\mathsf{vir}} = \bigcap_e \mathsf H^\Delta(\mathcal D_e)\cap \left[{\bigtimes_v}^\circ \mathsf{DT}(X_v)\right]^{\mathsf{vir}}.
\]
The virtual structure claim follows from the fact that diagonal condition coming from the gluing condition is compatible with the obstruction theory on $\mathsf{DT}_\gamma(Y_0)$. This latter result follows from~\cite{MPT10} and \cite[Section~6]{LiWu15}, and since our gluing divisor is smooth, the same argument applies here, so we will sketch the key points.

 We denote by $\mathcal{I}_0$ denote the universal ideal sheaf on $\mathsf{DT}_\gamma(Y_0)$; we denote by
$\mathcal{I}_v$ and $\mathcal{I}_e$ the restriction of this sheaf to $\mathcal{X}_v$ and $\mathcal{D}_e$ respectively.
From algebraic transversality, these are related by a short exact sequence
\[
0 \rightarrow \mathcal{I}_0 \rightarrow \bigoplus_{v} \mathcal{I}_v \rightarrow \bigoplus_{e} \mathcal{I}_e \rightarrow 0.
\]
This induces an exact triangle in the derived category:
\[
R\pi_\star\mathcal{RH}om(\mathcal{I}_0, \mathcal{I}_0)_0 \rightarrow \bigoplus_{v} R\pi_\star\mathcal{RH}om(\mathcal{I}_v, \mathcal{I}_v)_0 \rightarrow 
\bigoplus_{e} R\pi_\star\mathcal{RH}om(\mathcal{I}_e, \mathcal{I}_e)_0 \rightarrow\cdots.
\]
Each summand in right term in this triangle is then identified, after a shift, with the normal bundle of the relative diagonal 
$\mathsf H^\Delta(\mathcal D_e) \hookrightarrow \mathsf{H}^{\{2\}}(\mathcal D_e)$.
This implies that, after taking duals, we have a compatible triangle of obstruction theories\footnote{Abramovich suggests that such a triangle should be referred to as a \textit{justice} of obstruction theories~\cite[Section~6]{ACGS17}. } and the equality of virtual classes follows.  Again, we refer the reader to \cite{LiWu15} for more details.
\end{proof}

\section{The main splitting formula}\label{sec: main-splitting-formula}

We are now in a position to relate the output of the weak splitting formula to the statement of Theorem~\ref{thm: main-splitting-theorem}. We do this on the DT side, and then run a parallel argument in GW theory based on the output of~\cite{R19}.

\subsection{Gluing in DT theory}\label{gluing-in-dt} In the context of Proposition~\ref{prop: expanded-formula} in the previous section, we had a space $\bigtimes_v\mathsf{DT}(X_v)$, and a distinguished open subspace ${\bigtimes}^\circ_v\mathsf{DT}(X_v)$ that supported the evaluation map to the family Hilbert schemes $\mathsf H^{\{2\}}(\cD_e)$. The diagonal loci then allowed us to reconstruct the spaces $\mathsf{DT}_\gamma(Y_0)$.

We now manipulate this proposition using logarithmic intersection arguments, into the claimed form in our main splitting theorem. We keep the notation from that section, and in particular the space $\bigtimes_v\mathsf{DT}(X_v)$.

Recall that we have set
\[
\mathsf{Hilb}_{\beta}(D_v):=\prod_{i=1}^m \mathsf{Hilb}_{\beta\cdot D_{e_i(v)}}(D_{e_i(v)}). 
\]
We have also chosen logarithmic modifications such that
\[
\mathsf{ev}:\prod_v \mathsf{DT}_{\gamma(v)}(X_v) \to \prod_v \mathsf{Hilb}^\dagger_{\beta}(D_v),
\]
is combinatorially flat with reduced fibers. Let $\bigtimes_v \mathsf{Hilb}(D_v)$ be any modification of $\prod_v \mathsf{Hilb}(D_v)$ such that $\bigtimes_v\mathsf{DT}(X_v)$ maps to this modification via a combinatorially flat evaluation map. Note that this may require replacing $\bigtimes_v\mathsf{DT}(X_v)$ with a further subdivision.

\begin{lemma}
The morphism
\[
\bigtimes^{\circ}_v\mathsf{DT}(X_v)\to \bigtimes_v \mathsf{Hilb}(D_v)
\]
factors as
\[
\bigtimes^{\circ}_v\mathsf{DT}(X_v)\to {\prod_e}_{\bigtimes^{\circ}_v\mathsf{DT}(X_v)} \!\!\!\!\!\!\!\!\!\!\!\!\!\!\!\!\!\!\!\mathsf{H}^{\{2\}}(\mathcal D_e) \xrightarrow{f} \bigtimes_v \mathsf{Hilb}(D_v). 
\]
The second map $f$ is combinatorially flat. Note that the middle term is a fiber product of the spaces indexed by the edges, taken over $\bigtimes_v\mathsf{DT}(X_v)$.
\end{lemma}
\begin{proof}
The factorization follows easily from definitions. For the flatness, consider the projection
\[
\mathsf{H}^{\{2\}}(\mathcal D_e)\to \bigtimes_v\mathsf{DT}(X_v).
\]
from the fiberwise Hilbert square. We observe that the fiberwise Hilbert square parameterized subschemes supported on the smooth locus of the expansion, and therefore the relative logarithmic structure on this projection map is trivial, i.e. the projection is strict. It follows that the associated map of tropical spaces is a \textit{face morphism}: every cone maps isomorphically onto a cone of the image. Now, in order to compute the tropical map induced by $f$, we can take a cone of $\mathsf{H}^{\{2\}}(\mathcal D_e)$, project it isomorphically onto a cone $\bigtimes_v\mathsf{DT}(X_v)$ and apply the evaluation map. Since the evaluation has specifically been chosen to be combinatorially flat, the claimed combinatorial flatness follows. 
\end{proof}

We now deduce the first simplification of the degeneration formula. Recall from the context of Proposition~\ref{prop: expanded-formula}, we have a space $\bigtimes^\circ_v\mathsf{DT}(X_v)$ defined as an open in $\bigtimes_v\mathsf{DT}(X_v)$ by the cone subspace $\mathsf{ev}^{-1}(\Delta)$. 

\begin{lemma}\label{lem: diagonal-preimage}
    Let $\Delta^\ddagger$ be strict transform in $\bigtimes_v\mathsf{Hilb}(D_v)$ of the diagonal locus in $\prod_v \mathsf{Hilb}(D_v)$. The preimage of $\Delta^\ddagger$ along
    \[
{\bigtimes}_v\mathsf{DT}(X_v)\to {\bigtimes}_v\mathsf{Hilb}(D_v)
    \]
    is contained in the subspace $\bigtimes^\circ_v\mathsf{DT}(X_v)$.
\end{lemma}

\begin{proof}
The locus $\Delta^\ddagger$ is equal to the fine and saturated logarithmic pullback of the diagonal in $\prod_v \mathsf{Hilb}(D_v)$ under the subdivision $\bigtimes_v\mathsf{Hilb}(D_v)\to \prod_v \mathsf{Hilb}(D_v)$. We also know that the map ${\bigtimes}_v\mathsf{DT}_v(X_v)\to {\bigtimes}_v\mathsf{Hilb}(D_v)$ is combinatorially flat with reduced fibers. Putting these together, the usual scheme theoretic preimage in question is equal to the fine and saturated pullback associated to the diagram
\[
\begin{tikzcd}
    & {\bigtimes}_v\mathsf{DT}_v(X_v)\arrow{d}\\
\prod_e \mathsf{Hilb}(D_e)\arrow{r}&\prod_e \mathsf{Hilb}(D_e) \times \mathsf{Hilb}(D_e).
\end{tikzcd}
\]
By general arguments concerning logarithmic fiber products, this pullback is contained in the (possibly incomplete) subdivision of $\bigtimes_v\mathsf{DT}(X_v)$ induced by fiber product of the corresponding diagram of cone spaces. The latter fiber product is exactly the union of cones $\mathsf{ev}^{-1}(\Delta)$ discussed in the previous section. This defines the open subspace $\bigtimes^\circ_v\mathsf{DT}(X_v)$ so the claim follows. 
\end{proof}

Recall the morphism
\[
q\colon\mathsf{DT}_\gamma(Y_0)\to {\bigtimes}^\circ_v\mathsf{DT}_v(X_v)
\]
used in Proposition~\ref{prop: expanded-formula}. By composing with the open immersion we can consider the map
\[
p\colon \mathsf{DT}_\gamma(Y_0)\to {\bigtimes}_v\mathsf{DT}_v(X_v).
\]
We now state the first stage of our simplified gluing formula. 

\begin{corollary}
Let $\Delta^\ddagger$ be strict transform in $\bigtimes_v\mathsf{Hilb}(D_v)$ of the diagonal locus in $\prod_v \mathsf{Hilb}(D_v)$. The following equality of virtual classes holds:
\[
p_\star\left[\mathsf{DT}_\gamma(Y_0)\right]^{\mathsf{vir}} = \Delta^\ddagger \cap \left[\bigtimes_v \mathsf{DT}(X_v)\right]^{\mathsf{vir}},
\]
where we have suppressed the pullback of the cohomology class $\Delta^\ddagger$ from the notation.
\end{corollary}

\begin{proof}
    The statement of Proposition~\ref{prop: expanded-formula} is the statement:
\[
q_\star\left[\mathsf{DT}_\gamma(Y_0)\right]^{\mathsf{vir}} = \bigcap_e \mathsf H^\Delta(\mathcal D_e)\cap \left[{\bigtimes}^\circ_v \mathsf{DT}(X_v)\right]^{\mathsf{vir}},
    \]
    which is a refined intersection statement\footnote{Refined here is used in the sense of Fulton~\cite{Ful98}, i.e. the intersection product is supported on the scheme theoretic intersection rather than the non-compact ambient space.} supported on the space $\bigcap_e \mathsf H^\Delta(\mathcal D_e)\cap{\bigtimes}^\circ_v \mathsf{DT}(X_v)$. This locus is proper, and by Lemma~\ref{lem: diagonal-preimage}, the intersection class $\bigcap_e \mathsf H^\Delta(\mathcal D_e)$ is equal to the pullback of $\Delta^\ddagger$. The corollary follows by proper pushforward and the projection formula. 
\end{proof}

The gain we have made from the corollary is that we have collapsed the ``moving'' family of Hilbert schemes $\mathsf H^{\{2\}}(\mathcal D_e)$, to a blowup of an absolute product $\bigtimes_v\mathsf{Hilb}(D_v)$. 

Next, we use the fact that the virtual classes of the DT space are birationally invariant, to show that we can replace $\bigtimes_v\mathsf{DT}(X_v)$ with a simpler birational model. Let $\prod^\diamond_v\mathsf{DT}(X_v)$ be \textit{any} virtual birational model mapping to the space $\bigtimes_v\mathsf{Hilb}(D_v)$ constructed above.\footnote{By universal mapping properties, any such model is a subdivision of the fine and saturated fiber product of $\prod_v \mathsf{DT}_v(X_v)$ with $\bigtimes_v \mathsf{Hilb}(D_v)$ over $\prod_v \mathsf{Hilb}(D_v)$.}The cutting map again gives rise to
\[
p: \mathsf{DT}^\dagger_\gamma(Y_0)\to {\prod_v}^\diamond \mathsf{DT}_v(X_v).
\]

\begin{proposition}
The following formula holds:
\[
p_\star\left[ \mathsf{DT}^\dagger_\gamma(Y_0) \right]^{\mathsf{vir}} = \Delta^\ddagger\cap\left[ {\prod_v}^\diamond \mathsf{DT}_v(X_v) \right]^{\mathsf{vir}},
\]
where we have suppressed the pullback of the cohomology class $\Delta^\ddagger$ from the notation.
\end{proposition}

\begin{proof}
We have a commutative diagram
\[
\begin{tikzcd}
 \bigtimes_v\mathsf{DT}_v(X_v)\arrow{dr} \arrow{rr} && \bigtimes_v\mathsf{Hilb}(D_v)\\
 &\prod_v^\diamond \mathsf{DT}_v(X_v) \ar[ur]. & 
\end{tikzcd}
\]
The virtual classes are compatible under pushforward so this follows from the projection formula.
\end{proof}

We prove the main gluing result. With the work up to this point, we have found a \textit{particular} birational model of the evaluation map such that the pullback of the strict diagonal calculates the class of $\mathsf{DT}_\gamma$.

Recall we have produced a flattening of the evaluation morphism:
\[
\mathsf{ev}:\prod_v \mathsf{DT}_{\gamma(v)}(X_v) \to \prod_v \mathsf{Hilb}^\dagger_{\beta}(D_v).
\]
By further subdivision if required, we can assume that there is a subdivision
\[
\bigtimes_v\mathsf{Hilb}(D_v)\to \prod_v\mathsf{Hilb}^\dagger_\beta(D_v).
\]
We recall also that this map can be forced to have smooth source and target, and therefore is a local complete intersection morphism. 

\begin{lemma}\label{lem: good-square}
Consider the following diagram, defined by fine and saturated logarithmic base change:
\[
\begin{tikzcd}
\prod_v^\sim \mathsf{DT}(X_v)\arrow{d}{\mathsf{ev}^\sim}\arrow{r} & \prod_v \mathsf{DT}^\dagger_{\gamma(v)}(X_v)\arrow{d}{\mathsf{ev}^\dagger} \\
\bigtimes_v \mathsf{Hilb}(D_v) \arrow{r} &\prod_v \mathsf{Hilb}^\dagger_{\beta}(D_v).
\end{tikzcd}
\]
The underlying map of stacks is also Cartesian. 
\end{lemma}

\begin{proof}
Immediate from combinatorial flatness of the evaluation map. 
\end{proof}

We complete the proof of Theorem~\ref{thm: main-splitting-theorem} using an intersection theory argument.

\subsection*{Proof of Theorem~\ref{thm: main-splitting-theorem}} We inspect the fiber diagram:
\[
\begin{tikzcd}
\prod_v^\sim \mathsf{DT}(X_v)\arrow{d}{\mathsf{ev}^\sim}\arrow{r} & \prod_v \mathsf{DT}^\dagger_{\gamma(v)}(X_v)\arrow{d}{\mathsf{ev}^\dagger} \\
\bigtimes_v \mathsf{Hilb}(D_v) \arrow{r}{\alpha} &\prod_v \mathsf{Hilb}^\dagger_{\beta}(D_v).
\end{tikzcd}
\]
The map $\alpha$ is a local complete intersection morphism, so Gysin pullback is defined. 

\noindent
{\bf Key Claim.} There is an equality of cycle classes:
\[
\alpha^!\left[\prod_v \mathsf{DT}^\dagger_{\gamma(v)}(X_v) \right]^{\mathsf{vir}} = \left[{\prod_v}^\sim \mathsf{DT}(X_v) \right]^{\mathsf{vir}}.
\]
In order to see this, observe that since $\alpha$ is a subdivision, we can replace the bottom row of the square with the respective Artin fans:
\[
\begin{tikzcd}
\prod_v^\sim \mathsf{DT}(X_v)\arrow{d}{\mathsf{ev}^\sim}\arrow{r} & \prod_v \mathsf{DT}^\dagger_{\gamma(v)}(X_v)\arrow{d}{\mathsf{ev}^\dagger} \\
\bigtimes_v \mathsf{Exp}_0(D_v) \arrow{r}{a} &\prod_v\mathsf{Exp}^\dagger_0(D_v).
\end{tikzcd}
\]
The cone complexes of these spaces on the bottom row are the $\mathsf P(\Sigma_{D_e})$ spaces discussed in Section~\ref{sec: combinatorial-degeneration} parameterizing $0$-complexes; in fact, we can restrict to the locus with at most $m_e$ vertices, where $m_e$ the length of the putative subscheme along the edge, determined by the class $\beta_v$ in the Chow $1$-complex.

We observe that by construction, we have a factorization:
\[
\prod_v \mathsf{DT}^\dagger_{\gamma(v)}(X_v)\to \prod_v \mathsf{Exp}^\dagger(X_v)\to\prod_v\mathsf{Exp}^\dagger_0(D_v).
\]
By our choices of subdivision, the second arrow is flat. 
As a result, if we consider the fiber square 
\[
\begin{tikzcd}
\prod_v^\sim \mathsf{Exp}(X_v)\arrow{d}\arrow{r}{g} & \prod_v^\dagger \mathsf{Exp}(X_v)\arrow{d} \\
\bigtimes_v \mathsf{Exp}_0(D_v) \arrow{r}{a} &\prod_v\mathsf{Exp}^\dagger_0(D_v),
\end{tikzcd}
\]
the vertical arrows are flat, the two horizontal arrows have the same cotangent complex.  Furthermore, the subdivision:
\[
{\prod_v}^\sim \mathsf{DT}(X_v)\to \prod_v \mathsf{DT}^\dagger_{\gamma(v)}(X_v)
\]
is the tropical model associated to the top arrow $g$.  Therefore, in order to prove the claim, it now suffices to show:
\[
g^!\left[\prod_v \mathsf{DT}^\dagger_{\gamma(v)}(X_v) \right]^{\mathsf{vir}} = \left[{\prod_v}^\sim \mathsf{DT}(X_v) \right]^{\mathsf{vir}}.
\]
Consider the fiber square
\[
\begin{tikzcd}
\prod_v^\sim \mathsf{DT}(X_v)\arrow{d}\arrow{r} & \prod_v \mathsf{DT}^\dagger_{\gamma(v)}(X_v)\arrow{d}\\
\prod_v^\sim \mathsf{Exp}(X_v)\arrow{r}{g} & \prod_v^\dagger \mathsf{Exp}(X_v).
\end{tikzcd}
\]
By~\cite[Section~5]{MR20} the vertical arrows are equipped with the same obstruction theory. Since $g$ is also a local complete intersection, commutativity of virtual and Gysin pullback gives the claim. 

We return to the main statement. The arguments above produce the following fiber diagram:
\[
\begin{tikzcd}
\prod_v^\sim \mathsf{DT}(X_v)\arrow{d}{\mathsf{ev}^\sim}\arrow{r} & \prod_v \mathsf{DT}^\dagger_{\gamma(v)}(X_v)\arrow{d}{\mathsf{ev}^\dagger} \\
\bigtimes_v \mathsf{Hilb}(D_v) \arrow{r}{\alpha} &\prod_v \mathsf{Hilb}^\dagger_{\beta}(D_v).
\end{tikzcd}
\]
Let $\Delta^\ddagger$ and $\Delta^\dagger$ be the strict transform of the diagonal loci in the respective entries in the bottom row. Since $\alpha_\star[\Delta^\ddagger]$ is the class $[\Delta^\dagger]$ and the virtual class is pulled back under $\alpha$, the theorem follows compatibility between pushforward and Gysin pullback. 

\qed

\subsection{Review of the expanded GW degeneration formula} We quickly review the GW degeneration formula from~\cite{R19}. The basic setup is similar to the DT setting. We start with a degeneration $\mathcal Y\to B$. There is an associated moduli space of logarithmic stable maps:
\[
\mathsf{GW}_{\beta,g}(\mathcal Y/B)\to B.
\]
As explained in~\cite[Section~6.1]{R19}, the Abramovich--Chen--Gross--Siebert decomposition formula identifies rigid tropical maps from tropical curves, labelled by genus and degree, to $\Delta_{\mathcal Y}$. These are in bijection with the virtual irreducible components of the special fiber of $\mathsf{GW}_{\beta,g}(\mathcal Y/B)\to B$.

The rigid tropical curve is a map
\[
\gamma\to \Delta_{\mathcal Y},
\]
and we may and do assume that the vertices and edges of $\gamma$ map onto vertices and edges of $\Delta_{\mathcal Y}$. We will abuse notation slightly by confusing the graph $\gamma$ and the parameterized map. 

Each vertex $v$ determines a natural moduli space of logarithmic maps, with the local data specified at $v$. We denote it $\mathsf{GW}(X_v)$. The basic version of the gluing construction is based on the natural map:
\[
\mathsf{ev}\colon \prod_v \mathsf{GW}(X_v) \to \prod_e D_e\times D_e.
\]
The map is given by evaluation at the marked point corresponding to $e$ -- there is one copy for each side of the edge. So far, essentially any birational model of the GW space may be used.

The main construction in~\cite{R19} produces a modification
\[
\bigtimes_v\mathsf{GW}(X_v)\to \prod_v \mathsf{GW}(X_v).
\]
The subdivision carries universal expansions $\mathcal X_v$ for each vertex $v$. Moreover, if $e$ is an edge between $v$ and $v'$, the induced expansion of the divisor family $D_e$ is the same on $\mathcal X_v$ and $\mathcal X_{v'}$. 

We denote the universal divisor expansion $\mathcal D_e\to \bigtimes_v\mathsf{GW}(X_v)$ and by $\mathcal D_e^{\{2\}}$ its fiber square. We let $\mathcal D_e^{\Delta}\subset\mathcal D_e^{\{2\}}$ the fiberwise diagonal in the fibered square.

We can certainly find an appropriate model $\mathsf{GW}^\dagger_\gamma(Y_0)$ that maps to $\bigtimes_v\mathsf{GW}(X_v)$. The main theorem of~\cite[Section~6.4]{R19} is stated below. Let $\gamma$ be a fixed tropical stable map and $\overline\gamma$ its induced Chow $1$-complex. As in Remark~\ref{rem: decomposition-gw} for each rigid tropical stable map $\gamma$ with Chow $1$-complex $\overline \gamma$, let $\ell(\gamma)$ denote the product, taken over all edges $e$ in $\overline \gamma$, of the least common multiplies of the dilation factors $m_{e'}$ of edges $e'$ in $\gamma$ that map to $e$.

\begin{theorem}
    Given the morphism:
    \[
    q\colon\mathsf{GW}_\gamma(Y_0)\to \bigtimes_v\mathsf{GW}(X_v),
    \]
    there is an equality of Chow homology classes
    \[
    q_\star\left[\mathsf{GW}_\gamma(Y_0) \right]^{\mathsf{vir}} = \prod_e m_ \cdot \frac{1}{\ell(\gamma)}\cdot\bigcap_e \mathcal D_e^{\Delta}\cap \left[\bigtimes_v\mathsf{GW}(X_v)   \right]^{\mathsf{vir}},
    \]
    where $m_e$ is the tropical expansion factor along the edge $e$ in the tropical map $\gamma$. Equivalently, the integer $m_e$ is the contact order at the node corresponding to $e$.
\end{theorem}

\subsection{A simplified gluing formula in GW theory} We make a parallel set of manipulations of the degeneration formula in GW theory. We first state the theorem precisely.

Let $\mathcal Y\to B$ continue to be a degeneration and let $\gamma$ be a rigid tropical curve. We assume, as above, that the image of every vertex (resp. edge) of $\gamma$ is a vertex (resp. vertex or edge) of $\Delta_{\mathcal Y}$. Associated to each edge $e$ of $\gamma$ is a therefore a divisor $D_e$ in $Y_0$. We set:
\[
\mathsf{Ev}_\gamma(v) = \prod_{e\colon v\prec e} D_e,
\]
where the product runs over all edges and half edges adjacent to $v$. Each vertex of $\gamma$ determines a genus, degree, and contact order for a mapping space $\mathsf{GW}_{\gamma(v)}(X_v)$. In the previous section, we used an evaluation map from the moduli space $\bigtimes_v\mathsf{GW}(X_v)$ the fiber product over all edges of $\mathcal D_e^{\{2\}}$. There is a natural factorization of this evaluation morphism:
\[
\bigtimes_v\mathsf{GW}(X_v)\to {\prod_e}_{\bigtimes_v\mathsf{GW}(X_v)} \mathcal D_e^{\{2\}} \to \prod_v \mathsf{Ev}_\gamma(v).
\]
The second factor is a fibered product ranging over the edges $e$ of $\gamma$.

By semistable reduction, we can find a subdivision of the product of $\mathsf{Ev}_\gamma(v)$ over all vertices $v$, which we denote
\[
\bigtimes_v \mathsf{Ev}_\gamma(v) \to \prod_v\mathsf{Ev}_\gamma(v)
\]
such that the induced map
\[
\bigtimes_v\mathsf{GW}(X_v)\to\bigtimes_v \mathsf{Ev}^\dagger_\gamma(v)
\]
is combinatorially flat. Since the first map in the factorization above is strict, the map
\[
{\prod_e}_{\bigtimes_v\mathsf{GW}(X_v)} \mathcal D_e^{\{2\}} \to \bigtimes_v \mathsf{Ev}^\dagger_\gamma(v)
\]
is also combinatorially flat. 

Let $\Delta^\dagger$ denote the strict transform of the diagonal:
\[
\prod_e D_e\hookrightarrow \prod_e D_e\times D_e = \prod_v\mathsf{Ev}_\gamma(v),
\]
under the blowup
\[
\bigtimes_v \mathsf{Ev}_\gamma(v) \to \prod_v\mathsf{Ev}_\gamma(v)
\]

\begin{proposition}
    The preimage of $\Delta^\dagger$ under the map
    \[
    {\prod_e}_{\bigtimes_v\mathsf{GW}(X_v)} \mathcal D_e^{\{2\}} \to \bigtimes_v \mathsf{Ev}_\gamma(v).
    \]
    is the relative diagonal locus, i.e. the fibered product of $\mathcal D_e^\Delta$ over all edges $e$, with the fiber product taken over $\bigtimes_v\mathsf{GW}(X_v)$.
\end{proposition}

\begin{proof}
    Since the map is combinatorially flat, the preimage coincides with the fine and saturated logarithmic pullback filling in:
    \[
    \begin{tikzcd}
         & {\prod_e}_{\bigtimes_v\mathsf{GW}(X_v)} \mathcal D_e^{\{2\}}\arrow{d} \\
         \prod_e D_e\arrow{r} & \prod_e D_e\times D_e,
    \end{tikzcd}
    \]
    where the lower map is the factorwise diagonal. It therefore suffices to prove the statement by interpreting the diagram in the category of stacks over logarithmic schemes. The claim now follows by unwinding the functor of points of the fine and saturated base change.
\end{proof}

As a consequence of the proposition, we obtain the first refinement of the gluing formula. We continue with the notation $\mathcal Y$ for the degeneration, $Y_0$ for its special fiber, and $\gamma$ for the rigid tropical map. 

In what follows, we fix subdivisions such that there exists a combinatorially flat evaluation map:
\[
\mathsf{ev}^{\sim}:\bigtimes_v\mathsf{GW}(X_v)\to\bigtimes_v \mathsf{Ev}^\dagger_\gamma(v).
\]
As we have used before, we let $\ell$ denote the product, taken over all edges $e$ in $\overline \gamma$, of the least common multiplies of the dilation factors $m_{e'}$ of edges $e'$ in $\gamma$ that map to $e$.

\begin{corollary}\label{cor: collapsed-degeneration}
    Given the morphism:
    \[
    q: \mathsf{GW}_\gamma(Y_0)\to \bigtimes_v\mathsf{GW}(X_v),
    \]
    there is an equality of Chow homology classes
    \[
    q_\star \left[\mathsf{GW}^{}_\gamma(Y_0) \right]^{\mathsf{vir}} = \prod_e m_e\cdot \frac{1}{\ell(\gamma)}\cdot\left(\mathsf{ev}^\sim\right)^\star \Delta^\dagger \cap\left[\bigtimes_v\mathsf{GW}(X_v) \right]^{\mathsf{vir}}.
    \]
\end{corollary}

We present the final simplification of the formula. The result is parallel to Theorem~\ref{thm: main-splitting-theorem} on the DT side. Choose any combinatorial flattening:
\[
\mathsf{ev}^\dagger: \prod_v \mathsf{GW}^\dagger_{\gamma}(X_v)\to \prod_v\mathsf{Ev}_\gamma^\dagger(v).
\]
The number $\ell$ is as above. 

\begin{theorem}
    Given the morphism
    \[
\kappa: \mathsf{GW}^\dagger_\gamma(Y_0)\to \prod_v \mathsf{GW}^\dagger_{\gamma}(X_v),
\]
there is an equality of Chow homology classes
\[
\kappa_\star \left[\mathsf{GW}^\dagger_{\gamma}(Y_0) \right]^{\mathsf{vir}}= \prod_e m_e\cdot \frac{1}{\ell(\gamma)}\cdot\left(\mathsf{ev}^\dagger\right)^\star \Delta^\dagger \cap\left[\prod_v\mathsf{GW}^\dagger(X_v) \right]^{\mathsf{vir}}.
\]
\end{theorem}

We stress that the key difference between the corollary that precedes this result, and the theorem we have just stated, is that the second formula splits into a vertex-by-vertex expression. 

\begin{proof}
The proof is exactly parallel to the DT side. Since it is important for our main results, we sketch the steps. First, we have proved in the preceding corollary that there exists a combinatorially flat evaluation map, namely $\mathsf{ev}^\sim$, such that the pullback of the strict transform of the diagonal gives the required class associated to the rigid tropical curve $\gamma$. Second, we move this from a statement about a \textit{particular} combinatorial flattening of the evaluation map to an \textit{arbitrary} cone. We have fixed such an arbitrary flattening from the outset:
\[
\mathsf{ev}^\dagger: \prod_v \mathsf{GW}^\dagger_{\gamma}(X_v)\to \prod_v\mathsf{Ev}_\gamma^\dagger(v).
\]
We can also assume by performing further blowups, that there is a subdivision
\[
\bigtimes_v\mathsf{Ev}_\gamma^\dagger(v)\to\prod_v\mathsf{Ev}_\gamma^\dagger(v)
\]
We therefore end up with a commutative diagram:
\[
\begin{tikzcd}
    \bigtimes_v\mathsf{GW}_\gamma(X_v)\arrow{d}\arrow{r} & \prod_v\mathsf{GW}^\dagger_\gamma(X_v)\arrow{d}\\
    \bigtimes_v\mathsf{Ev}_\gamma^\dagger(v)\arrow{r}& \prod_v\mathsf{Ev}_\gamma^\dagger(v).
\end{tikzcd}
\]
Although this square is not a fiber square, we can \textit{replace} the top left part of the diagram with the fine and saturated fiber product: the result will still be a modification of the product since it is a pullback of a subdivision. Formal properties of the virtual class under pushforward guarantee that the conclusion of Corollary~\ref{cor: collapsed-degeneration} continues to hold with this replacement. Now, since the right vertical is combinatorially flat, this fine and saturated fibered diagram is also fibered in the category of stacks. We are now in exactly the situation of the proof of Theorem~\ref{thm: main-splitting-theorem} in the previous section. After making the appropriate cosmetic changes, replacing the Hilbert schemes of curves and points with the spaces of stable maps and their evaluation spaces, the argument follows from a formally identical diagram chase. 
\end{proof}

The discussion in this section concludes the simplified gluing formula that we have input into the key corollaries of Section~\ref{sec: splitting-compatibility}.

\bibliographystyle{siam} 
\bibliography{LogarithmicDT_Final}

\end{document}